\documentclass[psamsfonts]{amsart}
\usepackage[utf8]{inputenc}
\usepackage{amsfonts}
\usepackage{hyperref}
\hypersetup{
pdftitle={Cohomology of Hyperquot Schemes on Curves and Shifted Yangians},
pdfsubject={Mathematics, Algebraic Geometry, Representation Theory, Hyperquot Schemes, Yangians, Moduli of sheaves},
pdfauthor={Archi Kaushik}}
\usepackage{amsmath}
\usepackage{amssymb}
\usepackage{xcolor}
\usepackage{pdfpages}
\usepackage{amsthm}
\usepackage{pdflscape}
\usepackage{pgfplots}
\usepackage{mathrsfs}
\usepackage{pst-node}
\usepackage{tikz-cd}
\usepackage{rotating}

\usepackage{mathrsfs}
\usepackage{amsmath}
\newtheorem{thm}{Theorem}[section]
\newtheorem{cor}[thm]{Corollary}
\newtheorem{prop}[thm]{Proposition}
\newtheorem{lem}[thm]{Lemma}

\usepackage{enumitem}

\theoremstyle{definition}
\newtheorem{defn}[thm]{Definition}

\theoremstyle{remark}
\newtheorem{rem}[thm]{Remark}

\makeatletter
\let\c@equation\c@thm
\makeatother
\numberwithin{equation}{section}

\title [Cohomology of Hyperquot Schemes on Curves and Shifted Yangians]{The cohomology of Hyperquot Schemes on curves via shifted Yangians in type A}
 
\author{Archi Kaushik}

\begin{document}

\maketitle
\begin{abstract}
Let $V$ be a vector bundle of rank $r$ on a smooth projective complex curve $C$. The Hyperquot scheme $\text{F}^{n}\text{Quot}\,(V)$ is the moduli space of length $n$ flags of rank $r$ sub-sheaves of $V$. This article has two main results:
\begin{enumerate}

    \item We show that a certain shifted Yangian of $\mathfrak{sl}_{n+1}$ acts on $H^{*}\left(\text{F}^{n}\text{Quot}\,(V)\right)$ by correspondences.
    \item We define a family of $rn$ commuting Yangian operators which yields a natural basis for $H^{*}\left(\text{F}^{n}\text{Quot}\,(V)\right)$.
\end{enumerate}
  This generalises the work \cite{MaNe1} of Marian and Neguț, who proved the above results in the case $n=1$.
  The new feature, which makes this generalisation possible, is the use of so called skew-nested Quot schemes. The rank $1$ versions of these spaces, skew-nested Hilbert schemes, have been recently introduced by Sergej Monavari in the context of refined DT theory of local curves \cite{MoRefinedDT}. In the present article, skew-nested Quot schemes appear as correspondences associated with iterated commutators of Yangian elements. 
\end{abstract}
\tableofcontents

\section{Introduction}
\subsection{Moduli spaces of sheaves and representation theory.}
The purpose of this article is to obtain a precise description of the singular cohomology of Hyperquot schemes on curves. To do this, the perspective we take is to rigidify the vector space structure of the cohomology, by upgrading it to a representation of an algebra. This is a well-known principle in geometric representation theory which emerged in the independent works of Grojnowski \cite{Groj} and Nakajima \cite{Nak1} on the Hilbert scheme of points on a surface.
To provide some context to our results, let us briefly recall their classical works, in the formulation of \cite{Nak1}.\\\\
Given a smooth projective complex surface $S$, Nakajima arranges the various Hilbert schemes of points of $S$ together:
\[
\text{Hilb}(S):=\bigsqcup_{n\geq 0}\text{Hilb}_{n}(S)
\]
and defines $\mathbb{Q}-$linear operators:
\begin{equation}\label{Nakajima Operators}
a_{k}(\gamma):H^{*}(\text{Hilb}(S))\rightarrow H^{*}(\text{Hilb}(S))
\end{equation}
for all $k\in \mathbb{Z}$ and $\gamma\in H^{*}(S)$, using natural geometric correspondences. The main results of \cite{Nak1} are that the operators $a_{k}(\gamma)$ satisfy the defining relations of the infinite dimensional Heisenberg Lie algebra $\widehat{\mathfrak{gl}}_{1}$ 
and that the representation thus obtained is identified with the well-known Fock space.\\\\
In particular, if $k,l>0$ and $\{\gamma_{1},...,\gamma_{t}\}$ is a basis for $H^{*}(S)$, then the operators $a_k(\gamma_{i})$ and $a_{l}(\gamma_{j})$ commute and the set
\begin{equation}\label{Nakajima basis}
\left\{a_{k_l}(\gamma_{i_l}) ... a_{k_2}(\gamma_{i_2}) a_{k_1}(\gamma_{i_1})|0\rangle\right\}_{l\geq0}
\newline
\end{equation}
forms a basis of $H^{*}(\text{Hilb}(S))$, where $|0\rangle$ is a generator of the one dimensional vector space $H^{*}(\text{Hilb}_{0}(S)).$\\\\
The canonical Nakajima basis \eqref{Nakajima basis} makes it clear, that by using the operators $a_{k}(\gamma)$, the cohomology of $\text{Hilb}_{n}(S)$ can generated by the cohomologies of $\text{Hilb}_{m}(S)$ with $m<n$, inductively using the $\widehat{\mathfrak{gl}}_{1}$ action.\\\\
Therefore, not only does the Heisenberg algebra action categorify the Poincaré polynomial (generating series of the Betti numbers) of $\text{Hilb}(S)$, it yields a powerful tool to probe the cohomology of $\text{Hilb}(S)$. This was subsequently exploited to understand the intersection theory \cite{Lehn1}, cup product \cite{LehnSorger1}, quantum cohomology \cite{OkPa}, Gromov-Witten theory \cite{OberdieckGWK3}, Chow ring \cite{MauNeg} etc. of $\text{Hilb}(S).$
\subsection{Quot schemes of points on curves and shifted Yangians of $\mathfrak{sl}_{2}$}
Let $C$ be a fixed smooth projective curve over $\mathbb{C}$ and let $V$ be a fixed rank $r$ vector bundle on $C$. For $d\geq 0$, the Grothendieck Quot scheme $\text{Quot}_d(V)$ is the fine moduli space which parameterises rank $r$ sub-bundles 
\[
{E}\subseteq V
\] 
such that the torsion sheaf $V/{E}$ has length $d$. Let $\rho$ and $\pi$ be the projection maps from $\text{Quot}_{d}(V)\times C$ to the first and second factor respectively:
\[\begin{tikzcd}[cramped]
	& {\text{Quot}_d(V)\times C} \\
	{\text{Quot}_d(V)} && C.
	\arrow["\rho"', from=1-2, to=2-1]
	\arrow["\pi", from=1-2, to=2-3]
\end{tikzcd}\]
The space $\text{Quot}_{d}(V)\times C$ carries a universal inclusion
\[
\mathcal{E}\subseteq \pi^{*}(V)
\] 
and the space $\text{Quot}_{d}(V)$ is a smooth projective variety of dimension $rd$. Quot schemes on curves generalise the notion of symmetric powers of $C$ and are closely related to the moduli spaces of vector bundles on $C$ (see e.g. \cite{BehrendDhillon}, \cite{BifetGhioneLetizia} and \cite{HoskinsSimon2}).\\\\
On the other hand, Yangians are associative algebras; for a Lie algebra $\mathfrak{g}$, they arise as deformations of the universal enveloping algebra of the Lie algebra $\mathfrak{g}[t]$. First defined in \cite{Drin} for simple Lie algebras, different versions of Yangians are now ubiquitous in geometric representation theory (see e.g. \cite{Var} and \cite{MauOko}). In this article, we will make use of shifted Yangians that were defined in \cite{BrKl} and are generalisations of usual Yangians (see Definition \ref{Shifted Yangian defn}). \\\\
In the recent work \cite{MaNe1}, Marian and Neguț take the Grojnowski-Nakajima point of view to study the cohomology of Quot schemes by realising the cohomology of
\[
\text{Quot}(V):=\bigsqcup_{d\geq0}\text{Quot}_{d}(V),
\]
as a representation of $Y_{\hbar}^{r}({\mathfrak{sl}_{2}})$, a certain shifted (truncated) Yangian of $\mathfrak{sl}_{2}$.\\\\
A key role in the construction of their action is played by the nested Quot scheme $\text{Quot}_{d,d+1}$, with the following description of its closed points:
\[
\left\{E_{2}\subseteq E_{1}\subseteq V \mid \text{rank }E_{i}=r, \text{ length} \left(\frac{V}{E_{1}}\right)=d,\,\text{length} \left(\frac{V}{E_{2}}\right)=d+1 \right\}
\]
and which fits into the following diagram:

\begin{equation}\label{Marian Negut Correspondence}
\begin{tikzcd}[cramped]
	& {\text{Quot}_{d,d+1}(V)} \\
	{\text{Quot}_{d}(V)} & C & {\text{Quot}_{d+1}(V)}
	\arrow["{\pi_{-}}"', from=1-2, to=2-1]
	\arrow["{\pi_{C}}"', from=1-2, to=2-2]
	\arrow["{\pi_{+}}", from=1-2, to=2-3].
\end{tikzcd}
\end{equation}
\newline
Here, $\pi_{\pm}$ are the natural forgetful maps and $\pi_{C}$ remembers the support point of the length one sheaf $E_{1}/E_{2}.$ In fact, the morphism $\pi_{-}\times\pi_{C}$ realises $\text{Quot}_{d,d+1}(V)$ as the total space of the projective bundle
\[
\mathbb{P}(\mathcal{E})\rightarrow\text{Quot}_{d}\times C.
\]
As a consequence, $\text{Quot}_{d,d+1}(V)$ carries a tautological exact sequence:
\begin{equation}\label{exact sequence on Quot d,d+1}
0\rightarrow\mathcal{G}\rightarrow(\pi_{-}\times\pi_{C})^{*}(\mathcal{E})\rightarrow\mathscr{L}\rightarrow0,
\end{equation}
\noindent where $\mathscr{L}$ is the tautological line bundle $\mathcal{O}(1)$ and $\mathcal{G}$ is a rank $r-1$ vector bundle on on $\text{Quot}_{d,d+1}(V)$. Let us denote the first Chern class of $\mathscr{L}$ by $\lambda\in H^{2}(\text{Quot}_{d,d+1}(V)).$\\\\
On the algebraic side, the shifted Yangian $Y_{\hbar}^{r}({\mathfrak{sl}_{2}})$ has a presentation in terms of generators and relations:
\[
Y_{\hbar}^{r}({\mathfrak{sl}_{2}}):= \frac{\mathbb{Q}[\hbar]\left\langle e^{(v)},f^{(v)},m^{(t)} \right\rangle_{v\geq0,\,t\in\{1,...,r\}   }}{\text{Several relations, given in Definition 1 of \cite{MaNe1}}},
\]
where $\hbar$ is a formal parameter. Marian and Neguț define linear operators
\[
H^{*}(\text{Quot}(V))\rightarrow H^{*}(\text{Quot}(V)\times C),
\]
corresponding to the generators of $Y_{\hbar}^{r}({\mathfrak{sl}_{2}})$; For $t\in\{1,...,r\}$, there are multiplication operators:
\[
m^{(t)}:=c_{t}(\mathcal{E})\cdot\rho^{*}(-):H^{*}(\text{Quot}_d(V))\rightarrow H^{*}(\text{Quot}_d(V)\times C).
\]
Also, for $v\geq 0$, there are raising operators:
\[
e^{(v)}:=(\pi_{+}\times\pi_{C})_{*}(\lambda^{v}\cdot\pi_{-}^{*}(-)):H^{*}(\text{Quot}_d(V))\rightarrow H^{*}(\text{Quot}_{d+1}(V)\times C)
\]
and lowering operators:
\[
f^{(v)}:=(\pi_{-}\times\pi_{C})_{*}(\lambda^{v}\cdot\pi_{+}^{*}(-)):H^{*}(\text{Quot}_{d+1}(V))\rightarrow H^{*}(\text{Quot}_{d}(V)\times C),
\]
which are defined using powers of $\lambda\in H^{2}(\text{Quot}_{d,d+1}(V))$.\\\\
The first main result of \cite{MaNe1} is that the operators $m^{(t)}, e^{(v)}$ and $f^{(v) }$ satisfy the defining relations of $Y_{\hbar}^{r}({\mathfrak{sl}_{2}})$. That is, they define a $\mathbb{Q}-$linear assignment:
\begin{equation}\label{quot scheme action homomorphism}
    \varphi: Y_{\hbar}^{r}({\mathfrak{sl}_{2}})\rightarrow \text{Hom}(H^{*}(\text{Quot}(V)),H^{*}(\text{Quot}(V)\times C))
\end{equation}
which satisfies several natural properties (see Definition \ref{Action defn} for more details).
\begin{rem}
Using the Künneth decomposition $$H^{*}(\text{Quot}(V)\times C)=H^{*}(\text{Quot}(V))\otimes H^{*}(C),$$ the map $\varphi$ should be thought of as a family of actions of $Y_{\hbar}^{r}({\mathfrak{sl}_{2}})$ on $\text{Quot}(V)$, with $H^{*}(C)$ as a parameter space. Indeed, for every element $q\in Y_{\hbar}^{r}({\mathfrak{sl}_{2}}),$ we obtain operators 
\begin{equation}\label{how to get endomorphisms from Andrei action.}
   q(\gamma):=\rho_{*}(\pi^{*}(\gamma)\cdot q(-)):H^{*}(\text{Quot}(V))\rightarrow H^{*}(\text{Quot}(V)), 
\end{equation}
one for every $\gamma\in H^{*}(C)$. In fact, the Grojnowski-Nakajima action can also be restated in the language of \eqref{quot scheme action homomorphism}, that is, as a $\mathbb{Q}$-linear map:
\[
\psi:\widehat{\frak{gl}}_{1}\rightarrow \text{Hom}(H^{*}(\text{Hilb}(S)),\,H^{*}(\text{Hilb}(S)\times S))
\]
and the operators \eqref{Nakajima Operators} can be obtained by the procedure \eqref{how to get endomorphisms from Andrei action.}. In some contexts, for example in \cite{MauNeg} while working with Chow groups, the homomorphism $\psi$ packages strictly more information than the operators \eqref{Nakajima Operators}.\\
\end{rem}
\noindent Then, in order to probe the cohomology of $\text{Quot}(V)$, Marian and Neguț proceed to define a family of $r$ operators, using the rank $r-1$ vector bundle $\mathcal{G}$ on $\text{Quot}_{d,d+1}$ (cf. the exact sequence \eqref{exact sequence on Quot d,d+1});
\begin{equation}\label{Marian Negut commuting family of operators}
    a^{(w)}:= (\pi_{+}\times\pi_{C})_{*}(c_{w}(\mathcal{G})\cdot\pi_{-}^{*}(-)):H^{*}(\text{Quot}_{d}(V))\rightarrow H^{*}(\text{Quot}_{d+1}(V)\times C),
\end{equation}
for $w\in \{0,...,r-1\}$. The above operators can also be realised algebraically, that is, they belong to the image of the map $\varphi$ in equation \ref{quot scheme action homomorphism}, and correspond to some linear combinations of products of the form $e^{(i)} m^{(j)}$ in $Y_{\hbar}^{r}({\mathfrak{sl}_{2}})$ (see Remark 1 in \cite{MaNe1}).\\\\
It is then proved in \cite{MaNe1} that the operators $a^{(w)}$ commute with each other and this family of commuting operators is then used to construct a natural basis for $H^{*}(\text{Quot}(V))$. More precisely, if $\gamma_{1},...,\gamma_{2g+2}$ is a basis of $H^{*}(C)$, then 
\begin{equation}\label{Marian Negut Basis}
\left\{a^{(k_l)}(\gamma_{i_l}) ... a^{(k_2)}(\gamma_{i_2}) a^{(k_1)}(\gamma_{i_1})|0\rangle\right\}_{l\geq0}
\end{equation}
\noindent is a basis of $H^{*}(\text{Quot}(V))$. As before, the vacuum vector $|0\rangle$, is a generator of the one dimensional vector space $H^{*}(\text{Quot}_{0}(V))$. This basis mirrors the natural basis \eqref{Nakajima basis} for $\text{Hilb}(S)$, but in this case, one has to dig a bit deeper into the algebra that acts, to find a commuting family of operators which yields a basis.\\\\
In a subsequent work \cite{MaNe2}, Marian and Neguț, were able to categorify the situation one step further and obtained a semi-orthogonal decomposition for the derived category of $\text{Quot}_{d}(V)$ (see also \cite{TodaSOD}), which enabled them to prove a conjecture (see \cite{Krug} and \cite{OpreaSinha}) on the cohomology of certain tautological vector bundles on $\text{Quot}_{d}(V)$. 
\subsection{Hyperquot schemes of points on curves and shifted Yangians of $\mathfrak{sl}_{n+1}$ } For the remainder of the article, we fix a smooth projective complex curve $C$ and a vector bundle $V$ on $C$, of rank $r$.\\\\
Given a positive integer $n$ and an $n-$tuple of non-negative integers $\vec{d}=(d_1,...,d_n)$, we define the Hyperquot scheme $F^n\text{Quot}_{\vec{d}}\,(V)$, to be the moduli space parameterising flags of sub-sheaves of $V$, with the following description of its closed points:
\[
\left\{E_n\subseteq ...\subseteq E_1\subseteq V\mid \text{rank }E_i=r,\,\text{length}\left(\frac{V}{E_i}\right)=d_i \text{ for all }i\in\{1,...,n\}\right\}.
\]
Clearly the space $F^n\text{Quot}_{\vec{d}}\,(V)$ is non-empty if and only if the sequence $(d_1,...,d_n)$ is non-decreasing. In this case $F^n\text{Quot}_{\vec{d}}\,(V)$ is a smooth projective variety of dimension $rd_n$.\\\\
Let $\rho$ and $\pi$ be the projection maps from $F^n\text{Quot}_{\vec{d}}\,(V)\times C$ to the first and second factor respectively. The product $F^n\text{Quot}_{\vec{d}}\,(V)\times C$ carries a universal flag of locally-free sheaves:
\begin{equation}\label{Universal flag}
    \mathcal{E}_{n}\subseteq \mathcal{E}_{n-1}\subseteq...\mathcal{E}_1\subseteq \mathcal{E}_{0}=\pi^{*}(V).
\end{equation}
On the other hand, just as in the $\mathfrak{sl}_{2}$ case, the shifted Yangian of $\mathfrak{sl}_{n+1}$ has a generators and relations presentation:

\[
Y_{\hbar}^{r}({\mathfrak{sl}_{n+1}}):= \frac{\mathbb{Q}[\hbar]\left\langle e_{k}^{(v)},f_{k}^{(v)},m_k^{(t)}\right\rangle_{k\in\{1,...,n\},\,v\geq 0,\,t\in\{1,...,r\}}}{\text{Relations in Definition }\ref{Shifted Yangian defn}.}.
\]
\newline
Following the approach of \cite{MaNe1}, in order to understand the geometric representation theory of Hyperquot schemes, we consider:
\[
F^n\text{Quot}(V):=\bigsqcup\limits_{\vec{d}\in \mathbb{N}^{n}}F^n\text{Quot}_{\vec{d}}\,(V)
\]
and define operators:
\begin{equation}\label{e,f,m}
    e_{k}^{(v)},f_{k}^{(v)},m_k^{(t)}:H^{*}(F^n\text{Quot}(V))\rightarrow H^{*}(F^n\text{Quot}(V)\times C)
\end{equation}
for all $k\in\{1,...,n\}$, $v\geq 0$ and $t\in\{1,...,r\}$. First, we define the multiplication operators:
\[
m_{k}^{(t)}:=c_{t}(\mathcal{E}_{k})\cdot\rho^{*}(-): H^{*}(F^{n}\text{Quot}_{\vec{d}}\,(V))\rightarrow H^{*}(F^{n}\text{Quot}_{\vec{d}}\,(V)\times C),
\]
where $c_{u}(\mathcal{E}_{k})$ is the $u^{\text{th}}$ Chern class of the universal bundle $\mathcal{E}_{k}$. Then,
consider a new $n-$tuple $\vec{d}'_{k}:=(d_1,...,d_{k-1},d_{k}+1,d_{k+1},...,d_n)$ and the $(n+1)-$ tuple \\
$\vec{d}''_{k}:=(d_1,...,d_{k-1},d_{k},d_{k}+1,d_{k+1},...,d_n)$. We have the following diagram, analogous to Diagram \ref{Marian Negut Correspondence} for Hyperquot schemes:

\begin{equation}\label{correspondence diagram defining e,f}
    \begin{tikzcd}
    &F^{n+1}\text{Quot}_{{\vec{d}''}_{k}}\,(V) \arrow[swap]{dl}{\pi_{-}}\arrow{d}{\pi_{C}}\arrow{dr}{\pi_{+}}&\\
    F^n\text{Quot}_{\vec{d}}\,(V) &  C   &F^{n}\text{Quot}_{{\vec{d}'}_{k}}\,(V)\\
\end{tikzcd}.
\end{equation}
Given a closed point
\begin{equation}\label{closed poit of Hyperquot correspondence}
({F}_{n+1}\subseteq ...\subseteq{F}_{k+1}\subseteq{F}_{k}\subseteq...\subseteq{F}_{1}\subseteq V)\in F^{n+1}\text{Quot}_{{\vec{d}''}_{k}}\,(V),
\end{equation}
the morphisms $\pi_{-}$ and $\pi_{+}$ in Diagram \ref{correspondence diagram defining e,f} respectively forget the locally free sheaves $F_{k+1}$ and $F_{k}$, whereas the map $\pi_C$ remembers the support point of the length one sheaf $F_{k}/F_{k+1}$. In fact, the variety $F^{n+1}\text{Quot}_{{\vec{d}''}_{k}}\,(V)$ carries a tautological line bundle 
$\mathscr{L}_k$
whose fibre over a point \eqref{closed poit of Hyperquot correspondence}
is canonically identified with $F_{k}/F_{k+1}$. Subsequently, let us define
\[
\lambda_{k}:=c_{1}(\mathscr{L}_{k})\in H^{2}(F^{n+1}\text{Quot}_{{\vec{d}''}_{k}}\,(V)).
\]
Using powers of $\lambda_{k}$, for $v\geq 0$, we define the creation/annihilation operators;
\begin{equation}\label{definition of e operator}
e^{(v)}_{k}:=(\pi_{+}\times \pi_{C})_{*}(\lambda_{k}^{v}\cdot\pi_{-}^{*}(-)):H^{*}(F^{n}\text{Quot}_{\vec{d}}\,(V))\rightarrow H^{*}(F^{n}\text{Quot}_{\vec{d}'_{k}}\,(V)\times C).
\end{equation}
\begin{equation}\label{definition of f operator}
f^{(v)}_{k}:=(\pi_{-}\times \pi_{C})_{*}(\lambda_{k}^{v}\cdot\pi_{+}^{*}(-)):H^{*}(F^{n}\text{Quot}_{\vec{d}'_{k}}\,(V))\rightarrow H^{*}(F^{n}\text{Quot}_{\vec{d}}\,(V)\times C).
\end{equation}
\begin{rem}
 In some contexts, it will be convenient to consider the operators $e_k^{(v)},f_{k}^{(v)}, m^{(t)}_{k}$ as $v$ and $t$ vary. So we define the following power series for all $k\in \{1,...,n\}$:
\begin{equation}\label{currents}
e_{k}(z):=\sum\limits_{v=0}^{\infty}\frac{e_{k}^{(v)}}{z^{v+1}},\,\,\,\,\,\,f_{k}(z):=\sum\limits_{v=0}^{\infty}\frac{f_{k}^{(v)}}{z^{v+1}},\,\,\,\,\,\,m_{k}(z):=\sum_{t=0}^{r}(-1)^{t}m_{k}^{(t)}z^{r-t}.
\end{equation} 
Above, we make the convention that $m_{k}^{(0)}:=\rho^{*}$ for all $k\in\{1,...,n\}$. Also, given a vector bundle $E$ on $F^n\text{Quot}(V)\times C$, such as the universal bundles (\ref{Universal flag}), we denote the total Chern class of $E$ by
\begin{equation}\label{total chern class definition}
c({E},z):=\sum\limits_{s=0}^{r}z^{r-s}(-1)^{s}c_s(E).
\end{equation}
\end{rem}
\noindent Our first main result is that the operators $e_{k}^{(v)},f_{k}^{(v)}$ and $m_{k}^{(t)}$ satisfy the defining relations of $Y_{\hbar}^{r}({\mathfrak{sl}_{n+1}})$ (see Definition \ref{Shifted Yangian defn}). Let us state this in more precise terms:\\\\
Let $K_C\in H^{2}(C)$ be the canonical class of the curve $C$ and $\delta\in H^{2}(C\times C)$ be the diagonal class. We will also denote the pull-back of $\delta$ to $H^{*}(F^{n}\text{Quot}(V)\times C\times C))$ by the same notation and will also use the standard notation $\delta_{i,j}$ for the Kronecker delta but this will not cause any confusion with the diagonal class. Section \ref{section: The Yangian action} will be devoted to the proof of the following result:
\begin{thm}\emph{}\label{Yangian relations satisfied}\\
    \begin{enumerate}
\item  For all $i,j\in\{1,...,n\}$ and $s,t\geq 0$, we have the following equalities of $\mathbb{Q}-$ linear homomorphisms 
\[
H^{*}\left(F^{n}\text{Quot}(V))\rightarrow H^{*}(F^{n}\text{Quot}(V)\times C\times C\right):
\]

\begin{equation}\label{mm commutation action}
    \left[m_{i}^{(s)}, m_{j}^{(t)}\right]=0
\end{equation}

\begin{equation}\label{ee commutation 1 action}
    {\left[e_{i}^{(s+1)}, e_{i}^{(t)}\right]}-{\left[e_{i}^{(s)}, e_{i}^{(t+1)}\right]}=-\delta \cdot\left(e_{i}^{(s)}e_{i}^{(t)}+e_{i}^{(t)}e_{i}^{(s)}\right).
\end{equation}

\begin{equation}\label{ee commutation 2 action}
    {\left[e_{i}^{(s)}, e_{i-1}^{(t+1)}\right]}{}-{\left[e_{i}^{(s+1)}, e_{i-1}^{(t)}\right]}{}= -\delta \cdot e_{i}^{(s)} e_{i-1}^{(t)}.
\end{equation}

\begin{equation}\label{e commutation 3 action}
    \left[e_{i}^{(s)}, e_{j}^{(t)}\right]=0  \,\,\,\text{ for }|i-j|>1.
\end{equation}

\begin{equation} \label{em relations action}
    {\left[m_{i}^{(s)}, e_{j}^{(t)}\right]}=\delta_{i,j}\cdot\delta\cdot\left(\sum_{l=0}^{s-1}e_{j}^{(t+l)}m_{i}^{(s-l-1)}(-1)^{l+1}\right).
\end{equation}

\begin{equation}\label{ff commutation 1 action}
    {\left[ f_{i}^{(t)},f_{i}^{(s+1)}\right]}-{\left[f_{i}^{(t+1)},f_{i}^{(s)}\right]}=-\delta\cdot \left(f_{i}^{(s)}f_{i}^{(t)}+f_{i}^{(t)}f_{i}^{(s)}\right).
\end{equation}

\begin{equation}\label{ff commutation 2 action}
    {\left[f_{i}^{(s+1)}, f_{i-1}^{(t)}\right]}-{\left[f_{i}^{(s)}, f_{i-1}^{(t+1)}\right]}= -\delta \cdot f_{i}^{(t)} f_{i-1}^{(s)}
\end{equation}

\begin{equation}\label{f commutation 3 action}
    \left[f_{i}^{(s)}, f_{j}^{(t)}\right]=0  \,\,\,\text{ for }|i-j|>1.
\end{equation}

\begin{equation}\label{fm relations action}
    {\left[ f_{j}^{(t)},m_{i}^{(s)}\right]}=\delta_{i,j}\cdot\delta\cdot\left(\sum_{l=0}^{s-1}m_{i}^{(s-l-1)}f_{j}^{(t+l)}(-1)^{l+1}\right).
\end{equation}

\begin{equation}\label{ef relations action}
    {\left[e_{i}^{(s)}, f_{j}^{(t)}\right]}=-\delta_{i,j}\cdot\delta\cdot h_{_{i}}^{(s+t+1-r\cdot \delta_{i,n})},
    \end{equation}
where 
\[
h_i(z):=\sum\limits_{l=0}^{\infty}\frac{h_{i}^{(l)}}{z^{l+\delta_{i,n}\cdot r}}:=\frac{c(\mathcal{E}_{i+1},z)c(\mathcal{E}_{i-1},z+K_C)}{c(\mathcal{E}_{i},z)c(\mathcal{E}_{i},z+K_C)}.
\]
As before, the locally-free sheaves $\mathcal{E}_{i}$ are universal sheaves on $F^{n}\text{Quot}(V)$ (cf. \eqref{Universal flag}) and we define $c(\mathcal{E}_{n+1},z)=1.$ \\\\
In the above relations \eqref{mm commutation action}-\eqref{ef relations action}, we make the convention that the operators with superscripts $(s)$ and $(t)$ contribute to the first and second factor of $C\times C$ respectively.\\\\
\item\label{Serre relations action} For all $i,j\in\{1,...,n\}$ such that $|i-j|=1$ and for all $ s,t,u\geq 0$, we have the following identities of linear maps 
\[
H^{*}\left(F^{n}\text{Quot}(V)\right)\rightarrow H^{*}\left(F^{n}\text{Quot}(V)\times C\times C\times C\right):
\]
        
\begin{equation}\label{e Serre relations action}
    \left[e_{i}^{(s)},\left[e_{i}^{(t)}, e_{j}^{(u)}\right]\right]+\left[e_{i}^{(t)},\left[e_{i}^{(s)}, e_{j}^{(u)}\right]\right]=0.
\end{equation}

and

\begin{equation}\label{f Serre relations action}
    \left[f_{i}^{(s)},\left[f_{i}^{(t)}, f_{j}^{(u)}\right]\right]+\left[f_{i}^{(t)},\left[f_{i}^{(s)}, f_{j}^{(u)}\right]\right]=0.
\end{equation}
As before, we make the convention that the operators with superscripts $(u),\,(s)$ and $(t)$ contribute to the first, second, and third factor of $C\times C\times C$ respectively.

    \end{enumerate}

\end{thm}

\begin{cor}
    In view of the definition of the shifted Yangian $Y_{\hbar}^{r}({\mathfrak{sl}_{n+1}})$ in Definition \ref{Shifted Yangian defn}, Theorem \ref{Yangian relations satisfied} implies that the operators \ref{e,f,m} yield a $\mathbb{Q}-$ linear map
    \[
    \alpha:Y_{\hbar}^{r}({\mathfrak{sl}_{n+1}})\rightarrow Hom\left(H^{*}\left(F^n\text{Quot}(V)\right),H^{*}\left(F^n\text{Quot}(V)\times C\right)\right)
    \]
that defines an action in the sense of Definition \ref{Action defn}.
\end{cor}

\subsection{The commuting family of operators $a_{k}^{(v)}$ and a natural basis} In analogy to the natural bases \eqref{Nakajima basis} and \eqref{Marian Negut Basis} of $H^{*}(\text{Hilb}(S))$ and $H^{*}(\text{Quot}(V))$ respectively, we would like to construct a family of commuting Yangian operators on $H^{*}(F^{n}\text{Quot}\,(V))$, such that applying arbitrary products of these operators to the vacuum vector $$|0\rangle\in H^{*}( F^{n}\text{Quot}_{(0,...,0)}(V)),$$ yields a basis for the cohomology of $H^{*}(F^{n}\text{Quot}\,(V))$. Furthermore, we would like that such a class of operators specialises to the operators \eqref{Marian Negut commuting family of operators} of  Marian and Neguț in the case $n=1$.\\\\
Let us now explain how we go about defining such operators. Let $\vec{d}=(d_1,...,d_n)$ be an $n-$tuple of non-negative integers. Given $k\in\{1,...,n\}$, we define the $n-$tuple
\[
\vec{d}_{k,n} := (d_1,...,d_{k-1},d_k+1,d_{k+1}+1,...d_{n}+1).
\]
Consider the moduli space $\mathcal{Z}^{k,n}_{\vec{d}}$, a sub-scheme  of $F^{n}\text{Quot}_{\vec{d}}\,(V)\times F^{n}\text{Quot}_{\vec{d}_{k,n}}(V)$, which parameterises diagrams of inclusions of rank $r$ sub-sheaves of $V$ on $C$ of the following form:
\begin{equation}\label{double nested quot}
\begin{tikzcd}
	{E_{n}} & {...\,\,\,E_{k+2}} & {E_{k+1}} & {E_k} & {...\,\,\,E_2} & {E_1} \\
	{F_{n}} & {...\,\,\,F_{k+2}} & {F_{k+1}} & {F_k}
	\arrow[hook, from=1-1, to=1-2]
	\arrow[hook, from=1-2, to=1-3]
	\arrow[hook, from=1-3, to=1-4]
	\arrow[hook, from=1-4, to=1-5]
	\arrow[hook, from=1-5, to=1-6]
	\arrow["p", hook, from=2-1, to=1-1]
	\arrow[ hook, from=2-1, to=2-2]
	\arrow["p", hook, from=2-2, to=1-2]
	\arrow[hook, from=2-2, to=2-3]
	\arrow["p", hook, from=2-3, to=1-3]
	\arrow[hook, from=2-3, to=2-4]
	\arrow["p", hook, from=2-4, to=1-4].
\end{tikzcd}
\end{equation}
\newline
That is, the space $\mathcal{Z}^{k,n}_{\vec{d}}$ parameterises pairs of length $n$ flags
\[
(E_n\subseteq...\subseteq E_1\subseteq V)\in F^{n}\text{Quot}_{\vec{d}}\,(V)  \text{ and } (F_n\subseteq...\subseteq F_1\subseteq V)\in F^{n}\text{Quot}_{\vec{d}_{k,n}}\,(V),
\]
such that $E_i=F_i$ for $i=1,...,k-1$ and $E_j/F_j$ is a length one sheaf supported at the same, but arbitrary, point $p\in C$ for all $j=k,...,n$.\\\\
We will show that $\mathcal{Z}^{k,n}_{\vec{d}}$ is a smooth projective variety of dimension $r(d_{n}+1)$ (Proposition \ref{proposition: smoothness and dimension of moduli space Z}) and construct a rank $r-1$ vector bundle $\mathcal{G}^{k,n}_{\vec{d}}$ on $\mathcal{Z}^{k,n}_{\vec{d}}$ (Subsection \ref{subsection:Some sheaves on Z}).
Note that we have the following diagram: 
\begin{equation}\label{Diagram for Moduli space: Z}
    \begin{tikzcd}
    &\mathcal{Z}^{k,n}_{\vec{d}} \arrow[swap]{dl}{q_{-}}\arrow{d}{q_{C}}\arrow{dr}{q_{+}}&\\
    F^n\text{Quot}_{\vec{d}}\,(V) &  C   &F^n\text{Quot}_{\vec{d}_{k,n}}(V).\\
\end{tikzcd}.
\end{equation}
The maps $q_{-}$ and $q_+$ are the obvious forgetful maps and the map $q_C$ sends a point of $\mathcal{Z}^{k,n}_{\vec{d}}$, which is of the form \eqref{double nested quot}, to $p\in C$. Then we use the Chern classes of $\mathcal{G}^{k,n}_{\vec{d}}$ to define:

\begin{equation}\label{Operators a's definition}
    a_{k}^{(v)}:= (q_{+}\times q_{C})_{*}(c_{v}(\mathcal{G}^{k,n}_{\vec{d}})\cdot q_{-}^{*}(-)):H^{*}(F^n\text{Quot}_{\vec{d}}\,(V))\rightarrow H^{*}(F^n\text{Quot}_{\vec{d}_{k,n}}(V)\times C),
    \newline
\end{equation}
for $1\leq k\leq n$ and $0\leq v \leq r-1.$
As before in \ref{currents}, we also define the polynomial
\begin{equation}\label{a current}
    a_{k}(z):=\sum\limits_{v=0}^{r-1}(-1)^{v}a_{k}^{(v)}z^{r-1-v}.
\end{equation}

\begin{rem}\label{remark: a's as iterated commutators}
    Given $k\in \{1,...,n\}$, the operators $a_{k}^{(v)}$ correspond to Yangian elements defined by the following iterated commutator of generators:
    \[
    a_{k}(z)=\left[...\left[\left[e_{k}(z)m_{k}(z),e_{k+1}^{(0)}\right],e_{k+2}^{(0)}\right],...,e_{n}^{(0)}\right].
    \]
    We would not require this algebraic characterisation of the operators $a_{k}^{(v)}$ in the proofs of our results, except in the case $k=n$, where we will see that this follows from equation \ref{equality in cohomology:a=em}. However, the interested reader can show the general case by following the strategy in the proof of Corollary \ref{corollary on the operator [e_i,e_i-1]} by using an induction on $k$ and equation \ref{equality in cohomology:a=em}.
    \newline
\end{rem}
\noindent We will now describe the natural basis of $H^{*}(F^{n}\text{Quot}(V))$ which the operators $a_{k}^{(v)}$ produce.
First, given a linear operator
\[
T:H^{*}(F^n\text{Quot}(V))\rightarrow H^{*}(F^n\text{Quot}(V)\times C^{n})
\]
and a class $\gamma\in H^{*}(C^{n})$, we define the operator
\begin{equation}\label{operators coloured with cohomology classes T(gamma)}
T(\gamma):H^{*}(F^n\text{Quot}(V))\rightarrow H^{*}(F^n\text{Quot}(V))    
\end{equation}
to be the composition

\[
\begin{tikzcd}[cramped]
	{H^{*}(F^{n}\text{Quot}(V))} & {H^{*}(F^{n}\text{Quot}(V)\times C^{n})} \\
	& {H^{*}(F^{n}\text{Quot}(V)\times C^{n})} & {H^{*}(F^{n}\text{Quot}(V)).}
	\arrow["T", from=1-1, to=1-2]
	\arrow["{\cdot\pi^{*}(\gamma)}", from=1-2, to=2-2]
	\arrow["{\rho_{*}}", from=2-2, to=2-3]
\end{tikzcd}
\]
\newline
Above, we denote the projection maps from $F^{n}\text{Quot}(V)\times C^{n}$ to the first and second factor by $\rho$ and $\pi$ respectively.
Also, observe that the Hyperquot scheme $F^n\text{Quot}_{(0,...,0)}(V)$ is a point and let
\[
|0\rangle\in H^{*}(F^n\text{Quot}_{(0,...,0)}(V))
\]
be the unit class.
As before, let $\gamma_1,...,\gamma_{2g+2}$ be a basis of $H^{*}(C)$. The content of Section \ref{section: natural basis} will be the proofs of the following two theorems:

\begin{thm}\label{a's commute intro}
    For all $j,k\in\{1,...,n\}$ and $u,v\in\{0,...,r-1\}$ we have the following equality of operators $H^{*}(F^n\text{Quot}(V))\rightarrow H^{*}(F^n\text{Quot}(V)\times C\times C)$:
    \[
    a_{j}^{(u)} a_{k}^{(v)} =a_{k}^{(v)} a_{j}^{(u)}.
    \]
    Above, we make the convention that the operators $a_{j}^{(u)}$ and $a_{k}^{(v)}$ respectively contribute to the first and second factors of $C\times C$ on both sides of the equality. 

\end{thm}
\noindent In particular, this implies immediately that the $a_{k}^{(v)}(\gamma)$ commute with each other. This commutativity will be one of the major steps in the the proof of:

\begin{thm}\label{basis of cohomology intro}
    The set of elements
\[
\mathcal{B}:=\left\{a_{k_l}^{(v_l)}(\gamma_{i_l})...a_{k_2}^{(v_2)}(\gamma_{i_2})a_{k_1}^{(v_1)}(\gamma_{i_1})|0\rangle\right\}_{l\geq 0}
\]
with $k_{i}\in\{1,...,n\}$ and $v_{j}\in\{0,...,r-1\}$, forms a basis of  $H^{*}(F^{n}Quot(V))$.
\newline
\end{thm}
\noindent As a result of Theorems \ref{basis of cohomology intro} and \ref{a's commute intro}, we obtain a precise description of the additive structure of the singular cohomology of $F^{n}\text{Quot}(V)$, which we can succinctly describe as the following expression:
 
\begin{cor} \label{description as graded vec space}
    We have an isomorphism of graded vector spaces :
    \begin{equation}
        H^{*}(F^{n}\text{Quot}(V))\cong \text{Sym}\left(\bigoplus\limits^{i=1,...,2g+2}_{\substack{{k=1,...,n} \\ {v=0,...,r-1}}}\mathbb{Q}\cdot a_{k}^{(v)}(\gamma_{i})\right)|0\rangle.
    \end{equation}
\end{cor}
\noindent Note that the vector space $H^{*}(F^n\text{Quot}(V))$ has a natural $\mathbb{N}-$ grading which is the cohomological grading and a $\mathbb{N}^n-$ grading given by the $n-$tuples $\vec{d}$.\\\\
Corollary \ref{description as graded vec space} allows us to easily de-categorify and recover the Poincaré polynomial of $F^{n}\text{Quot}(V)$, which can also be recovered from a motivic computation in \cite{MoRi}.

\begin{cor}
    \[
        \sum\limits_{{\vec{d}\in\mathbb{N}^n}}\sum\limits_{k=0}^{2rd_{n}} \text{dim}_{\mathbb{Q}}\,H^{k}(F^{n}\text{Quot}_{\vec{d}}\,(V))t^{\vec{d}}z^{k}
    \]
    \[
  =  \prod_{k=1}^{n}\prod_{i=0}^{r-1}\frac{\left(1+t_{k}t_{k+1}...t_{n}z^{2i+1}\right)^{2g}}{\left(1-t_{k}t_{k+1}...t_{n}z^{2i}\right)\left(1-t_{k}t_{k+1}...t_{n}z^{2i+2}\right)}.
    \]
    \newline
    Above, we use the multi-index notation $t^{\vec{d}}=t_{1}^{d_{1}}...t_{n}^{d_{n}}.$

\end{cor}
\subsection{Related Works} 
\begin{enumerate}
    \item PBW bases for Yangians: In the work of Tsymbaliuk (for instance see Example 6.13 in \cite{Tsym1}), iterated commutators of the $e_{k}^{(v)}$ are used to define PBW bases for Yangians of $\mathfrak{sl}_{n}$. Not only do these elements share a resemblance with the algebraic description of the operators $a_{k}^{(v)}$ (cf. Remark \ref{remark: a's as iterated commutators}), one can show that the operators on $H^{*}(F^{n}\text{Quot}(V))$ associated to Tsymbaluik's PBW basis elements are also given by correspondences supported on the moduli space $\mathcal{Z}_{\vec{d}}^{i,n}$ \eqref{double nested quot}. However, these operators do not commute and are hence harder to work with for the purposes of this article.
    \item Skew-nested Quot schemes: The space $\mathcal{Z}_{\vec{d}}^{i,n}$, as well as several moduli spaces that we will define in Section $2$ (e.g. \ref{The moduli space: U}, \ref{The moduli space: W}, \ref{The moduli space: X}, \ref{Moduli space: e_i-1e_i} etc.) are examples of Skew-nested Quot schemes. Skew-nested Quot schemes are moduli spaces that parameterise ``flags of flags" of rank $r$ sub-sheaves of $V$. When rk $V=1$, skew-nested Hilbert schemes have been used by Sergej Monavari in \cite{MoRefinedDT} and \cite{Mo1} (see also \cite{DoubleNestedHilbPaper}) to study the enumerative geometry of local curves. Some of the skew-nested Quot schemes studied in this article, e.g. the space \ref{The moduli space: U}, are also reminiscent of ``moduli spaces of quadruples", that were crucially used in \cite{Neg1} (see also \cite{MaNe2} and \cite{YuZhao}). 
    \item The Cohomological Hall Algebra (CoHA) of a Curve: In an upcoming work \cite{Jindal-Lim}, Shivang Jindal and Woonam Lim show that the positive half of the action in \cite{MaNe1} can be thought of as an action of the CoHA of $0-$dimensional sheaves on a curve. The local version of this CoHA, i.e. when $C=\mathbb{A}^{1}$, is given by the CoHA associated to the stack of representations of the Jordan quiver:
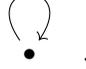
\[
\begin{tikzcd}[cramped]
	\bullet
	\arrow[from=1-1, to=1-1, loop, in=55, out=125, distance=10mm]
\end{tikzcd}
.
\]
    We expect that the positive half of the action in the present article, i.e. the action of the sub-algebra of $Y_{\hbar}^{r}({\mathfrak{sl}_{n+1}})$ that is generated by the $e_{k}^{(v)}$'s, comes naturally from the CoHA associated to the stack of length-$n$ sequences of morphisms of torsion sheaves on $C$. The local version of this CoHA is the CoHA associated to the stack of representations of the following ``unframed hand-saw quiver" \cite{Naka-Handsaw}, \cite{FinkelbergRybnikovQuantization}:
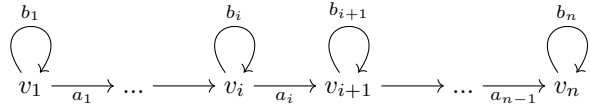
\[\begin{tikzcd}[cramped]
	{v_{1}} & {...} & {v_{i}} & {v_{i+1}} & {...} & {v_{n}}
	\arrow["{b_{1}}", from=1-1, to=1-1, loop, in=55, out=125, distance=10mm]
	\arrow["{a_{1}}"', from=1-1, to=1-2]
	\arrow[from=1-2, to=1-3]
	\arrow["{b_{i}}", from=1-3, to=1-3, loop, in=55, out=125, distance=10mm]
	\arrow["{a_{i}}"', from=1-3, to=1-4]
	\arrow["{b_{i+1}}", from=1-4, to=1-4, loop, in=55, out=125, distance=10mm]
	\arrow[from=1-4, to=1-5]
	\arrow["{a_{n-1}}"', from=1-5, to=1-6]
	\arrow["{b_{n}}", from=1-6, to=1-6, loop, in=55, out=125, distance=10mm]
\end{tikzcd}\]
with the relations $a_{i}\circ {b_{i}}=b_{i+1}\circ a_{i}$ for all $i\in\{1,...,n-1\}$.
By the work of Davison on dimensional reduction (see the appendix of \cite{Davison-Tripled-COHA}), this CoHA is isomorphic to the critical CoHA associated to the tripled quiver with standard potential of the $A_{n}-$quiver. Which again, by the dimensional reduction results of $\textit{loc. cit.}$, is isomorphic to the pre-projective CoHA of the $A_{n}-$quiver, that is well known to be the positive half of the Yangian of $\mathfrak{sl}_{n+1}$ (e.g. \cite{Var},\cite{Yang-Zhao},\cite{SchVass}, \cite{Shivangthesis} and \cite{Andrei-wheel} in K-theory).
    \item Quasimaps from $\mathbb{P}^{1}$ to flag varieties: The geometric representation theory of moduli of flags of sheaves on $C=\mathbb{P}^{1}$, is well-studied in the related case of Quasimap spaces, see e.g.\cite{Finkelberg-Kuznetsov}, \cite{FFFR}, \cite{Naka-Handsaw}, \cite{Braverman-Finkelberg}, \cite{Andrei-firstpaper} and \cite{che-shen}. In particular, Yangian actions on the (equivariant) cohomology of these spaces have been constructed in \cite{BFFR}, \cite{FFNR} and \cite{Kamnitzer-BFN-Springer}. See also \cite{Chen}, \cite{ciocan-fontanine}, \cite{Kim-Bumsig}, \cite{Stromme} and \cite{sinha-ontani2025intersectiontheoryhyperquotschemes}, where the geometry of these spaces is studied. 
\end{enumerate}

\subsection{Organisation of the article}
In Section $2$, we will introduce several moduli spaces. In the first few subsections we will discuss relative versions of spaces studied in \cite{MaNe1} and then we will introduce many examples of skew-nested Quot schemes that are relevant to our study. We discuss the geometric properties of these spaces, such as smoothness, irreducibility, l.c.i.-ness, reducedness, dimension counts etc. We will also define sheaves/divisors on these spaces and record interesting identities in cohomology between their Chern classes. In the last subsection of section $2$, we discuss the formalism of correspondences and state standard results that would help us compute compositions of geometrically defined linear maps on cohomology. The results of this section will enable us to prove our main theorems in Sections $3$ and $4$.\\\\
The goal of Section $3$ would be to prove Theorem \ref{Yangian relations satisfied}, that the operators \ref{e,f,m} satisfy the relations \eqref{mm commutation action} - \eqref{f Serre relations action} and finally, in Section $4$, we will prove Theorems \ref{a's commute intro} and \ref{basis of cohomology intro}.\\\\
In order for the reader to reach the main results of the article efficiently and quickly, we recommend that the reader first familiarises themselves with Subsection \ref{subsection:formalism of correspondences}, on the formalism of correspondences. After that, one may start reading Sections $3$ and $4$. At the start of every technical subsection of these sections, we give a precise reference to the parts of Section $2$ that we will use in that particular subsection.  
\subsection{Notations and Conventions}
In this article, we will consider several moduli spaces that parameterise certain diagrams of inclusions of rank $r$ sub-sheaves of $V$. Whenever there is an instance of a labeled inclusion:
\[\begin{tikzcd}[cramped]
	{\textbf{Sheaf}_{1}} & {\textbf{Sheaf}_{2}}
	\arrow["x", hook, from=1-1, to=1-2]
\end{tikzcd}\]
in a such a diagram, we mean that $$\textbf{Sheaf}_{2}/\textbf{Sheaf}_{1}\cong \mathbb{C}_{x}.$$ That is, $\textbf{Sheaf}_{2}/\textbf{Sheaf}_{1}$ is a length 1 sheaf on $C$ supported on some point $x\in C$ (see Diagram \ref{double nested quot} for an example).
All moduli spaces considered in this article are projective schemes over $\mathbb{C}$ since they can be written as fibre products of morphisms of Hyperquot Schemes, which also explains the representability of their moduli functors. Given a morphism of schemes $f:X\rightarrow Y$ and a sheaf $F$ on $Y$, we will often denote $f^{*}(F)$ by $F$, when no confusion is likely to occur, for the sake of legibility. The singular cohomology groups in this article are always defined with $\mathbb{Q}-$coefficients.
\subsection{Acknowledgements} First and foremost, I would like to sincerely express my gratitude to Andrei Neguț, this article would not have been possible without his insight, perspective and support. I would also like to thank Alina Marian for her encouragement and useful suggestions, which I really appreciate. This article is intellectually indebted, in large part, to the article \cite{MaNe1} of Marian and Neguț.\\\\
Special thanks go to Jefferson Baudin and Sergej Monavari, several discussions with them significantly improved my understanding of the geometric aspects of this article.\\\\
I would also like to thank Camilla Felisetti, Niccolò Giacomini, Shivang Jindal, Joel Kamnitzer, Christina Kapatsori, Enrico Lampetti, Elsa Maneval, Anton Mellit, Luca Morstabilini, Aliaksandra Novik, Tudor Pădurariu, Sebastian Schlegel Mejia, Shubham Sinha, Dimitri Wyss and Yu Zhao for interesting discussions, questions and suggestions.\\
\section{Geometry of Some Nested and Skew-nested Quot Schemes.}

\subsection{The Hyperquot scheme ${\text{F}^{\text{n+1}}\text{Quot}_{{\vec{\text{d}}''_{k}}}\,(\text{V})}$}\label{subsection on moduli space: Fn+1 Quot}
Let us recall from the introduction that given an $n-$tuple $\vec{d}=(d_{1},...,d_{n})$, we can define the $n-$tuple\\ $\vec{d}'_{k}:=(d_1,...,d_{k-1},d_{k}+1,d_{k+1},...,d_n)$ and the $(n+1)-$tuple $\vec{d}''_{k}:=\\(d_1,...,d_{k-1},d_{k},d_{k}+1,d_{k+1},...,d_n)$. They give rise to the corresponding Hyperquot schemes $F^n\text{Quot}_{\vec{d}}\,(V)$,  $F^n\text{Quot}_{\vec{d}'_{k}}\,(V)$ and  $F^{n+1}\text{Quot}_{\vec{d}''_{k}}\,(V)$, which fit into the diagram \ref{correspondence diagram defining e,f}, involving morphisms $\pi_{\pm}$ and $\pi_{C}$. To ensure the non-emptiness of these spaces, we will assume that $\vec{d}$ is non-decreasing and $d_{k+1}-d_{k}>0$.\\\\
In this subsection, we will construct some sheaves on $F^{n+1}\text{Quot}_{\vec{d}''_{k}}\,(V)$, which play a key role later on, and note some useful facts about them.
Let 
\[
\mathcal{E}_{n}\subseteq ...\subseteq\mathcal{E}_{k+1}\subseteq \mathcal{F}_{k}\subseteq \mathcal{E}_{k}\subseteq...\subseteq\mathcal{E}_1\subseteq \mathcal{E}_{0}=\pi^{*}(V)
\]
be the universal flag on $F^{n+1}\text{Quot}_{\vec{d}''_{k}}\,(V)\times C$ and consider the embedding 
\[
\Delta:=Id\times\pi_{C}:F^{n+1}\text{Quot}_{\vec{d}''_{k}}\,(V)\hookrightarrow F^{n+1}\text{Quot}_{\vec{d}''_{k}}\,(V)\times C.
\]
Let 
$\mathscr{F}_{k}:=  \mathcal{E}_{k}/ \mathcal{F}_{k}$,
which is flat over $F^{n+1}\text{Quot}_{\vec{d}''_{k}}\,(V)$, and observe that 
\[
\mathscr{L}_{k}:={\mathscr{F}_{k}}_{_{|_{\Delta}}}
\] 
is a line bundle on $F^{n+1}\text{Quot}_{\vec{d}''_{k}}\,(V)$. In fact, the bundle $\mathscr{L}_{k}$ is precisely the one that was used to define the operators \ref{definition of e operator} and \ref{definition of f operator} in the introduction. By a harmless abuse of notation, we denote the pull-backs of $\mathscr{L}_{k}$ and that of it's first Chern class $\lambda_{k}$ via the natural projection map $F^{n+1}\text{Quot}_{\vec{d}''_{k}}\,(V)\times C\rightarrow F^{n+1}\text{Quot}_{\vec{d}''_{k}}\,(V)$ by the same notation. \\\\
We record the following short exact sequence on $F^{n+1}\text{Quot}_{\vec{d}''_{k}}\,(V)\times C$:
\begin{equation}\label{short exact seq on Fn+1Quot x C}
    0\rightarrow\mathcal{F}_{k}\rightarrow\mathcal{E}_{k}\rightarrow \mathscr{L}_{k}\otimes\mathcal{O}_{\Delta}\rightarrow 0,
\end{equation}
which yields the following identity in $H^{*}( F^{n+1}\text{Quot}_{\vec{d}''_{k}}\,(V)\times C)[z]$:
\begin{equation}\label{short exact seq cohomology identity on Fn+1Quot x C}
    c(\mathcal{F}_{k},z)=c(\mathcal{E}_{k},z)\cdot \frac{z-\lambda_{k}+\delta}{z-\lambda_{k}}.
\end{equation}
Where $\delta:=[\Delta]$.
\\\\
If we restrict \ref{short exact seq on Fn+1Quot x C} to $\Delta$, we obtain an exact sequence on $F^{n+1}\text{Quot}_{\vec{d}''_{k}}(V)$ :
\[
{\mathcal{F}_{k}}_{|_{\Delta}}\rightarrow{\mathcal{E}_{k}}_{|_{\Delta}}\rightarrow\mathscr{L}_{k}\rightarrow 0.
\]
Let us define 
\begin{equation}\label{definition of r-1 rk bundle G}
    \mathcal{G}_{k}:=\text{Ker }({\mathcal{E}_{k}}_{|_{\Delta}}\rightarrow\mathscr{L}_{k})=\text{Im }({\mathcal{F}_{k}}_{|_{\Delta}}\rightarrow{\mathcal{E}_{k}}_{|_{\Delta}}),
\end{equation}
and note that $\mathcal{G}_{k}$ is a rank $r-1$ vector bundle on $F^{n+1}\text{Quot}_{\vec{d}''_{k}}(V)$.

\subsection{Virtual projective bundles}\label{subsection: projective bundles}
In this subsection, we realise $F^{n+1}\text{Quot}_{\vec{d}''_{k}}(V)$ as a projective bundle of a two step complex of locally free sheaves, in two different ways and state a useful formula in cohomology. We follow the exposition in \cite{MaNe1}.\\\\
Given a locally free sheaf $A$ on $X$, recall that the projective bundle of one-dimensional quotients 
\[
\pi:\mathbb{P}_{X}(A)\rightarrow X
\]
carries a tautological line bundle $\mathcal{O}(1)$, along with a universal quotient morphism:
\[
\pi^{*}(A)\rightarrow \mathcal{O}(1).
\]
\begin{defn}\label{virtual projective bundle}
    Let $\phi:B\rightarrow A$ be a map of locally free sheaves on a variety $X$. We define the projectivisation of $\phi$, denoted by $\pi{'}:\mathbb{P}_{X}(A-B)\rightarrow X$ to be the derived zero locus of the composition 
    \[
    \pi^{*}(B)\xrightarrow{\pi^{*}(\phi)}\pi^{*}(A)\rightarrow\mathcal{O}(1).
    \]
\end{defn}
\noindent The derived zero locus coincides with the classical scheme theoretic zero locus precisely when the dimension of the latter is the expected dimension 
\[
\text{dim}\,X+\text{rank}\,A-\text{rank}\,B-1,
\]
in which case, the morphism $\pi{'}$ is automatically l.c.i..\\\\
Assuming that $\mathbb{P}_{X}(A-B)$ is of expected dimension, let us denote $\mathcal{O}(1)$, to be the pull-back of the tautological bundle on $\mathbb{P}_{X}(A)$ to $\mathbb{P}_{X}(A-B)$. Let us define $\lambda:=c_{1}(\mathcal{O}(1))\in H^{2}(\mathbb{P}_{X}(A-B))$. We have the following useful pushforward formula in cohomology (see section 2 in \cite{MaNe1}):

\begin{equation}\label{pushforward formula projective bundle}
    \pi^{'}_{*}\left[\frac{1}{z-\lambda}\right]=\left[\frac{c(B,z)}{c(A,z)}\right]_{z<0}.
\end{equation}
\newline
Here, the total Chern class $c(-,z)$ is defined as in equation \ref{total chern class definition}.
The subscript $z<0$ means that we expand the R.H.S. of the above formula in non-positive powers of $z$ and remove all terms in non-negative powers of $z$. Note that by setting $B$ to be the $0$ sheaf, we recover the well known case of a usual projective bundle of a locally free sheaf.\\\\
Now we will express the maps $\pi_{+}\times \pi_{C}$ and $\pi_{-}\times \pi_{C}$ (cf. diagram \ref{correspondence diagram defining e,f}) as certain virtual projective bundles. We do not provide proofs of the following two propositions since they are mutatis mutandis the proofs in equation (4) and Lemma 1 in \cite{MaNe1}.
\begin{prop}\label{pi-times piC as a projective bundle}
Let
\[
\mathcal{E}_{n}\subseteq ...\subseteq\mathcal{E}_{k+1}\subseteq\mathcal{E}_{k}\subseteq...\subseteq\mathcal{E}_1\subseteq \mathcal{E}_{0}=\pi^{*}(V).
\]
be the universal flag on $F^{n}\text{Quot}_{{\vec{d}}}\,(V)$. Then, the map $$\pi_{-}\times \pi_{C}:F^{n+1}\text{Quot}_{{\vec{d}''}_{k}}\,(V)\rightarrow F^{n}\text{Quot}_{\vec{d}}\,(V)\times C$$ can be identified with the virtual projective bundle $$\mathbb{P}_{F^{n}\text{Quot}_{\vec{d}}\,(V)\times C}(\mathcal{E}_{k}-\mathcal{E}_{k+1})$$  for all $k=1,...,n$ and such that the tautological bundle $\mathscr{L}_{k}$ is identified with $\mathcal{O}(1)$. Here we make the convention that $\mathcal{E}_{n+1}=0$, therefore $F^{n+1}\text{Quot}_{{\vec{d}''}_{n}}\,(V)$ is the usual projectivisation of $\mathcal{E}_{n}$ over $F^{n}\text{Quot}_{\vec{d}}\,(V)\times C$.
\end{prop}

\begin{prop}\label{pi+times piC as a projective bundle}
Let
\[
\mathcal{E}_{n}\subseteq ...\subseteq\mathcal{E}_{k+1}\subseteq\mathcal{F}_{k}\subseteq\mathcal{E}_{k-1}\subseteq...\subseteq\mathcal{E}_1\subseteq \mathcal{E}_{0}=\pi^{*}(V).
\]
be the universal flag on $F^{n}\text{Quot}_{{\vec{d}'}_{k}}\,(V)$. Then, the map $$\pi_{+}\times \pi_{C}:F^{n+1}\text{Quot}_{{\vec{d}''}_{k}}\,(V)\rightarrow F^{n}\text{Quot}_{{\vec{d}'}_{k}}\,(V)\times C$$ can be identified with the virtual projective bundle $$\mathbb{P}_{F^{n}\text{Quot}_{{\vec{d}'}_{k}}\,(V)\times C}\,(\mathcal{F}_{k}^{\vee}\otimes \omega_{C}-\mathcal{E}_{k-1}^{\vee}\otimes\omega_{C})$$ for all $k=1,...,n$ and such that the tautological bundle $\mathscr{L}_{k}$ is identified with $\mathcal{O}(-1)$.
\end{prop}

\subsection{Equations defining some Hyperquot schemes}\label{subsection on diagonal embeddings}
In this subsection, we study two instances of certain diagonal embeddings of Hyperquot schemes being cut-out by regular sections of vector bundles on bigger spaces and record some useful facts for later use. First, let us fix $\vec{d}=(d_1,...,d_{n})$ to be an $n-$tuple of non-negative integers. Let us define $\mathfrak{P}_{\vec{d}}$ and $\mathfrak{Z}_{\vec{d}}$ respectively to be the fine moduli spaces parameterising the following diagrams of inclusions of rank $r$ sub-sheaves of $V$ on $C$:

\begin{equation}\tag{$\mathfrak{P}_{\vec{d}}$}\label{Moduli space : mathfrak P}
\begin{tikzcd}[cramped]
	{E^{'}_{n}} \\
	& {E_{n-1}} & {E_{n-2}\hookrightarrow...\hookrightarrow E_{1}} & V \\
	{E_{n}}
	\arrow[hook, from=1-1, to=2-2]
	\arrow[hook, from=2-2, to=2-3]
	\arrow[hook, from=2-3, to=2-4]
	\arrow[hook, from=3-1, to=2-2]
\end{tikzcd},
\end{equation}\\

\begin{equation}\tag{$\mathfrak{Z}_{\vec{d}}$}\label{Moduli space : mathfrak Z}
\begin{tikzcd}[cramped]
	{E_{n}} \\
	& {F_{n}} & {E_{n-1}\hookrightarrow...\hookrightarrow E_{1}} & V \\
	{H_{n}}
	\arrow["x", hook, from=1-1, to=2-2]
	\arrow[hook, from=2-2, to=2-3]
	\arrow[hook, from=2-3, to=2-4]
	\arrow["x", hook, from=3-1, to=2-2]
\end{tikzcd}.
\end{equation}\\
The above are diagrams of inclusions of certain rank $r$ sub-sheaves of $V$ such that:
\begin{enumerate}
    \item $\text{dim}_{\mathbb{C}}(V/E_{*})=d_{*}$ for all $*\in\{1,...,n\}$.
    \item In diagram \ref{Moduli space : mathfrak P}, we also require $\text{dim}_{\mathbb{C}}(V/E^{'}_{n})=d_{n}$.
\end{enumerate}
\subsection{The moduli space $\mathfrak{P}_{\vec{d}}$}
First, let us focus on the space \ref{Moduli space : mathfrak P}. To ensure it's non-emptiness, let us assume that the sequence $(d_{1},...,d_{n})$ is non-decreasing. 
If we denote the $(n-1)-$tuple $(d_{1},...,d_{n-1})$ by $\vec{d}_{\downarrow}$ and if $\varepsilon$ is the natural forgetful map:
\begin{equation}\label{forgetful map FnQuot to Fn-1}
    \varepsilon :F^{n}\text{Quot}_{{\vec{d}}}\,(V)\rightarrow F^{n-1}\text{Quot}_{{\vec{d}_{\downarrow}}}(V),
\end{equation}
then note that $\mathfrak{P}_{\vec{d}}$ is precisely the self-fibre product of $\varepsilon$, i.e. 
\begin{equation}\label{self-fibre product}
\mathfrak{P}_{\vec{d}}=F^{n}\text{Quot}_{{\vec{d}}}\,(V)\times_{F^{n-1}\text{Quot}_{{\vec{d}_{\downarrow}}}(V)} F^{n}\text{Quot}_{{\vec{d}}}\,(V).    
\end{equation}
\noindent Observe that the map $\varepsilon$ is smooth of relative dimension $r(d_{n}-d_{n-1})$, indeed the fibres of $\varepsilon$ are Quot schemes. Therefore, we have the following proposition:
\begin{prop}\label{proposition on smoothness and dimension of moduli space P}
    The space $\mathfrak{P}_{\vec{d}}$ is a smooth projective variety of dimension $r({2d_{n}-d_{n-1}})$.
\end{prop}
\noindent Let $\Delta$ be the diagonal embedding:
\[
\Delta:F^{n}\text{Quot}_{{\vec{d}}}\,(V)\hookrightarrow  \mathfrak{P}_{\vec{d}}
\]
and let us try to cut-out the sub-scheme $\Delta$ by a regular section of a vector bundle. To this end, consider the universal diagram on $\mathfrak{P}_{\vec{d}}\times C$:
\[
\begin{tikzcd}[cramped]
	{\mathcal{E}^{'}_{n}} \\
	& {\mathcal{E}_{n-1}} & {\mathcal{E}_{n-2}\hookrightarrow...\hookrightarrow \mathcal{E}_{1}} & V \\
	{\mathcal{E}_{n}}
	\arrow[hook, from=1-1, to=2-2]
	\arrow[hook, from=2-2, to=2-3]
	\arrow[hook, from=2-3, to=2-4]
	\arrow[hook, from=3-1, to=2-2]
\end{tikzcd}
\]
and consider the universal quotient $$\mathscr{F}^{'}_{n}:=\frac{\mathcal{E}_{n-1}}{\mathcal{E}_{n}^{'}},$$ which is a coherent sheaf on $\mathfrak{P}_{\vec{d}}\times C$, flat over $\mathfrak{P}_{\vec{d}}$. The composition 
\[
s:\mathcal{E}_{n}\hookrightarrow \mathcal{E}_{n-1}\rightarrow \mathscr{F}^{'}_{n}
\]
defines a global section of the coherent sheaf $\mathcal{H}om(\mathcal{E}_{n},\mathscr{F}^{'}_{n})$ on $\mathfrak{P}_{\vec{d}}\times C$ and consequently a section of
\[
\rho_{*}\mathcal{H}om(\mathcal{E}_{n},\mathscr{F}^{'}_{n}),
\]
where $\rho:\mathfrak{P}_{\vec{d}}\times C\rightarrow \mathfrak{P}_{\vec{d}}$ is the natural projection. The following proposition is a straightforward consequence of cohomology and base change applied to the flat morphism $\rho$:

\begin{prop}\label{proposition to be used in tautological generation}
   The sheaf $\rho_{*}\mathcal{H}om(\mathcal{E}_{n},\mathscr{F}^{'}_{n})$ is a rank $r(d_{n}-d_{n-1})$ vector bundle on $\mathfrak{P}_{\vec{d}}$, the higher direct images $$R^{i}\rho_{*}\mathcal{H}om(\mathcal{E}_{n},\mathscr{F}^{'}_{n})$$
   vanish for $i>0$ and the section $$s\in H^{0}(\rho_{*}\mathcal{H}om(\mathcal{E}_{n},\mathscr{F}^{'}_{n}))$$ cuts out the sub-scheme $\Delta$. By the dimension count in Proposition \ref{proposition on smoothness and dimension of moduli space P}, the section $s$ is regular.
\end{prop}
\subsection{The moduli space $\mathfrak{Z}_{\vec{d}}$} Now, let us turn our attention to the space \ref{Moduli space : mathfrak Z}. To ensure it's non-emptiness, let us assume that the sequence $(d_{1},...,d_{n})$ is non-decreasing and that $d_{n}-d_{n-1}\geq 1$. Consider the $n-$tuple $${}^{'}\vec{d}_{n}:=(d_1,...,d_{n-1},d_{n}-1),$$ the $(n+1)-$tuple $${}^{''}\vec{d}_{n}:=(d_1,...,d_{n-1},d_{n}-1,d_{n})$$ and recall the map
\[
\pi_{-}\times \pi_{C}:F^{n+1}\text{Quot}_{{}^{''}{\vec{d}}_{n}}\,(V)\rightarrow F^{n}\text{Quot}_{{}^{'}{\vec{d}}_{n}}\,(V)\times C.
\]
Note that $\mathfrak{Z}_{\vec{d}}$ is precisely the fibre product of $\pi_{-}\times \pi_{C}$ with itself i.e.
\[
\mathfrak{Z}_{\vec{d}}=F^{n+1}\text{Quot}_{{}^{''}{\vec{d}}_{n}}\,(V)\times_{F^{n}\text{Quot}_{{}^{'}{\vec{d}}_{n}}\,(V)\times C}F^{n+1}\text{Quot}_{{}^{''}{\vec{d}}_{n}}\,(V).
\]
Let $p_{1},p_{2}:\mathfrak{Z}_{\vec{d}}\rightarrow F^{n+1}\text{Quot}_{{}^{''}{\vec{d}}_{n}}\,(V)$ be the projection maps to the two factors. It follows from Proposition \ref{pi-times piC as a projective bundle} that $\pi_{-}\times\pi_{C}$ is a $\mathbb{P}^{r-1}-$ bundle. Therefore we obtain:

\begin{prop}
The space $\mathfrak{Z}_{\vec{d}}$ is a smooth projective variety of dimension $r(d_{n}+1)-1.$ 
\end{prop}

\noindent Let $\mathscr{L}_{1}$ and $\mathscr{L}_{2}$ be the tautological line bundles on $\mathfrak{Z}_{d}$, whose fibres over a closed point, as in diagram \ref{Moduli space : mathfrak Z}, correspond to the length one sheaves $F_{n}/E_{n}$ and $F_{n}/H_{n}$, respectively. Let us also denote their respective first Chern classes by $\lambda_{1}$ and $\lambda_{2}$.\\\\
Consider the universal diagram of rank $r$ locally free sheaves on $\mathfrak{Z}_{\vec{d}}\times C$ :

\[
\begin{tikzcd}[cramped]
	{\mathcal{E}_{n}} \\
	& {\mathcal{F}_{n}} & {\mathcal{E}_{n-1}\hookrightarrow...\hookrightarrow \mathcal{E}_{1}} & V \\
	{\mathcal{H}_{n}}
	\arrow[ hook, from=1-1, to=2-2]
	\arrow[hook, from=2-2, to=2-3]
	\arrow[hook, from=2-3, to=2-4]
	\arrow[ hook, from=3-1, to=2-2]
\end{tikzcd},
\]
\newline
note that their is a map $t_{C}:=\pi_{C}\circ p_{1}=\pi_{C}\circ p_{2}: \mathfrak{Z}_{\vec{d}}\rightarrow C$, which sends a closed point, as in diagram \ref{Moduli space : mathfrak Z}, to the support point 
\[
x=\text{Supp }(F_{n}/E_{n})=\text{Supp }(F_{n}/H_{n})
\]
and let $\Gamma:\mathfrak{Z}_{\vec{d}}\hookrightarrow \mathfrak{Z}_{\vec{d}}\times C$ be the graph of $t_{C}$. As in equation \ref{definition of r-1 rk bundle G},  we may define rank $r-1$ locally free sheaves $\mathcal{G}_{1}$ and $\mathcal{G}_{2}$ on $\mathfrak{Z}_{\vec{d}}\,$:
\[
 \mathcal{G}_{1}:=\text{Ker }({\mathcal{F}_{n}}_{|_{\Gamma}}\rightarrow\mathscr{L}_{1})=\text{Im }({\mathcal{E}_{n}}_{|_{\Gamma}}\rightarrow{\mathcal{F}_{n}}_{|_{\Gamma}}),\,\, \mathcal{G}_{2}:=\text{Ker }({\mathcal{F}_{n}}_{|_{\Gamma}}\rightarrow\mathscr{L}_{2})=\text{Im }({\mathcal{H}_{n}}_{|_{\Gamma}}\rightarrow{\mathcal{F}_{n}}_{|_{\Gamma}}).
\]
\noindent Now, let $\Delta$ be the diagonal embedding;
\[
\Delta:F^{n+1}\text{Quot}_{{}^{''}{\vec{d}}_{n}}\,(V)\hookrightarrow  \mathfrak{Z}_{\vec{d}}\,.
\]
As before, we would like to cut-out $\Delta$ by a regular section of a vector bundle, for later use. Indeed, one observes that this is achieved by the section of the rank $r-1$ vector bundle $\mathcal{G}_{2}^{\vee}\otimes\mathscr{L}_{1}$ given by the composition:
\begin{equation}\label{regular section needed for multiplication operators}
\mathcal{G}_{2}\hookrightarrow {\mathcal{F}_{n}}_{|_{\Gamma}}\rightarrow\mathscr{L}_{1}.
\end{equation}
As a consequence, let us derive a useful identity in the cohomology of\\ $F^{n}\text{Quot}_{{\vec{d}}}\,(V)\times F^{n}\text{Quot}_{{\vec{d}}}\,(V)\times C.$ To this end, consider the following commutative diagram:

\[\begin{tikzcd}
	{F^{n+1}\text{Quot}_{{}^{''}{\vec{d}}_{n}}\,(V)} &&& {\mathfrak{Z}_{\vec{d}}} \\
	{F^{n}\text{Quot}_{{\vec{d}}}\,(V)\times C} &&& {F^{n}\text{Quot}_{{\vec{d}}}\,(V)\times F^{n}\text{Quot}_{{\vec{d}}}\,(V)\times C.}
	\arrow["\Delta", from=1-1, to=1-4]
	\arrow["{\pi_{+}\times\pi_{C}}"', from=1-1, to=2-1]
	\arrow["{k_{1}\times k_{2}\times t_{C}}", from=1-4, to=2-4]
	\arrow["{\Delta_{F^{n}\text{Quot}_{{\vec{d}}}\,(V)}\times Id_{C}}"', from=2-1, to=2-4]
\end{tikzcd}\]
\newline
In the above diagram, for $i=1,2$, we define $k_{i}:=\pi_{+}\circ p_{i}$. Let $\lambda$ be the first Chern class of the tautological line bundle on $F^{n+1}\text{Quot}_{{}^{''}{\vec{d}}_{n}}\,(V)$ and note that $\lambda=\Delta^{*}(\lambda_{1})=\Delta^{*}(\lambda_{2})$. Let us pushforward the power series

\[
\frac{1}{z-\lambda}\in H^{*}(F^{n+1}\text{Quot}_{{}^{''}{\vec{d}}_{n}}\,(V))[[z^{-1}]]
\]
\newline
down to ${F^{n}\text{Quot}_{{\vec{d}}}\,(V)\times F^{n}\text{Quot}_{{\vec{d}}}\,(V)\times C}$ in the two different ways given by the above commutative diagram.\\\\
On one hand, the fact that $\Delta$ is cut out by a regular section of $\mathcal{G}_{2}^{\vee}\otimes\mathscr{L}_{1}$, yields that

\[
\Delta_{*}\left(\frac{1}{z-\lambda}\right)=\frac{c(\mathcal{G}_{2},\lambda_{1})}{z-\lambda_{1}}.
\]
\newline
On the other hand, it follows from formula \ref{pushforward formula projective bundle} and Proposition \ref{pi+times piC as a projective bundle} that

\[
(\pi_{+}\times \pi_{C})_{*}\left(\frac{1}{z-\lambda}\right)=\left[-\frac{c(\mathcal{E}_{n-1},z+K_{C})}{c(\mathcal{E}_{n},z+K_{C})}\right]_{z<0}=1-\frac{c(\mathcal{E}_{n-1},z+K_{C})}{c(\mathcal{E}_{n},z+K_{C})}.
\]
\newline
Here recall that $\mathcal{E}_{n-1}$ and $\mathcal{E}_{n}$ are universal vector bundles on $F^{n}\text{Quot}_{\vec{d}}\times C$ (cf. equation \ref{Universal flag}).
Therefore, we obtain the following identity in \\$H^{*}(F^{n}\text{Quot}_{{\vec{d}}}\,(V)\times F^{n}\text{Quot}_{{\vec{d}}}\,(V)\times C)$:
\begin{equation}\label{equation that will yield multiplication operators}
\left(k_{1}\times k_{2}\times t_{C}\right)_{*}\left(\frac{c(\mathcal{G}_{2},\lambda_{1})}{z-\lambda_{1}}\right)=\left(\Delta_{F^{n}\text{Quot}_{{\vec{d}}}\,(V)}\times Id_{C}\right)_{*}\left(1-\frac{c(\mathcal{E}_{n-1},z+K_{C})}{c(\mathcal{E}_{n},z+K_{C})}\right).
\end{equation}

\subsection{The skew-nested Quot scheme $\mathcal{Z}_{\vec{d}}^{i,j}$}\label{moduli space Z}
 Let $\vec{d}=(d_1,...,d_{n})$ be an $n-$tuple of non-negative integers and for $1\leq i<j\leq n$, let $\vec{d}_{i,j}$ be the $n-$tuple $(d_1,...,d_{i-1},d_{i}+1,...,d_{j}+1,d_{j+1},...,d_{n})$.  We define the space $\mathcal{Z}_{\vec{d}}^{i,j}$ to be the fine moduli space with the following description of it's closed points: 
\begin{equation}\label{Moduli space: The one key}
\begin{tikzcd}[cramped]
	& {E_{j}} & {...\hookrightarrow E_{i+1}\,\,\,\,} & {E_{i}} & {...\hookrightarrow E_{1}} & V \\
	{E_{n}\hookrightarrow...\hookrightarrow E_{j+1}} & {F_{j}} & {...\hookrightarrow F_{i+1}} & {F_{i}}
	\arrow[hook, from=1-2, to=1-3]
	\arrow["{\,\,\,\,\,\,\,\,\,\,\,\,...}", shift left=5, draw=none, from=1-2, to=2-2]
	\arrow[hook, from=1-3, to=1-4]
	\arrow[shift right=2, draw=none, from=1-3, to=2-3]
	\arrow[shift right=5, draw=none, from=1-3, to=2-3]
	\arrow[hook, from=1-4, to=1-5]
	\arrow[hook, from=1-5, to=1-6]
	\arrow[hook, from=2-1, to=2-2]
	\arrow["p"'{pos=0.4}, hook, from=2-2, to=1-2]
	\arrow[hook, from=2-2, to=2-3]
	\arrow["p"'{pos=0.4}, shift right=3, hook, from=2-3, to=1-3]
	\arrow[hook, from=2-3, to=2-4]
	\arrow["p"'{pos=0.4}, hook, from=2-4, to=1-4].
\end{tikzcd}
\end{equation}
\newline That is, $\mathcal{Z}_{\vec{d}}^{i,j}$ parameterises pairs of length $n$ flags; 
\[
(E_n\subseteq...\subseteq E_1\subseteq V)\in F^{n}\text{Quot}_{\vec{d}}\,(V) \text{ and } 
(F_n\subseteq...\subseteq F_1\subseteq V)\in F^{n}\text{Quot}_{\vec{d}_{i,j}}\,(V),
\]
such that $E_t=F_t$ for $t=1,...,i-1,j+1,...,n$ and $F_{s}\hookrightarrow E_{s}$, where $E_s/F_s$ are length one sheaves supported on the same, but arbitrary, point $p\in C$ for $s=i,i+1,...,j$.\\\\
We note that $\mathcal{Z}_{\vec{d}}^{i,j}$ is non-empty if and only if $\vec{d}$ is non-decreasing and $d_{j+1}-d_{j}\geq 1.$ For the proceeding statements about $\mathcal{Z}^{i,j}_{\vec{d}}$ in this section, we naturally assume this non-emptiness condition.
\subsection{Some sheaves on $\mathcal{Z}_{\vec{d}}^{i,j}$}\label{subsection:Some sheaves on Z} Consider the natural map $q_{C}\colon\mathcal{Z}_{\vec{d}}^{i,j}\rightarrow C$ which sends a closed point of the form \ref{Moduli space: The one key} to the support point $p\in C$ and let \\$\Delta:\mathcal{Z}_{\vec{d}}^{i,j}\hookrightarrow \mathcal{Z}_{\vec{d}}^{i,j}\times C$ denote the graph of the function $q_{C}$. The space $\mathcal{Z}_{\vec{d}}^{i,j}\times C$ carries a universal diagram of inclusions of locally free sheaves:
\begin{equation}\label{Universal diagram on the one key}
\begin{tikzcd}[cramped]
	& {\mathcal{E}_{j}} & {...\hookrightarrow \mathcal{E}_{i+1}\,\,\,\,} & {\mathcal{E}_{i}} & {...\hookrightarrow \mathcal{E}_{1}} & V \\
	{\mathcal{E}_{n}\hookrightarrow...\hookrightarrow \mathcal{E}_{j+1}} & {\mathcal {F}_{j}} & {...\hookrightarrow \mathcal {F}_{i+1}} & {\mathcal {F}_{i}}
	\arrow[hook, from=1-2, to=1-3]
	\arrow["{\,\,\,\,\,\,\,\,\,\,\,\,...}", shift left=5, draw=none, from=1-2, to=2-2]
	\arrow[hook, from=1-3, to=1-4]
	\arrow[shift right=2, draw=none, from=1-3, to=2-3]
	\arrow[shift right=5, draw=none, from=1-3, to=2-3]
	\arrow[hook, from=1-4, to=1-5]
	\arrow[hook, from=1-5, to=1-6]
	\arrow[hook, from=2-1, to=2-2]
	\arrow[{pos=0.4}, hook, from=2-2, to=1-2]
	\arrow[hook, from=2-2, to=2-3]
	\arrow[{pos=0.4}, shift right=3, hook, from=2-3, to=1-3]
	\arrow[hook, from=2-3, to=2-4]
	\arrow[{pos=0.4}, hook, from=2-4, to=1-4].
\end{tikzcd}
\end{equation}
Hence, we have universal short exact sequences on $\mathcal{Z}_{\vec{d}}^{i,j}$ for all $l=i,i+1,...,j$:
\[
0\rightarrow \mathcal{F}_{l}\rightarrow\mathcal{E}_{l}\rightarrow\mathscr{F}_{l}\rightarrow 0.
\]
Where the torsion sheaf $\mathscr{F}_{l}$ is the cokernel of $ \mathcal{F}_{l}\hookrightarrow\mathcal{E}_{l}$.
If we restrict the above short exact sequence to the copy of $\mathcal{Z}_{\vec{d}}^{i,j}$, given by the embedding $\Delta:\mathcal{Z}_{\vec{d}}^{i,j}\hookrightarrow \mathcal{Z}_{\vec{d}}^{i,j}\times C$, we obtain an exact sequence 
\begin{equation}\label{construction of line bundle L}
{\mathcal{F}_{l}}_{|_{\Delta}}\rightarrow{\mathcal{E}_{l}}_{|_{\Delta}}\rightarrow\mathscr{L}_{l}\rightarrow 0.
\end{equation}
Observe that $\mathscr{L}_{l}$ is precisely the tautological line bundle on $\mathcal{Z}_{\vec{d}}^{i,j}$, whose fibre over a closed point as in diagram \ref{Moduli space: The one key} is identified canonically with the line $E_{l}/F_{l}$. Moreover, since ${\mathcal{E}_{l}}_{|_{\Delta}}$ and $\mathscr{L}_{l}$ are locally free sheaves, we obtain that the sheaf 
$$\mathcal{G}_{l}:=\,\text{Ker}\,({\mathcal{E}_{l}}_{|_{\Delta}}\rightarrow\mathscr{L}_{l})$$
is a locally free sheaf of rank $r-1$ on $\mathcal{Z}_{\vec{d}}^{i,j}$.\\\\
Let $1\leq k\leq n$ and let us define the rank $r-1$ bundle $\mathcal{G}_{\vec{d}}^{k,n}$ on $\mathcal{Z}_{\vec{d}}^{k,n}$ as
\begin{equation}\label{equation: definition of G for mod space Z}
    \mathcal{G}_{\vec{d}}^{k,n}:=\mathcal{G}_{k}:=\,\text{Ker}\,({\mathcal{E}_{k}}_{|_{\Delta}}\rightarrow\mathscr{L}_{k}).
\end{equation}
 This is precisely the sheaf we used to define the operators \ref{Operators a's definition} in the introduction.
 \begin{rem}
     It turns out that the choice $l=k$ to define the operators $a_{k}^{(v)}$ is important in the sense that if we used the vector bundles $\mathcal{G}_{l}$, with $l\neq i$ to define these operators, then they would not commute.
 \end{rem}
\subsection{Smoothness of $\mathcal{Z}_{\vec{d}}^{i,j}$}
We will now prove the following proposition:

\begin{prop}\label{proposition: smoothness and dimension of moduli space Z}
    The moduli space $\mathcal{Z}_{\vec{d}}^{i,j}$, is a smooth projective variety of dimension $rd_{n}$, if $j\neq n$ and dimension $r(d_{n}+1)$ if $j=n$.
\end{prop}

\begin{proof}
    Let $\vec{f}$ be the $j-$tuple $(d_{1},...,d_{j})$. Consider the natural map
\[
\mathcal{Z}_{\vec{d}}^{i,j}\rightarrow \mathcal{Z}_{\vec{f}}^{i,j},
\]
that forgets the sheaves $E_{j+1},...,E_{n}$ of a closed point of the form \eqref{Moduli space: The one key} $\in \mathcal{Z}_{\vec{d}}^{i,j}$ and note that all fibres of this map are smooth hyperquote schemes of dimension $r(d_{n}-d_{j}-1)$. Therefore, it suffices to prove Proposition \ref{proposition: smoothness and dimension of moduli space Z} in the case $j=n.$\\\\
Let us induct on $n$ to show that $\mathcal{Z}_{\vec{d}}^{i,n}$, whenever non-empty, is smooth of dimension $r(d_{n}+1)$ for all $1\leq i\leq n$. The case $n=1$ is clear since $\mathcal{Z}_{d_{1}}^{1,1}$ is just a Hyperquot scheme.
Also, if $\vec{d}_{\downarrow}$ is the $(n-1)-$ tuple $(d_{1},...,d_{n-1})$, then by the induction hypothesis, we have that $\mathcal{Z}_{\vec{d}_{\downarrow}}^{i,n-1}$ is a smooth projective variety of dimension $r(d_{n-1}+1)$,  for all $1\leq i\leq n-1$.\\\\
Let us define some moduli spaces that we will use in the induction step and whose closed points parameterise the following diagrams of rank $r$ sub-sheaves on $V$:

\begin{equation}\tag{X}\label{Modui space: straight X}
    \begin{tikzcd}[cramped]
	{E_{n}} & {E_{n-1}} & {...\hookrightarrow E_{i+1}\,\,\,\,} & {E_{i}} & {...\hookrightarrow E_{1}} & V \\
	{} & {F_{n-1}} & {...\hookrightarrow F_{i+1}} & {F_{i}}
	\arrow[hook, from=1-1, to=1-2]
	\arrow[hook, from=1-2, to=1-3]
	\arrow["{\,\,\,\,\,\,\,\,\,\,\,\,...}", shift left=5, draw=none, from=1-2, to=2-2]
	\arrow[hook, from=1-3, to=1-4]
	\arrow[shift right=2, draw=none, from=1-3, to=2-3]
	\arrow[shift right=5, draw=none, from=1-3, to=2-3]
	\arrow[hook, from=1-4, to=1-5]
	\arrow[hook, from=1-5, to=1-6]
	\arrow["p"'{pos=0.4}, hook, from=2-2, to=1-2]
	\arrow[hook, from=2-2, to=2-3]
	\arrow["p"'{pos=0.4}, shift right=3, hook, from=2-3, to=1-3]
	\arrow[hook, from=2-3, to=2-4]
	\arrow["p"'{pos=0.4}, hook, from=2-4, to=1-4]
\end{tikzcd}
\end{equation}

\begin{equation}\tag{Z}\label{Modui space: straight Z}
\begin{tikzcd}[cramped]
	& {E_{n-1}} & {...\hookrightarrow E_{i+1}\,\,\,\,} & {E_{i}} & {...\hookrightarrow E_{1}} & V \\
	{E_{n}} & {F_{n-1}} & {...\hookrightarrow F_{i+1}} & {F_{i}}
	\arrow[hook, from=1-2, to=1-3]
	\arrow["{\,\,\,\,\,\,\,\,\,\,\,\,...}", shift left=5, draw=none, from=1-2, to=2-2]
	\arrow[hook, from=1-3, to=1-4]
	\arrow[shift right=2, draw=none, from=1-3, to=2-3]
	\arrow[shift right=5, draw=none, from=1-3, to=2-3]
	\arrow[hook, from=1-4, to=1-5]
	\arrow[hook, from=1-5, to=1-6]
	\arrow[hook, from=2-1, to=2-2]
	\arrow["p"'{pos=0.4}, hook, from=2-2, to=1-2]
	\arrow[hook, from=2-2, to=2-3]
	\arrow["p"'{pos=0.4}, shift right=3, hook, from=2-3, to=1-3]
	\arrow[hook, from=2-3, to=2-4]
	\arrow["p"'{pos=0.4}, hook, from=2-4, to=1-4]
\end{tikzcd}
\end{equation}

\begin{equation}\tag{E}\label{Modui space: straight E}
\begin{tikzcd}[cramped]
	& {E_{n-1}} & {...\hookrightarrow E_{i+1}\,\,\,\,} & {E_{i}} & {...\hookrightarrow E_{1}} & V \\
	{E_{n}} & {F_{n-1}} & {...\hookrightarrow F_{i+1}} & {F_{i}} \\
	{F_{n}}
	\arrow[hook, from=1-2, to=1-3]
	\arrow["{\,\,\,\,\,\,\,\,\,\,\,\,...}", shift left=5, draw=none, from=1-2, to=2-2]
	\arrow[hook, from=1-3, to=1-4]
	\arrow[shift right=2, draw=none, from=1-3, to=2-3]
	\arrow[shift right=5, draw=none, from=1-3, to=2-3]
	\arrow[hook, from=1-4, to=1-5]
	\arrow[hook, from=1-5, to=1-6]
	\arrow[hook, from=2-1, to=2-2]
	\arrow["p"'{pos=0.4}, hook, from=2-2, to=1-2]
	\arrow[hook, from=2-2, to=2-3]
	\arrow["p"'{pos=0.4}, shift right=3, hook, from=2-3, to=1-3]
	\arrow[hook, from=2-3, to=2-4]
	\arrow["p"'{pos=0.4}, hook, from=2-4, to=1-4]
	\arrow["p"', hook, from=3-1, to=2-1]
\end{tikzcd}
\end{equation}

\begin{equation}\tag{M}\label{Modui space: straight M}
\begin{tikzcd}[cramped]
	{E_{n}} & {E_{n-1}} & {...\hookrightarrow E_{i+1}\,\,\,\,} & {E_{i}} & {...\hookrightarrow E_{1}} & V \\
	{F_{n}} & {F_{n-1}} & {...\hookrightarrow F_{i+1}} & {F_{i}}
	\arrow[hook, from=1-1, to=1-2]
	\arrow[hook, from=1-2, to=1-3]
	\arrow["{\,\,\,\,\,\,\,\,\,\,\,\,...}", shift left=5, draw=none, from=1-2, to=2-2]
	\arrow[hook, from=1-3, to=1-4]
	\arrow[shift right=2, draw=none, from=1-3, to=2-3]
	\arrow[shift right=5, draw=none, from=1-3, to=2-3]
	\arrow[hook, from=1-4, to=1-5]
	\arrow[hook, from=1-5, to=1-6]
	\arrow["p"', hook, from=2-1, to=1-1]
	\arrow["p"'{pos=0.4}, hook, from=2-2, to=1-2]
	\arrow[hook, from=2-2, to=2-3]
	\arrow["p"'{pos=0.4}, shift right=3, hook, from=2-3, to=1-3]
	\arrow[hook, from=2-3, to=2-4]
	\arrow["p"'{pos=0.4}, hook, from=2-4, to=1-4]
\end{tikzcd}.
\end{equation}
Each of the above figures represents a moduli space that parameterises diagrams of inclusions of certain rank $r$ sub-sheaves $E_{*},F_{*}$ of $V$ on $C$. In each of the above diagrams we impose the condition $\text{dim}_{\mathbb{C}}(V/E_{*})=d_{*}$ for all $*\in\{1,...,n\}$.\\\\
Let us also recall that $\mathcal{Z}_{\vec{d}}^{i,n}$ and $\mathcal{Z}_{\vec{d}_{\downarrow}}^{i,n-1}$ are moduli spaces with the following description of their closed points:

\begin{equation}\tag{$\mathcal{Z}_{\vec{d}_{\downarrow}}^{i,n-1}$}\label{moduli space: Z downarrow}
\begin{tikzcd}[cramped]
	{E_{n-1}} & {...\hookrightarrow E_{i+1}\,\,\,\,} & {E_{i}} & {...\hookrightarrow E_{1}} & V \\
	{F_{n-1}} & {...\hookrightarrow F_{i+1}} & {F_{i}}
	\arrow[hook, from=1-1, to=1-2]
	\arrow["{\,\,\,\,\,\,\,\,\,\,\,\,...}", shift left=5, draw=none, from=1-1, to=2-1]
	\arrow[hook, from=1-2, to=1-3]
	\arrow[shift right=2, draw=none, from=1-2, to=2-2]
	\arrow[shift right=5, draw=none, from=1-2, to=2-2]
	\arrow[hook, from=1-3, to=1-4]
	\arrow[hook, from=1-4, to=1-5]
	\arrow["p"'{pos=0.4}, hook, from=2-1, to=1-1]
	\arrow[hook, from=2-1, to=2-2]
	\arrow["p"'{pos=0.4}, shift right=3, hook, from=2-2, to=1-2]
	\arrow[hook, from=2-2, to=2-3]
	\arrow["p"'{pos=0.4}, hook, from=2-3, to=1-3]
\end{tikzcd}
\end{equation}

\begin{equation}\tag{$\mathcal{Z}_{\vec{d}}^{i,n}$}\label{moduli space: Z the one key in the proof}
\begin{tikzcd}[cramped]
	{E_{n}} & {E_{n-1}} & {...\hookrightarrow E_{i+1}\,\,\,\,} & {E_{i}} & {...\hookrightarrow E_{1}} & V \\
	{F_{n}} & {F_{n-1}} & {...\hookrightarrow F_{i+1}} & {F_{i}}
	\arrow[hook, from=1-1, to=1-2]
	\arrow[hook, from=1-2, to=1-3]
	\arrow["{\,\,\,\,\,\,\,\,\,\,\,\,...}", shift left=5, draw=none, from=1-2, to=2-2]
	\arrow[hook, from=1-3, to=1-4]
	\arrow[shift right=2, draw=none, from=1-3, to=2-3]
	\arrow[shift right=5, draw=none, from=1-3, to=2-3]
	\arrow[hook, from=1-4, to=1-5]
	\arrow[hook, from=1-5, to=1-6]
	\arrow["p"', hook, from=2-1, to=1-1]
	\arrow[hook, from=2-1, to=2-2]
	\arrow["p"'{pos=0.4}, hook, from=2-2, to=1-2]
	\arrow[hook, from=2-2, to=2-3]
	\arrow["p"'{pos=0.4}, shift right=3, hook, from=2-3, to=1-3]
	\arrow[hook, from=2-3, to=2-4]
	\arrow["p"'{pos=0.4}, hook, from=2-4, to=1-4]
\end{tikzcd}.
\end{equation}
Assuming the induction hypothesis, we make note of the following facts about the geometry of the above spaces:
\begin{enumerate}
    \item The space $\mathcal{Z}_{\vec{d}_{\downarrow}}^{i,n-1}$ is smooth and projective of dimension $r(d_{n-1}+1)$: This directly follows from the induction hypothesis.
    \item The space \ref{Modui space: straight Z} is a smooth projective variety of dimension $rd_{n}$: Indeed, the fibres of the natural forgetful map from $\ref{Modui space: straight Z}$ to $\mathcal{Z}_{\vec{d}_{\downarrow}}^{i,n-1}$ are smooth Quot schemes with dimensions $r(d_{n}-d_{n-1}-1)$.
    \item The space \ref{Modui space: straight X} is a smooth projective variety of dimension $r(d_{n}+1)$: This is because the fibres of the natural forgetful map from $\ref{Modui space: straight X}$ to $\mathcal{Z}_{\vec{d}_{\downarrow}}^{i,n-1}$ are smooth quot schemes with dimensions $r(d_{n}-d_{n-1})$.
    \item The space \ref{Modui space: straight E} is a smooth projective variety of dimension $r(d_{n}+1)-1$: The natural forgetful map from $\ref{Modui space: straight E}$ to \ref{Modui space: straight Z} is a $\mathbb{P}^{r-1}-$bundle, using the same argument as for Proposition \ref{pi-times piC as a projective bundle} in the case $k=n$.
    \item The space \ref{Modui space: straight M} is a smooth projective variety of dimension $r(d_{n}+2)-1$: The natural forgetful map from $\ref{Modui space: straight M}$ to \ref{Modui space: straight X} is a $\mathbb{P}^{r-1}-$bundle.
\end{enumerate}
\noindent We will now identify the natural forgetful map
\[
\kappa:\mathcal{Z}_{\vec{d}}^{i,n}\rightarrow X
\]
with the blow-up of $X$ along the sub-scheme $Z$, with exceptional divisor $E$. This would immediately prove the proposition.\\\\
First note from point (3) above that $\kappa^{-1}(Z)=E$ is a $\mathbb{P}^{r-1}$ bundle over $Z$. Let us now observe that $\kappa$ is an isomorphism over $X\backslash Z$. A point in $X\backslash Z$ corresponds to a diagram of the form \eqref{Modui space: straight X} such that $E_{n}$ is not contained in $F_{n-1}$. Then since 
\[
\frac{E_{n}}{E_{n}\cap F_{n-1}}=\frac{F_{n-1}+E_{n}}{F_{n-1}}=\frac{E_{n-1}}{F_{n-1}}\cong \mathbb{C}_{p}
\] 
and the fact that any sub-sheaf of $E_{n}$ and $F_{n-1}$ is a sub-sheaf of $E_{n}\cap F_{n-1}$, by defining
\[
F_{n}:=E_{n}\cap F_{n-1}
\]
we obtain an inverse for $\kappa_{|_{\mathcal{Z}_{\vec{d}}^{i,n}\backslash E}}$. Hence, if we prove the following statements, the universal property of the blow-up will yield an isomorphism between $\mathcal{Z}_{\vec{d}}^{i,n}$ and $Bl_{Z}(X)$:
\begin{enumerate}
    \item The space $\mathcal{Z}_{\vec{d}}^{i,n}$ is irreducible and reduced.
    \item The codimension 1 sub-variety $E\subseteq \mathcal{Z}_{\vec{d}}^{i,n}$, is cut-out by a section of a line bundle on $\mathcal{Z}_{\vec{d}}^{i,n}$ and is hence a Cartier divisor. 
\end{enumerate}
It is not hard to see that $E$ is cut-out by a section of a line bundle on $\mathcal{Z}_{\vec{d}}^{i,n}$; The inclusions $\mathcal{E}_{n}\subseteq\mathcal{E}_{n-1}$ and $\mathcal{F}_{n}\subseteq\mathcal{F}_{n-1}$ of universal vector bundles on $\mathcal{Z}_{\vec{d}}^{i,n}\times C$ (cf. diagram \eqref{Universal diagram on the one key}), induce a morphism of line bundles
\[
s:\mathscr{L}_{n}\rightarrow \mathscr{L}_{n-1}
\]
on $\mathcal{Z}_{\vec{d}}^{i,n}$ which vanishes at a closed point of the form \eqref{moduli space: Z the one key in the proof} if and only if $E_{n}\subseteq F_{n-1}$.\\\\
So it remains to prove (1), that $\mathcal{Z}_{\vec{d}}^{i,n}$ is irreducible and reduced. Since $\kappa_{|_{\mathcal{Z}_{\vec{d}}^{i,n}\backslash E}}$ is an isomorphism and $\kappa_{|_{E}}$ is a $\mathbb{P}^{r-1}-$bundle over $Z$, either $\mathcal{Z}_{\vec{d}}^{i,n}$ is irreducible of dimension $r(d_{n}+1)$ or has two components of dimensions $r(d_n+1)$ and $r(d_{n}+1)-1$. We will realise $\mathcal{Z}_{\vec{d}}^{i,n}$ as the zero locus of a section of a rank $r-1$ vector bundle on $M$, which has dimension $r(d_{n}+2)-1$. This will immediately yield that $\mathcal{Z}_{\vec{d}}^{i,n}$ is irreducible of dimension $r(d_{n}+1)$ and l.c.i.. This will also readily imply reducedness, indeed being reduced is equivalent to the properties $R_{0}$ and $S_{1}$ . The property $S_{1}$ holds for $\mathcal{Z}_{\vec{d}}^{i,n}$ since it is l.c.i. and the property $R_{0}$ holds for $\mathcal{Z}_{\vec{d}}^{i,n}$ since it is reduced on the complement of $E$, where it is isomorphic to the smooth variety $X\backslash Z$.\\\\
So, to prove (1) and the proposition, it remains to cut $\mathcal{Z}_{\vec{d}}^{i,n}$ out by a section of a rank $r-1$ vector bundle on $M$. To this end, we consider the map $p^{M}:M\rightarrow C$, which assigns to a closed point of the form \eqref{Modui space: straight M}, the point $p\in C$ and let
\[
\Delta_{M}:M\hookrightarrow M\times C
\]
be the graph of $p^{M}$.
For $l=i,.. .,n$, consider the universal exact sequence on $M\times C$:
\[
0\rightarrow \mathcal{F}_{l}\rightarrow\mathcal{E}_{l}\rightarrow\mathscr{F}_{l}\rightarrow 0,
\]
whose restriction to a closed point of the form \eqref{Modui space: straight M} corresponds to the inclusion ${F}_{l}\subseteq {E}_{l}$. As before, let us define line bundles $\mathscr{L}_{l}$ and rank $r-1$ bundles $\mathcal{G}_{l}$ on $M$: 
\[
\mathscr{L}_{l}:=\mathscr{F}_{l_{|_{\Delta_{M}}}}\text{ and }\mathcal{G}_{l}:=\text{Ker}\,(\mathcal{E}_{{l}_{|\Delta_{M}}}\rightarrow\mathscr{L}_{l}).
\]
Then one observes that the section of the rank $r-1$ vector bundle $\mathcal{G}_{n}^{\vee}\otimes\mathscr{L}_{n-1}$ given by the composition:
\[
\mathcal{G}_{n}\rightarrow \mathcal{E}_{{n}_{|\Delta_{M}}}\rightarrow \mathcal{E}_{{n-1}_{|\Delta_{M}}} \rightarrow\mathscr{L}_{n-1}
\]
cuts out $\mathcal{Z}_{\vec{d}}^{i,n}$. This completes the proof.
 \end{proof}
\noindent Let us end the subsection by recording a useful equality in $H^{*}(\mathcal{Z}_{\vec{d}}^{k,n})$. Let us define 
 \[
 \lambda_{k}:=c_{1}(\mathscr{L}_{k})\in H^{2}(\mathcal{Z}_{\vec{d}}^{k,n}),
 \]
  then it follows from the definition of $\mathcal{G}^{k,n}_{\vec{d}}$ in \eqref{equation: definition of G for mod space Z}, that:
 \begin{equation}\label{equality in cohomology:a=em}
 c(\mathcal{G}^{k,n}_{\vec{d}},z)=\frac{c({\mathcal{E}_{k}}_{|_{\Delta}},z)}{z-\lambda_{k}}.
 \end{equation}
\subsection{The skew-nested Quot scheme $\mathcal{Y}_{\vec{d}}^{i-1,i}$}\label{subsection:moduli space Y}
As before, let $\vec{d}=(d_1,...,d_{n})$ be an $n-$tuple of non-negative integers, for $k\in \{1,...,n\}$ let $\vec{d}'_{k}=\\(d_1,...,d_{k-1},d_{k}+1,d_{k+1},...,d_{n})$ and for $i\in \{2,...,n\}$, let $\vec{d}_{i-1,i}$ be the $n-$tuple $(d_1,...,d_{i-1}+1,d_{i}+1,...,d_{n})$. Furthermore, let $\vec{d}_{i-1,i}^{\bullet}$ be the $(n+2)-$tuple $(d_1,...,d_{i-1},d_{i-1}+1,d_{i},d_{i}+1,...d_{n})$. We define the space $\mathcal{Y}_{\vec{d}}^{i-1,i}$ to be the fine moduli space with the following description of it's closed points: 
\begin{equation}\label{Moduli space: e_i-1e_i}\tag{$\mathcal{Y}_{\vec{d}}^{i-1,i}$}
\begin{tikzcd}[cramped]
	&& {E_i} & {E_{i-1}} & {E_{i-2}\hookrightarrow...\hookrightarrow E_{1}} & V \\
	{E_{n}\hookrightarrow...\hookrightarrow E_{i+2}} & {E_{i+1}} & {F_{i}} & {F_{i-1}}
	\arrow[hook, from=1-3, to=1-4]
	\arrow[hook, from=1-4, to=1-5]
	\arrow[hook, from=1-5, to=1-6]
	\arrow[hook, from=2-1, to=2-2]
	\arrow[hook, from=2-2, to=2-3]
	\arrow["q"', hook, from=2-3, to=1-3]
	\arrow[hook, from=2-3, to=2-4]
	\arrow["p"', hook, from=2-4, to=1-4]
\end{tikzcd}
\end{equation}
\noindent That is, $\mathcal{Y}_{\vec{d}}^{i-1,i}$ parameterises pairs of length $n$ flags of rank $r$ locally free sheaves on $C$; 
\[
(E_n\subseteq...\subseteq E_1\subseteq V)\in F^{n}\text{Quot}_{\vec{d}}\,(V) \text{ and } 
(F_n\subseteq...\subseteq F_1\subseteq V)\in F^{n}\text{Quot}_{\vec{d}_{i-1,i}}\,(V),
\]
such that $E_t=F_t$ for $t=1,...,i-2,i+1,...,n$ and $E_j/F_j$ are length one sheaves supported on some arbitrary points $p\in C$ and $q\in C$ for $j=i-1$ and $j=i$ respectively.
For the proceeding discussion, we will further ask that $\mathcal{Y}_{\vec{d}}^{i-1,i}$ is non-empty, otherwise the statements are vacuous. This is equivalent to $\vec{d}$ being non-decreasing and $d_{i+1}-d_{i}\geq 1.$ In fact, we will also assume that $d_{i}\neq d_{i-1}$, since in this case $\mathcal{Y}_{\vec{d}}^{i-1,i}$ is just a Hyperquot scheme.\\\\
Note that $\mathcal{Y}_{\vec{d}}^{i-1,i}$ carries tautological line bundles $$\mathscr{L}_{i}\text{ and }\mathscr{L}_{i-1}$$ whose fibres over a closed point, as in Diagram \ref{Moduli space: e_i-1e_i}, are the length one quotients $E_{i}/F_{i}$ and $E_{i-1}/F_{i-1}$ respectively.
Moreover, if $\mathcal{E}_{i-1},\mathcal{E}_{i},\mathcal{F}_{i-1}$ and $\mathcal{F}_{i}$ are universal sheaves corresponding to the sheaves $E_{i-1},E_{i},F_{i-1}$ and $F_{i}$ respectively, then the inclusions $\mathcal{E}_{i}\hookrightarrow \mathcal{E}_{i-1}$ and $\mathcal{F}_{i}\hookrightarrow \mathcal{F}_{i-1}$ induce a map of line bundles 
\begin{equation}\label{equation:section of line bundle moduli space Y}
\mathscr{L}_{i}\xrightarrow{s}\mathscr{L}_{i-1}.
\end{equation}
There are also natural maps
\begin{equation}\label{q_C}
    q_{C}^{i},\,\,q_{C}^{i-1}:\mathcal{Y}_{\vec{d}}^{i-1,i}\rightarrow C \times C,
\end{equation}
which assign the closed point \ref{Moduli space: e_i-1e_i} to Supp $E_{i}/F_{i} =q$ and Supp $E_{i-1}/F_{i-1} =p$ respectively.\\\\
Let $i$ denote the natural embedding of $\mathcal{Y}_{\vec{d}}^{i-1,i}$ into $F^{n}\text{Quot}_{\vec{d}}\,(V) \times 
F^{n}\text{Quot}_{\vec{d}_{i-1,i}}\,(V)$. Using the morphisms \ref{q_C}, it would also be convenient to consider the following closed embedding:
\begin{equation}\label{embedding k for moduli space Y}
    k:=i\times q^{i}_{C}\times q_{C}^{i-1}:\mathcal{Y}_{\vec{d}}^{i-1,i}\hookrightarrow F^{n}\text{Quot}_{\vec{d}}\,(V) \times 
F^{n}\text{Quot}_{\vec{d}_{i-1,i}}\,(V)\times C\times C.
\end{equation}

\begin{prop}\label{geometry of Y}
    The moduli space $\mathcal{Y}_{\vec{d}}^{i-1,i}$ has two equidimensional irreducible components, namely the smooth varieties $\mathcal{Z}_{\vec{d}}^{i-1,i}$ and $F^{n+2}\text{Quot}_{\vec{d}_{i-1,i}^{\bullet}}(V)$. In particular, the space $\mathcal{Y}_{\vec{d}}^{i-1,i}$ has dimension $rd_{n }$ if $i\neq n$ and dimension $r(d_{n}+1)$ if $i=n$.
\end{prop}
\begin{proof}
    Note that the section
    \[
    s:\mathscr{L}_{i}\rightarrow \mathscr{L}_{i-1}
    \]
    from equation \eqref{equation:section of line bundle moduli space Y}
    cuts out the sub-scheme of $\mathcal{Y}_{\vec{d}}^{i-1,i}$, which parametrises diagrams of the form \eqref{Moduli space: e_i-1e_i} such that $E_{i}\subseteq F_{i-1}$. One observes that this sub-scheme is precisely $F^{n+2}\text{Quot}_{\vec{d}_{i-1,i}^{\bullet}}(V)$ and contains the open locus $$\left(q^{i}_{C}\times q_{C}^{i-1}\right)^{-1}((C\times C)\backslash \Delta),$$ where $\Delta$ is the diagonal in $C\times C$. On the other hand, note that the closed sub-scheme $\left(q^{i}_{C}\times q_{C}^{i-1}\right)^{-1}(\Delta)$ is precisely $\mathcal{Z}_{\vec{d}}^{i-1,i}$. This shows that $\mathcal{Y}_{\vec{d}}^{i-1,i}$ has two irreducible componenets. The dimension count follows from Proposition \ref{proposition: smoothness and dimension of moduli space Z}.  
\end{proof}

\begin{prop}\label{lci ness of Y}
    The space $\mathcal{Y}_{\vec{d}}^{i-1,i}$ is l.c.i and reduced.
\end{prop}

\begin{proof}
We will realise $\mathcal{Y}_{\vec{d}}^{i-1,i}$ as the intersection of smooth sub-varieties $Z_{1}$ and $Z_{2}$ of $$Z:=F^{n}\text{Quot}_{\vec{d}}\,(V)\times F^{n}\text{Quot}_{\vec{d}_{i}'}\,(V) \times F^{n}\text{Quot}_{\vec{d}_{i-1,i}}\,(V)$$ and such that $$\text{codim}_{Z}(Z_{1})+\text{codim}_{Z}(Z_{2})=\text{codim}_{Z}(Z_{1}\cap Z_{2}).$$This would immediately imply that $\mathcal{Y}_{\vec{d}}^{i-1,i}$ is l.c.i.. To this end, let us denote closed points of $Z=F^{n}\text{Quot}_{\vec{d}}\,(V)\times F^{n}\text{Quot}_{\vec{d}_{i}'}\,(V) \times F^{n}\text{Quot}_{\vec{d}_{i-1,i}}\,(V)$ by triples of flags of rank $r$ sub-sheaves of $V$ of the form $$\left\{(A_{n} \subseteq ...\subseteq A_{1}\subseteq V),((B_{n} \subseteq ...\subseteq B_{1}\subseteq V)),(C_{n} \subseteq ...\subseteq C_{1}\subseteq V)\right\}.$$ Let $Z_{1}$ be the sub-scheme of $Z$ given by the condition $$A_k=B_k\text{ if }k\neq i\text{ and }B_{i}\subseteq A_{i} \text{ with dim}_{\mathbb{C}}(A_{i}/B_{i})=1$$ and let $Z_{2}$ be the sub-scheme of $Z$ given by the condition $$B_k=C_k\text{ if }k\neq i-1\text{ and }C_{i-1}\subseteq B_{i-1} \text{ with dim}_{\mathbb{C}}(B_{i-1}/C_{i-1})=1.$$
Then one observes that both $Z_{1}$ and $Z_{2}$ are products of Hyperquot schemes and that the scheme theoretic intersection $Z_{1}\cap Z_{2}\cong \mathcal{Y}_{\vec{d}}^{i-1,i}$.\\\\
Reducedness is equivalent to properties $R_{0}$ and $S_{1}$ holding, since $\mathcal{Y}_{\vec{d}}^{i-1,i}$ is l.c.i., it is enough to show that a dense open of $\mathcal{Y}_{\vec{d}}^{i-1,i}$ is reduced. This follows from Proposition \ref{geometry of Y} because $\mathcal{Z}_{\vec{d}}^{i-1,i}$ and $F^{n+2}\text{Quot}_{\vec{d}_{i-1,i}^{\bullet}}(V)$ are smooth. 
\end{proof}

\noindent The line bundles $\mathscr{L}_{i}$ and $\mathscr{L}_{i-1}$ also define cohomology classes on $\mathcal{Y}_{\vec{d}}^{i-1,i}$:
\[
\lambda_{i}:=c_{1}(\mathscr{L}_{i})\text{ and }\lambda_{i-1}:=c_{1}(\mathscr{L}_{i-1}).
\]
By a harmless abuse of notation, we will also refer to the restrictions of the line bundles $\mathscr{L}_{i}$ and $\mathscr{L}_{i-1}$ to the irreducible components $\mathcal{Z}_{\vec{d}}^{i-1,i}$ and $F^{n+2}\text{Quot}_{\vec{d}_{i-1,i}^{\bullet}}(V)$ by the same notations.

\begin{lem}\label{key lemma to prove the e_i e_i-1 commutation}
     Let $\delta\in H^{2}(\mathcal{Y}_{\vec{d}}^{i-1,i})$ be the pull-back of the diagonal class in $C\times C$ via the map $q^{i}_{C}\times q_{C}^{i-1}$. Then we have the following equality of cohomology classes in $H^{*}(\mathcal{Y}_{\vec{d}}^{i-1,i})$:
     \[
     (\lambda_{i-1}-\lambda_{i})[\mathcal{Z}_{\vec{d}}^{i-1,i}]=\delta\cdot[F^{n+2}\text{Quot}_{\vec{d}_{i-1,i}^{\bullet}}].
     \]
     In the above equality, the terms in square brackets [] denote the fundamental classes of the enclosed sub-varieties.
\end{lem}
\begin{proof}
    We claim that both the L.H.S. and the R.H.S. of the above equality of classes in $H^{*}(\mathcal{Y}_{\vec{d}}^{i-1,i})$ are equal to the fundamental class of the sub-variety $J$ of $\mathcal{Y}_{\vec{d}}^{i-1,i}$, which parameterises flags of rank $r$ sub-sheaves of $V$ the with following description:

\begin{equation}\label{moduli space J in proof of Y coho equality}\tag{J}
\begin{tikzcd}[cramped]
	&&& {E_{i-1}} & {E_{i-2}\hookrightarrow...\hookrightarrow E_{1}} & V \\
	{E_{n}\hookrightarrow...\hookrightarrow E_{i+1}} & {F_{i}} & {E_{i}} & {F_{i-1}}
	\arrow[hook, from=1-4, to=1-5]
	\arrow[hook, from=1-5, to=1-6]
	\arrow[hook, from=2-1, to=2-2]
	\arrow["p"', hook, from=2-2, to=2-3]
	\arrow[hook, from=2-3, to=2-4]
	\arrow["p"', hook, from=2-4, to=1-4]
\end{tikzcd}.
\end{equation}
In the above diagram, we impose that $\text{dim}_{\mathbb{C}}(V/E_{*})=d_{*}$ for all $*\in\{1,...,n\}$.\\\\
The sub-variety $J$ is cut-out by the restriction of the section $s$ in equation \eqref{equation:section of line bundle moduli space Y}, to the component $\mathcal{Z}_{\vec{d}}^{i-1,i}$ of $\mathcal{Y}_{\vec{d}}^{i-1,i}$. Therefore $ (\lambda_{i-1}-\lambda_{i})[\mathcal{Z}_{\vec{d}}^{i-1,i}]$ is equal to the fundamental class of $J$. On the other hand, $J$ is the scheme-theoretic pull-back $$J=(q^{i}_{C}\times q_{C}^{i-1})_{|_{F^{n+2}\text{Quot}_{\vec{d}_{i-1,i}^{\bullet}}(V)}}^{-1}(\Delta).$$ The projection formula then implies that $\delta\cdot[{F^{n+2}\text{Quot}_{\vec{d}_{i-1,i}^{\bullet}}(V)}]$ is also equal to the fundamental class of $J$ in $H^{*}(\mathcal{Y}_{\vec{d}}^{i-1,i})$, which finishes the proof.

\end{proof}

\subsection{The skew-nested Quot schemes $\mathcal{U}_{\vec{d}}^{i,j}$, $\mathcal{W}_{\vec{d}}^{i,j}$ and $\mathcal{X}_{\vec{d}}^{i,j}$ }\label{subsection: moduli spaces U,X,W}
 Let $\vec{d}=(d_1,...,d_{n})$ be an $n-$tuple of non-negative integers. For $1\leq i\leq j\leq n$, define $n-$tuples $\vec{d'}_{i,j}:=(d_1,...,d_{i-1},d_{i}+1,...,d_{j}+1,d_{j+1},...,d_{n})$ and let 
 $\vec{d''}_{i,j}:=\\
 (d_1,...,d_{i-1},d_{i}+1,...,d_{j-1}+1,d_{j}+2,d_{j+1},...,d_{n})$. We define the spaces $\mathcal{U}_{\vec{d}}^{i,j}$, $\mathcal{W}_{\vec{d}}^{i,j}$ and $\mathcal{X}_{\vec{d}}^{i,j}$ to be the fine moduli spaces parameterising the following diagrams of inclusions of rank $r$ sub-sheaves of $V$:

 \begin{equation}\tag{$\mathcal{U}_{\vec{d}}^{i,j}$}\label{The moduli space: U}
\begin{tikzcd}[cramped]
	&& {E_{j}} && {...\hookrightarrow E_{i+1}\,\,\,\,} & {E_{i}} & {...\hookrightarrow E_{1}} & V \\
	& {F^{'}_{j}} && {F_{j}} & {...\hookrightarrow F_{i+1}} & {F_{i}} \\
	{E_{n}\hookrightarrow...\hookrightarrow E_{j+1}} & {} & {H_{j}}
	\arrow[hook, from=1-3, to=1-5]
	\arrow[hook, from=1-5, to=1-6]
	\arrow[shift right=2, draw=none, from=1-5, to=2-5]
	\arrow[shift right=5, draw=none, from=1-5, to=2-5]
	\arrow[hook, from=1-6, to=1-7]
	\arrow[hook, from=1-7, to=1-8]
	\arrow["p"', hook, from=2-2, to=1-3]
	\arrow["q"', hook, from=2-4, to=1-3]
	\arrow[hook, from=2-4, to=2-5]
	\arrow["q"'{pos=0.4}, shift right=3, hook, from=2-5, to=1-5]
	\arrow[hook, from=2-5, to=2-6]
	\arrow["q"'{pos=0.4}, hook, from=2-6, to=1-6]
	\arrow[hook, from=3-1, to=3-3]
	\arrow["q"', hook, from=3-3, to=2-2]
	\arrow["p"', hook, from=3-3, to=2-4]
\end{tikzcd}
 \end{equation}

\begin{equation}\tag{$\mathcal{W}_{\vec{d}}^{i,j}$}\label{The moduli space: W}
\begin{tikzcd}[cramped]
	& {E_{j}} & {E_{j-1}} & {...\hookrightarrow E_{i+1}} & {E_{i}} & {...\hookrightarrow E_{1}} & V \\
	& {F^{'}_{j}} && {} \\
	{E_{n}\hookrightarrow...\hookrightarrow E_{j+1}} & {H_{j}} & {F_{j-1}} & {...\hookrightarrow F_{i+1}} & {F_{i}}
	\arrow[hook, from=1-2, to=1-3]
	\arrow[hook, from=1-4, to=1-5]
	\arrow[shift right=2, draw=none, from=1-4, to=2-4]
	\arrow[hook, from=1-5, to=1-6]
	\arrow[hook, from=1-6, to=1-7]
	\arrow["p"', hook, from=2-2, to=1-2]
	\arrow[hook, from=3-1, to=3-2]
	\arrow["q"', hook, from=3-2, to=2-2]
	\arrow[hook, from=3-2, to=3-3]
	\arrow["{q\,\,\,\,\,\,\,\,\,\,\,\,\,\,\,\,\,...}"', hook, from=3-3, to=1-3]
	\arrow["q"', shift right=3, hook, from=3-4, to=1-4]
	\arrow[hook, from=3-4, to=3-5]
	\arrow["q"', hook, from=3-5, to=1-5]
\end{tikzcd}
 \end{equation}

\begin{equation}\tag{$\mathcal{X}_{\vec{d}}^{i,j}$}\label{The moduli space: X}
\begin{tikzcd}[cramped]
	& {E_{j}} & {...\hookrightarrow E_{i+1}\,\,\,\,} & {E_{i}} & {...\hookrightarrow E_{1}} & V \\
	& {F_{j}} & {...\hookrightarrow F_{i+1}} & {F_{i}} \\
	{E_{n}\hookrightarrow...\hookrightarrow E_{j+1}} & {H_{j}}
	\arrow[hook, from=1-2, to=1-3]
	\arrow["{\,\,\,\,\,\,\,\,\,\,\,\,...}", shift left=5, draw=none, from=1-2, to=2-2]
	\arrow[hook, from=1-3, to=1-4]
	\arrow[shift right=2, draw=none, from=1-3, to=2-3]
	\arrow[shift right=5, draw=none, from=1-3, to=2-3]
	\arrow[hook, from=1-4, to=1-5]
	\arrow[hook, from=1-5, to=1-6]
	\arrow["q"'{pos=0.4}, hook, from=2-2, to=1-2]
	\arrow[hook, from=2-2, to=2-3]
	\arrow["q"'{pos=0.4}, shift right=3, hook, from=2-3, to=1-3]
	\arrow[hook, from=2-3, to=2-4]
	\arrow["q"'{pos=0.4}, hook, from=2-4, to=1-4]
	\arrow[hook, from=3-1, to=3-2]
	\arrow["p"', hook, from=3-2, to=2-2]
\end{tikzcd}
 \end{equation}
\newline
The above are diagrams of inclusions of certain rank $r$ sub-sheaves $E_{*},F_{*},F^{'}_{j},H_{*}$ of $V$ on $C$, for $*\in\{1,...,n\}$. In each of the above diagrams we impose that $\text{dim}_{\mathbb{C}}(V/E_{*})=d_{*}$.\\\\
For example, the space $\mathcal{X}_{\vec{d}}^{i,j}$ can be thought of as the moduli space parameterising triples of length $n$ flags of rank $r$ sub-sheaves of $V$:
\[
(E_n\subseteq...\subseteq E_1\subseteq V)\in F^{n}\text{Quot}_{\vec{d}}\,(V),\, 
(F_n\subseteq...\subseteq F_1\subseteq V)\in F^{n}\text{Quot}_{\vec{d'}_{i,j}}\,(V)
\]
\[
\text{ and }(H_n\subseteq...\subseteq H_1\subseteq V)\in F^{n}\text{Quot}_{\vec{d''}_{i,j}}\,(V), 
\]
such that $E_t=F_t$ for $t=1,...,i-1,j+1,...,n$ and $F_{s}\hookrightarrow E_{s}$, where $E_s/F_s$ are length one sheaves supported on the same, but arbitrary, point $q\in C$ for $s=i,i+1,...,j$. And $H_{t}=F_{t}$ for $t=1,...,j-1,j+1,...,n$ and $H_{j}\hookrightarrow F_{j}$, where $F_{j}/H_{j}$ is a length one sheaf on $C$, supported on some arbitrary point $p\in C$. The space $\mathcal{W}_{\vec{d}}^{i,j}$ has a similar description and hence we obtain closed embeddings:

\begin{equation}\label{natural embedding of moduli space X}
    i_{X}:\mathcal{X}_{\vec{d}}^{i,j}\hookrightarrow F^{n}\text{Quot}_{\vec{d}}\,(V)\times F^{n}\text{Quot}_{\vec{d'}_{i,j}}\,(V)\times F^{n}\text{Quot}_{\vec{d''}_{i,j}}\,(V).
\end{equation}
\begin{equation}\label{natural embedding of moduli space W}
    i_{W}:\mathcal{W}_{\vec{d}}^{i,j}\hookrightarrow F^{n}\text{Quot}_{\vec{d}}\,(V)\times F^{n}\text{Quot}_{\vec{d'}_{i,j}}\,(V)\times F^{n}\text{Quot}_{\vec{d''}_{i,j}}\,(V).
\newline
\end{equation} 
It would also be useful to consider the composition of $i_{X}$ and $i_{W}$ with the natural projection: $$F^{n}\text{Quot}_{\vec{d}}\,(V)\times F^{n}\text{Quot}_{\vec{d'}_{i,j}}\,(V)\times F^{n}\text{Quot}_{\vec{d''}_{i,j}}\,(V)\rightarrow F^{n}\text{Quot}_{\vec{d}}\,(V)\times  F^{n}\text{Quot}_{\vec{d''}_{i,j}}\,(V).$$
Hence, we obtain maps:
\begin{equation}\label{E,H forgetful map for moduli space X}
    k_{X}:\mathcal{X}_{\vec{d}}^{i,j}\rightarrow F^{n}\text{Quot}_{\vec{d}}\,(V)\times F^{n}\text{Quot}_{\vec{d''}_{i,j}}\,(V).
\end{equation}
\begin{equation}\label{E,H forgetful map for moduli space W}
    k_{W}:\mathcal{W}_{\vec{d}}^{i,j}\rightarrow F^{n}\text{Quot}_{\vec{d}}\,(V)\times F^{n}\text{Quot}_{\vec{d''}_{i,j}}\,(V).
    \newline
\end{equation}

\noindent Note that the moduli spaces $\mathcal{X}_{\vec{d}}^{i,j}$,$\mathcal{W}_{\vec{d}}^{i,j}$ and $\mathcal{U}_{\vec{d}}^{i,j}$ are non-empty if and only if the $n-$tuples $\vec{d}$ and $\vec{d''}_{i,j}$ are non-decreasing. We assume this non-emptiness condition for the rest of the section.
\subsection{Geometry of $\mathcal{U}_{\vec{d}}^{i,j}$, $\mathcal{W}_{\vec{d}}^{i,j}$ and $\mathcal{X}_{\vec{d}}^{i,j}$}
In the next few propositions, we will address the geometric properties of the spaces $\mathcal{U}_{\vec{d}}^{i,j}$, $\mathcal{W}_{\vec{d}}^{i,j}$ and $\mathcal{X}_{\vec{d}}^{i,j}$ in question. Namely we will discuss their dimension counts, irreducibility and smoothness.\\\\
First, let us note that there exist maps 
\begin{equation}\label{birational maps : eta and theta}
    \eta_{\vec{d}}^{i,j}:\mathcal{U}_{\vec{d}}^{i,j}\rightarrow\mathcal{X}_{\vec{d}}^{i,j} \text{ and } \theta_{\vec{d}}^{i,j}:\mathcal{U}_{\vec{d}}^{i,j}\rightarrow\mathcal{W}_{\vec{d}}^{i,j},
\end{equation}
\noindent which respectively forget the sheaves $F^{'}_{j}$ and $F_{j}$ (cf. Diagrams (\ref{The moduli space: U}),(\ref{The moduli space: W}) and (\ref{The moduli space: X}).)

\begin{prop}\label{Geometry of moduli space $U$}
    The moduli space $\mathcal{U}_{\vec{d}}^{i,j}$ is a smooth projective variety of dimension $rd_{n}$ if $j\neq n$ and of dimension $r(d_{n}+2)$ if $j=n$.
\end{prop}

\begin{proof}
    Note that the space $\mathcal{U}_{\vec{d}}^{i,j}$ is just the space $\mathcal{Z}_{\vec{f}}^{i,j+1}$, for the $(n+1)-$tuple $\vec{f}:=(d_{1},...,d_{j},d_{j}+1,d_{j+1},...,d_{n})$. Consequently, this proposition follows from Proposition \ref{proposition: smoothness and dimension of moduli space Z}. 
\end{proof}

\begin{prop}\label{Geometry of moduli space $X$}
    The moduli space $\mathcal{X}_{\vec{d}}^{i,j}$ is a smooth projective variety of dimension $rd_{n}$ if $j\neq n$ and of dimension $r(d_{n}+2)$ if $j=n$.
\end{prop}

\begin{proof}
    Similar to the proof of Proposition \ref{Geometry of moduli space $U$}, one observes that the moduli space \ref{The moduli space: X} is an instance of the moduli space $\mathcal{Z}_{\vec{f}}^{i,j}$ for the $(n+1)-$tuple $\vec{f}:=(d_{1},...,d_{j},d_{j}+2,d_{j+1},...,d_{n})$. Proposition \ref{proposition: smoothness and dimension of moduli space Z} then gives the desired result. 
\end{proof}

\begin{prop}\label{Geometry of moduli space $W$}
    The moduli space $\mathcal{W}_{\vec{d}}^{i,j}$ is a reduced, irreducible, l.c.i. projective variety of dimension $rd_{n}$ if $j\neq n$ and of dimension $r(d_{n}+2)$ if $j=n$.
\end{prop}
\begin{proof}
    We claim that the map ${\theta}_{\vec{d}}^{i,j}:\mathcal{U}_{\vec{d}}^{i,j}\rightarrow\mathcal{W}_{\vec{d}}^{i,j}$ is surjective at the level of closed points and an isomorphism over an open locus of $\mathcal{W}_{\vec{d}}^{i,j}$. This would imply that $\mathcal{W}_{\vec{d}}^{i,j}$ is irreducible and with dimension equal to that of $\mathcal{U}_{\vec{d}}^{i,j}$. To this end, let $I$ be the closed sub-scheme of $\mathcal{W}_{\vec{d}}^{i,j}$ given by the condition $E_{j}\subseteq F_{j-1}$ (c.f diagram \eqref{The moduli space: W}). It is not difficult to see that points in $I$ have pre-images in $\mathcal{U}_{\vec{d}}^{i,j}$, indeed pre-images of a point in $I$ correspond to length one quotients:
    \[
    \frac{E_{j}}{H_{j}}\rightarrow\mathbb{C}_{q}\rightarrow 0.
    \]
On the other hand, a point in $\mathcal{W}_{\vec{d}}^{i,j}\backslash I$ has a unique pre-image in $\mathcal{U}_{\vec{d}}^{i,j}$, given by $F_{j}':=E_{j}\cap F_{j-1}$. In particular, the morphism ${\theta}_{\vec{d}}^{i,j}$ is an isomorphism over this locus since this inverse can be constructed in families in the same manner. Therefore we have shown that $\mathcal{W}_{\vec{d}}^{i,j}$ is irreducible and has the same dimension as that of $\mathcal{U}_{\vec{d}}^{i,j}$.\\\\
Let us now show that $\mathcal{W}_{\vec{d}}^{i,j}$ is l.c.i.. The discussion in the above paragraph shows that $\mathcal{U}_{\vec{d}}^{i,j}$ is reduced on a dense open, so we will also obtain the reducedness of $\mathcal{W}_{\vec{d}}^{i,j}$ by the $R_{0}+S_{1}$ criterion. To see that $\mathcal{W}_{\vec{d}}^{i,j}$ is l.c.i., we will realise it as the intersection of smooth sub-varieties $Z_{1}$ and $Z_{2}$ of $$Z:=F^{n}\text{Quot}_{\vec{d}}\,(V)\times F^{n}\text{Quot}_{\vec{d'}_{i,j}}(V)\times F^{n}\text{Quot}_{\vec{d''}_{i,j}}(V)$$ and such that $$\text{codim}_{Z}(Z_{1})+\text{codim}_{Z}(Z_{2})=\text{codim}_{Z}(Z_{1}\cap Z_{2}).$$ Let us denote closed points of $Z=F^{n}\text{Quot}_{\vec{d}}\,(V)\times F^{n}\text{Quot}_{\vec{d'}_{i,j}}(V)\times F^{n}\text{Quot}_{\vec{d''}_{i,j}}(V)$ by triples of flags of rank $r$ sub-sheaves of $V$ of the form $$\left\{(A_{n} \subseteq ...\subseteq A_{1}\subseteq V),((B_{n} \subseteq ...\subseteq B_{1}\subseteq V)),(C_{n} \subseteq ...\subseteq C_{1}\subseteq V)\right\}.$$ Let $Z_{1}$ be the sub-scheme of $Z$ given by the condition $$A_k=B_k\text{ if }k\neq j\text{ and }B_{j}\subseteq A_{j}\text{ with dim}_{\mathbb{C}}(A_{j}/B_{j})=1$$ and let $Z_{2}$ be the sub-scheme of $Z$ given by the condition $$B_k=C_k\text{ if }k\notin \left\{i,i+1,...,j\right\}\text{ and } C_{k}\subseteq B_{k}\text{ with }B_{k}/C_{k}\cong\mathbb{C}_{q} \text{ otherwise}.$$ Above, the point $q\in C$ is arbitrary but same for all $k\in\{i,...,j\}.$
One observes that $Z_{1}$ and $Z_{2}$ are smooth, indeed $Z_{1}$ is a product of Hyperquot schemes and $Z_{2}$ is a product of $\mathcal{Z}_{\vec{d}}^{i,j}$ with a Hyperquot scheme. Finally, one notes that the scheme theoretic intersection $Z_{1}\cap Z_{2}\cong \mathcal{W}_{\vec{d}}^{i,j}$. In this proof, we showed that dim $\mathcal{W}_{\vec{d}}^{i,j}=$ dim $\mathcal{U}_{\vec{d}}^{i,j}$, therefore, the co-dimension counts follow from the Propositions \ref{Geometry of moduli space $U$} and \ref{proposition: smoothness and dimension of moduli space Z}.
\end{proof}
\subsection{Some sheaves on $\mathcal{U}_{\vec{d}}^{i,j}$, $\mathcal{W}_{\vec{d}}^{i,j}$ and $\mathcal{X}_{\vec{d}}^{i,j}$}
We have morphisms 
\begin{equation}\label{morphisms p on U,X,W to C}
p^{U}:\mathcal{U}_{\vec{d}}^{i,j}\rightarrow C,\,\,p^{X}:\mathcal{X}_{\vec{d}}^{i,j}\rightarrow C \text{ and }p^{W}:\mathcal{W}_{\vec{d}}^{i,j}\rightarrow C,
\end{equation}
which assign to the points (\ref{The moduli space: U}), (\ref{The moduli space: X}) and (\ref{The moduli space: W}), the points\\ Supp$\,E_{j}/F^{'}_{j}=$ Supp$\,F_{j}/H_{j}$, Supp$\,F_{j}/H_{j}$ and Supp$\,E_{j}/F^{'}_{j}$ respectively.\\\\
Similarly we have morphisms 
\begin{equation}\label{morphisms q on U,X,W to C}
q^{U}:\mathcal{U}_{\vec{d}}^{i,j}\rightarrow C,\,\,q^{X}:\mathcal{X}_{\vec{d}}^{i,j}\rightarrow C \text{ and }q^{W}:\mathcal{W}_{\vec{d}}^{i,j}\rightarrow C,
\end{equation}
which assign to the points (\ref{The moduli space: U}), (\ref{The moduli space: X}) and (\ref{The moduli space: W}), the points\\ Supp$\,E_{j}/F_{j}=$ Supp$\,F^{'}_{j}/H_{j}$, Supp$\,E_{j}/F_{j}$ and Supp$\,F^{'}_{j}/H_{j}$ respectively.\\\\
Using the graphs of the above morphisms and replicating the construction of \ref{construction of line bundle L} to the moduli spaces $\mathcal{U}_{\vec{d}}^{i,j},\mathcal{W}_{\vec{d}}^{i,j}$ and $\mathcal{X}_{\vec{d}}^{i,j}$, we obtain some tautological line bundles; On $\mathcal{U}_{\vec{d}}^{i,j}$, we have line bundles 
\[\mathscr{L}_{1},\mathscr{L}_{2},\tilde{\mathscr{L}}_{1}, \tilde{\mathscr{L}}_{2}, \mathscr{L}_{3},\]
 whose respective fibres over a closed point of the moduli given by the diagram (\ref{The moduli space: U}), are the length one quotients $E_{j}/F_{j}, F_{j}/H_{j},E_{j}/F^{'}_{j},F^{'}_{j}/H_{j}, E_{i}/F_{i}$. We denote the respective first Chern classes of these bundles as
 \[
 \lambda_{1},\lambda_{2},\tilde{\lambda}_{1},\tilde{\lambda}_{2},\lambda_{3}
 \]
On $\mathcal{X}_{\vec{d}}^{i,j}$, we have line bundles 
\[\mathscr{L}_{1},\mathscr{L}_{2}, \mathscr{L}_{3},\]
 whose respective fibres over a closed point of the moduli given by the diagram (\ref{The moduli space: X}), are the length one quotients $E_{j}/F_{j}, F_{j}/H_{j}, E_{i}/F_{i}$. As before, we denote the respective first Chern classes of these bundles as
 \[
 \lambda_{1},\lambda_{2},\lambda_{3}
 \]
 On $\mathcal{W}_{\vec{d}}^{i,j}$, we have line bundles 
\[\tilde{\mathscr{L}}_{1}, \tilde{\mathscr{L}}_{2}, \mathscr{L}_{3},\]
 whose respective fibres over a closed point of the moduli given by the diagram (\ref{The moduli space: W}), are the length one quotients $E_{j}/F^{'}_{j},F^{'}_{j}/H_{j}, E_{i}/F_{i}$. Finally, we denote the respective first Chern classes of these bundles as
 \[
 \tilde{\lambda}_{1},\tilde{\lambda}_{2},\lambda_{3}.
 \]
Note that we have intentionally confused the notation of line bundles/Chern classes on $\mathcal{W}_{\vec{d}}^{i,j}$ or $\mathcal{X}_{\vec{d}}^{i,j}$ with their respective pull-backs on $\mathcal{U}_{\vec{d}}^{i,j}$ via the maps $\mathcal{\eta}_{\vec{d}}^{i,j}$ or $\mathcal{\theta}_{\vec{d}}^{i,j}$.\\\\
If $\Gamma_{q}:\mathcal{U}_{\vec{d}}^{i,j}\rightarrow \mathcal{U}_{\vec{d}}^{i,j} \times C$ and  $\Gamma_{p}:\mathcal{U}_{\vec{d}}^{i,j}\rightarrow \mathcal{U}_{\vec{d}}^{i,j} \times C$ are the graphs of the maps $q^{U}$ and $p^{U}$ respectively, then we can define rank $r-1$ vector bundles on $\mathcal{U}_{\vec{d}}^{i,j}$
\[
\mathcal{G}_{1},\mathcal{G}_{2},\tilde{\mathcal{G}}_{1},\tilde{\mathcal{G}}_{2},\mathcal{G}_{3}
\]
to be the kernels of the morphisms:

\[
\mathcal{E}_{{j}_{|_{\Gamma_{q}}}}\rightarrow \mathscr{L}_{1},\,\mathcal{F}_{{j}_{|_{\Gamma_{p}}}}\rightarrow \mathscr{L}_{2},\,\mathcal{E}_{{j}_{|_{\Gamma_{p}}}}\rightarrow \tilde{\mathscr{L}}_{1}, \,\mathcal{F}'_{{j}_{|_{\Gamma_{q}}}}\rightarrow \tilde{\mathscr{L}_{2}},\,\mathcal{E}_{{i}_{|_{\Gamma_{q}}}}\rightarrow \mathscr{L}_{3}
\]
\newline
respectively. Above, the sheaves $\mathcal{E}_{{j}},\mathcal{E}_{{i}}$, $\mathcal{F}_{{j}}, \mathcal{F}'_{{j}}$ are universal sheaves on $\mathcal{U}_{\vec{d}}^{i,j}\times C$ whose fibres over a closed point of the form \eqref{The moduli space: U} are the sheaves $E_{j},E_{i},F_{j},F_{j}'$ respectively.\\\\
Similarly, we can define rank $r-1$ vector bundles on $\mathcal{X}_{\vec{d}}^{i,j}$:
\[
\mathcal{G}_{1},\mathcal{G}_{2},\mathcal{G}_{3}
\]
which are associated with the inclusions $F_{j}\subseteq E_{j}$, $H_{j}\subseteq F_{j}$ and $F_{i}\subseteq E_{i}$ respectively (cf. diagram \ref{The moduli space: X}) and rank $r-1$ vector bundles on $\mathcal{W}_{\vec{d}}^{i,j}$:
\[
\tilde{\mathcal{G}}_{1},\tilde{\mathcal{G}}_{2},\mathcal{G}_{3}
\]
which are associated with the inclusions $F'_{j}\subseteq E_{j}$, $H_{j}\subseteq F'_{j}$ and $F_{i}\subseteq E_{i}$ respectively (cf. diagram \ref{The moduli space: W}).
\subsection{Some divisors on $\mathcal{U}_{\vec{d}}^{i,j}$: $\Delta_{U}$, Diag and Exp}\label{Subsection: Divisors}
Let $\Delta\hookrightarrow C\times C$ denote the diagonal in $C\times C$ and let 
\begin{equation}\label{Equation: Definition of Delta U}
\Delta_{U}:=(p^{U}\times q^{U})^{-1}(\Delta),
\end{equation}
be the scheme theoretic pre-image of $\Delta\subseteq C\times C$ in $\mathcal{U}_{\vec{d}}^{i,j}$. In other words, $\Delta_{U}$
is the sub-scheme of $\mathcal{U}_{\vec{d}}^{i,j}$ given by the condition $p=q$ in Diagram \ref{The moduli space: U} and defines an effective Cartier Divisor.\\\\
Now, observe that we have the following morphisms of line bundles on $\mathcal{U}_{\vec{d}}^{i,j}$:
\[
\mathscr{L}_{2}\xrightarrow{s_{1}}\tilde{\mathscr{L}}_{1} \text{ and }\tilde{\mathscr{L}}_{2}\xrightarrow{s_{2}}\mathscr{L}_{1},
\]
that we think  as sections of the line bundles $\mathscr{L}_{2}^{-1}\tilde{\mathscr{L}}_{1}$ and $\tilde{\mathscr{L}}_{2}^{-1}{\mathscr{L}}_{1}$ respectively. Then, one observes that the zero locus of both the sections $s_{1}$ and $s_{2}$ is the divisor on $\mathcal{U}_{\vec{d}}^{i,j}$ obtained by imposing the conditions $F_{j}=F_{j}^{'}$ and $p=q$  (cf. Diagram \ref{The moduli space: U}).  We denote this divisor by Diag. The scheme Diag has the following description of it's closed points:

\begin{equation}\tag{Diag}\label{The moduli space: Diag}
    \begin{tikzcd}[cramped]
	& {E_{j}} & {...\hookrightarrow E_{i+1}\,\,\,\,} & {E_{i}} & {...\hookrightarrow E_{1}} & V \\
	& {F_{j}=F_{j}^{'}} & {...\hookrightarrow F_{i+1}} & {F_{i}} \\
	{E_{n}\hookrightarrow...\hookrightarrow E_{j+1}} & {H_{j}}
	\arrow[hook, from=1-2, to=1-3]
	\arrow["{\,\,\,\,\,\,\,\,\,\,\,\,...}", shift left=5, draw=none, from=1-2, to=2-2]
	\arrow[hook, from=1-3, to=1-4]
	\arrow[shift right=2, draw=none, from=1-3, to=2-3]
	\arrow[shift right=5, draw=none, from=1-3, to=2-3]
	\arrow[hook, from=1-4, to=1-5]
	\arrow[hook, from=1-5, to=1-6]
	\arrow["q"'{pos=0.4}, hook, from=2-2, to=1-2]
	\arrow[hook, from=2-2, to=2-3]
	\arrow["q"'{pos=0.4}, shift right=3, hook, from=2-3, to=1-3]
	\arrow[hook, from=2-3, to=2-4]
	\arrow["q"'{pos=0.4}, hook, from=2-4, to=1-4]
	\arrow[hook, from=3-1, to=3-2]
	\arrow["q"', hook, from=3-2, to=2-2]
\end{tikzcd}.
\end{equation}
\begin{prop}\label{smoothness of Diag}
    The scheme Diag is smooth.
\end{prop}
\begin{proof}
    This follows from the fact that the scheme Diag is a $\mathbb{P}^{r-1}-$bundle followed by a Hyperquot scheme bundle, over the space $\mathcal{Z}_{(d_{1},...,d_{j})}^{i,j}$, which is smooth by Proposition \ref{proposition on smoothness and dimension of moduli space P}.
\end{proof}
\noindent Let us denote the fundamental class of Diag by $[\text{Diag}]\in H^{2}(\mathcal{U}^{i,j}_{\vec{d}})$. Then the following proposition is a consequence of the fact that Diag is the zero locus of both $s_{1}$ and $s_{2}$:
\begin{prop}\label{Diag in cohomology}
    The following equality holds in $H^{2}(\mathcal{U}^{i,j}_{\vec{d}}):$
    \[
    [\text{Diag}]=\lambda_{1}-\tilde{\lambda}_{2}=\tilde{\lambda}_{1}-\lambda_{2}.
    \]
\end{prop}

\noindent Note that the birational map $\eta_{\vec{d}}^{i,j}:\mathcal{U}_{\vec{d}}^{i,j}\rightarrow\mathcal{X}_{\vec{d}}^{i,j}$ is a $\mathbb{P}^{1}$-bundle over the closed locus $\mathfrak{D}$, which consists of closed points of the form \eqref{The moduli space: X}, such that $p=q$ and $E_{j}/H_{j}\cong \mathbb{C}_{p}\oplus \mathbb{C}_{p}$. Also note that $\eta_{\vec{d}}^{i,j}$ is an isomorphism on the complement of $\mathfrak{D}$. Since $\mathcal{U}_{\vec{d}}^{i,j}$ and $\mathcal{X}_{\vec{d}}^{i,j}$ are smooth, it follows that $\mathfrak{D}$ is of codimension $2$ in $\mathcal{X}_{\vec{d}}^{i,j}$.  Let us denote by $\mathfrak{T}:=\mathcal{X}_{\vec{d}}^{i,j}\backslash\mathfrak{D}$, the open locus of $\mathcal{X}_{\vec{d}}^{i,j}$ over which $\eta_{\vec{d}}^{i,j}$ is an isomorphism.\\\\
\noindent Let us define the effective divisor Exp as:
\begin{equation}\label{equation: definition of exp}
    \text{Exp}:=\Delta_{U}-\text{Diag}.
\end{equation} 
 One observes that the open immersion $\left({\eta_{\vec{d}}^{i,j}}_{|_{\Delta_{U}}}\right)^{-1}(\mathfrak{T})\rightarrow \Delta_{U}$ factors through the closed immersion Diag$\hookrightarrow \Delta_{U}$. Hence we obtain that Diag is not contained in the support of $\text{Exp}$ and moreover, one notes that the support of Exp coincides with the the exceptional locus of $\eta_{\vec{d}}^{i,j}$. Explicitly, closed points of Exp are of the form of diagram \eqref{The moduli space: U}, such that $p=q$ and $E_{j}/H_{j}\cong \mathbb{C}_{p}\oplus \mathbb{C}_{p}.$

 \begin{lem}\label{Lemma: exact sequence, Diag and Exp}
Let $\mathcal{I}_{\text{Diag}}\subseteq\mathcal{O}_{\mathcal{U}^{i,j}_{\vec{d}}}$ be the ideal sheaf of Diag, then we have the following short exact sequence of sheaves on $\mathcal{U}^{i,j}_{\vec{d}}$:
\begin{equation}\label{equation:exact sequence, Diag and Exp}
    0\rightarrow\mathcal{I}_{\text{Diag}}\otimes\mathcal{O}_{\text{Exp}}\rightarrow {\mathcal{O}}_{\Delta_{U}}\rightarrow  \mathcal{O}_{\text{Diag}}\rightarrow 0.
\end{equation}
 \end{lem}
 \begin{proof}Since $\Delta_{U}$ is the scheme theoretic union of Diag and Exp, we have the exact sequence
 \[
 0\rightarrow \frac{\mathcal{I}_{\text{Diag}}}{\mathcal{I}_{\text{Diag}}\cap\mathcal{I}_{\text{Exp}}}\rightarrow {\mathcal{O}}_{\Delta_{U}}\rightarrow  \mathcal{O}_{\text{Diag}}\rightarrow 0,
 \]
     where $\mathcal{I}_{\text{Exp}}$ is the ideal sheaf of Exp on $\mathcal{U}^{i,j}_{\vec{d}}$. It remains to argue that \\
     ${\mathcal{I}_{\text{Diag}}\cap\mathcal{I}_{\text{Exp}}}={\mathcal{I}_{\text{Diag}}\cdot\mathcal{I}_{\text{Exp}}}$. This is true whenever locally, the equation cutting out Exp is a non-zero divisor in Diag, or in other words codim$_{\mathcal{U}^{i,j}_{\vec{d}}}(\text{Diag}\,\cap\,\text{Exp})=2$. This is because $\eta^{i,j}_{\vec{d}}$ maps $\text{Diag}\,\cap\,\text{Exp}$ bijectively onto $\mathfrak{D}$, the fact that $\mathfrak{D}$ is a closed set of codimension $2$ in $\mathcal{X}^{i,j}_{\vec{d}}$ and that dim $\mathcal{X}^{i,j}_{\vec{d}}=$ dim $\mathcal{U}^{i,j}_{\vec{d}}$.
 \end{proof}

\begin{prop}\label{proposition: short exact sequence of G's 1}
   On $\mathcal{U}^{i,j}_{\vec{d}}$, we have the following short exact sequence of sheaves:

\[
0\rightarrow\mathcal{G}_{2}\rightarrow \tilde{\mathcal{G}}_{1}\rightarrow \tilde{\mathscr{L}}_{2}\otimes \mathcal{O}_{\text{Exp}}\rightarrow 0.
\]
\end{prop}

\begin{proof}
    First, restricting universal sheaves on $\mathcal{U}^{i,j}_{\vec{d}}\times C$ to $\Gamma_{p}$ yields the following short exact sequence of two-step complexes on $\mathcal{U}^{i,j}_{\vec{d}}$:
    \begin{equation}
\begin{tikzcd}[cramped]
	& {\mathcal{}} &&& \\
	0 & {{\mathcal{F}_{j}}_{|_{\Gamma_{p}}}} & {{\mathcal{E}_{j}}_{|_{\Gamma_{p}}}} & {{\mathscr{L}}_{1}\otimes\mathcal{O}_{{\Delta}_{U}}} & 0 \\
	0 & {\mathscr{L}_{2}} & {\tilde{\mathscr{L}}_{1}} & {\tilde{\mathscr{L}}_{1}\otimes\mathcal{O}_{\text{Diag}}} & 0 \\
	& 0 & 0 & 0
	\arrow[from=2-1, to=2-2]
	\arrow[from=2-2, to=2-3]
	\arrow[from=2-2, to=3-2]
	\arrow[from=2-3, to=2-4]
	\arrow[from=2-3, to=3-3]
	\arrow[from=2-4, to=2-5]
	\arrow[from=2-4, to=3-4]
	\arrow[from=3-1, to=3-2]
	\arrow[from=3-2, to=3-3]
	\arrow[from=3-2, to=4-2]
	\arrow[from=3-3, to=3-4]
	\arrow[from=3-3, to=4-3]
	\arrow[from=3-4, to=3-5]
	\arrow[from=3-4, to=4-4]
\end{tikzcd}
    \end{equation}
\noindent The first two columns follow from the definitions of the involved sheaves, the last column is simply the quotient 
\[
\mathcal{O}_{\Delta_{U}}\rightarrow \mathcal{O}_{\text{Diag}} 
\]
tensored with $\mathscr{L}_{1}$ and the observation that $\mathscr{L}_{{1}_{|_{\text{Diag}}}}=\tilde{\mathscr{L}}_{{1}_{|_{\text{Diag}}}}$. The second row is implied by the fact that Diag is cut out by the section $$s_1:{\mathscr{L}}_{2}\rightarrow\tilde{\mathscr{L}}_{1}.$$ 
The first row is obtained by restricting the following universal exact sequence on $\mathcal{U}^{i,j}_{\vec{d}}\times C$ to $\Gamma_{p}$:
\[
0\rightarrow\mathcal{F}_{j}\rightarrow \mathcal{E}_{j}\rightarrow\mathscr{L}_{1}\otimes\mathcal{O}_{\Gamma_{q}}\rightarrow 0,
\]
 as well as observing that $\Gamma_{p}\cap \Gamma_{q}\cong \Delta_{U}$ and that $\mathcal{T}or_{1}^{\mathcal{U}^{i,j}_{\vec{d}}\times C}\left(\mathcal{O}_{\Gamma_{q}},\mathcal{O}_{\Gamma_{p}}\right)=0.$ Indeed, the vanishing of the Tor follows from the fact that $\Gamma_{p}$ and $\Gamma_{q}$ are divisors whose intersection is a pure codimension $2$ sub-scheme in $\mathcal{U}^{i,j}_{\vec{d}}\times C$. We leave the straightforward verification of the commutativity of the two squares to the reader.\\\\
 Since the vertical maps are surjective, taking the long exact sequence of the above exact sequence of complexes yields a short exact sequence whose terms are the kernels of the three columns. Hence, from Lemma \ref{Lemma: exact sequence, Diag and Exp} and the isomorphism $\mathcal{I}_{\text{Diag}}\cong \tilde{\mathscr{L}}_{2}\otimes\mathscr{L}_{1}^{-1}$ induced by the section $s_{2}$, we obtain the exact sequence:
 \[
 0\rightarrow\mathcal{G}_{2}\rightarrow \tilde{\mathcal{G}}_{1}\rightarrow \tilde{\mathscr{L}}_{2}\otimes \mathcal{O}_{\text{Exp}}\rightarrow 0.
 \]
 as desired.
\end{proof}

\begin{cor}\label{corollary (use: proof of commutativity of a k neq n): difference of c(G,z)}
    The following identity holds in $H^{*}(\mathcal{U}^{i,j}_{\vec{d}})[z]$:

    \[
    c(\mathcal{G}_{2},z)-c(\tilde{\mathcal{G}}_{1},z)=[\text{Exp}]\cdot\frac{c\left({\mathcal{E}_{j}}_{_{|_{\Gamma_{p}}}},z\right)}{c(\mathscr{L}_{1},z)\cdot c(\mathscr{L}_{2},z)}.
    \]
\end{cor}
\begin{proof}
Proposition \ref{proposition: short exact sequence of G's 1} implies that
\[
c(\tilde{\mathcal{G}_{1}},z)=c({\mathcal{G}_{2}},z)\cdot c(\tilde{\mathscr{L}}_{2}\otimes \mathcal{O}_{\text{Exp}},z).
\]
Therefore, 
\[
c(\mathcal{G}_{2},z)-c(\tilde{\mathcal{G}}_{1},z)=c(\tilde{\mathcal{G}}_{1},z)\left[\frac{1}{c(\tilde{\mathscr{L}}_{2}\otimes \mathcal{O}_{\text{Exp}},\,z)}-1\right].
\]
\[
=c(\tilde{\mathcal{G}}_{1},z)\left[\frac{c(\tilde{\mathscr{L}}_{2}\otimes\mathcal{O}(-\text{Exp}),z)}{c(\tilde{\mathscr{L}}_{2},\,z)}-1\right]=[\text{Exp}]\cdot\frac{c(\tilde{\mathcal{G}}_{1},z)}{c(\tilde{\mathscr{L}}_{2},z)}
\]
\[
=[\text{Exp}]\cdot\frac{c\left({\mathcal{E}_{j}}_{_{|_{\Gamma_{p}}}},z\right)}{c(\tilde{\mathscr{L}}_{1},z)\cdot c(\tilde{\mathscr{L}}_{2},z)}.
\]
    The last equality follows from the definition of $\tilde{\mathcal{G}}_{1}$. It remains to show that 
    \[
    {c(\tilde{\mathscr{L}}_{1},z)\cdot c(\tilde{\mathscr{L}}_{2},z)}={c({\mathscr{L}}_{1},z)\cdot c({\mathscr{L}}_{2},z)}.
    \]
    Indeed,
    \[
    c(\tilde{\mathscr{L}}_{1},z)\cdot c(\tilde{\mathscr{L}}_{2},z)=(z-\tilde{\lambda}_{1})(z-\tilde{\lambda}_{2}).
    \]
    By Proposition \ref{Diag in cohomology}, this expression is equal to 
    \[
    (z-\tilde{\lambda}_{1})(z-{\lambda}_{1}+\text{Diag})=z^{2}+\left(-{\lambda}_{1}-\tilde{\lambda}_{1}+\text{Diag}\right)z+{\lambda}_{1}\tilde{\lambda}_{1}-\text{Diag}\cdot\tilde{\lambda}_{1}
    \]
    \[
    =z^{2}+\left(-{\lambda}_{1}-\tilde{\lambda}_{1}+\text{Diag}\right)z+{\lambda}_{1}\tilde{\lambda}_{1}-\text{Diag}\cdot{\lambda}_{1}
    =(z-\lambda_{1})(z-\tilde{\lambda}_{1}+\text{Diag}).
    \]
    The last line follows from the fact that $\tilde{\lambda}_{{1}_{|_{\text{Diag}}}}={\lambda}_{{1}_{|_{\text{Diag}}}}$. Using Proposition \ref{Diag in cohomology} again gives
    \[
    (z-\lambda_{1})(z-\tilde{\lambda}_{1}+\text{Diag})=(z-\lambda_{1})(z-{\lambda}_{2})={c({\mathscr{L}}_{1},z)\cdot c({\mathscr{L}}_{2},z)}.
    \]
    This finishes the proof.
\end{proof}
\subsection{The spaces $\mathcal{U}_{\vec{d}}^{i,j}$, $\mathcal{W}_{\vec{d}}^{i,j}$ and $\mathcal{X}_{\vec{d}}^{i,j}:$ The case $i=j$}\label{subsection: U,X,W case:i=j} Note that when $i=j$, the spaces $\mathcal{W}_{\vec{d}}^{j,j}$ and $\mathcal{X}_{\vec{d}}^{j,j};$ are just hyperquot schemes that parameterise flags of length $n+2$. In particular, the space  $\mathcal{W}_{\vec{d}}^{j,j}$ is smooth, which turn out to not be the case when $i< j$ and rank $V>1$. Furthermore, observe that there is an obvious isomorphism 
\[
\xi:\mathcal{W}_{\vec{d}}^{j,j}\rightarrow\mathcal{X}_{\vec{d}}^{j,j},\,F_{j}^{'}\mapsto F_{j}:=F_{j}^{'}
\]
(cf. diagrams \ref{The moduli space: W} and \ref{The moduli space: X}). However, it would still be convenient to distinguish between the two spaces since $\xi$ is not compatible with the natural forgetful maps from $\mathcal{U}_{\vec{d}}^{j,j}$, indeed
\[
\xi\circ\theta^{j,j}_{\vec{d}} \neq\eta^{j,j}_{\vec{d}}
\]
\newline
since $\xi$ does not exchange the support points $p,q$.\\\\
As before let $\delta:=[\Delta]$, $\Delta_{W}:=(p^{W}\times q^{W})^{-1}(\Delta)$ and $\Delta_{X}:=(p^{X}\times q^{X})^{-1}(\Delta)$. The following lemma follows directly from the definitions.
\begin{lem}\label{theta and eta are iso onto delta for i=j.}
    The maps $\theta^{j,j}_{\vec{d}}$ and $\eta^{j,j}_{\vec{d}}$ map Diag isomorphically onto $\Delta_{W}$ and $\Delta_{X}$ respectively.
\end{lem}
\noindent Therefore, we have the following equalities in cohomology:
\begin{cor}\label{corollay about pushforward of diag being delta}
    From lemma \ref{theta and eta are iso onto delta for i=j.}, it follows that:
    \[
    {\theta^{j,j}_{\vec{d}}}_{*}\left([\text{Diag}]\right)=\left(p_{W}\times q_{W}\right)^{*}(\delta) \text{ and }{\eta^{j,j}_{\vec{d}}}_{*}\left([\text{Diag}]\right)=(p_{X}\times q_{X})^{*}(\delta).
    \]
\end{cor}
\noindent We omit the proof of the following proposition, since it is identical to the proof of Proposition \ref{proposition: short exact sequence of G's 1}.
\begin{prop}\label{proposition: short exact sequence of G's 2}
On $\mathcal{U}^{j,j}_{\vec{d}}$, we have the following short exact sequences of sheaves:

\begin{equation}\label{Equation: SES k=n 1}
    0\rightarrow\mathcal{G}_{2}\rightarrow \tilde{\mathcal{G}}_{1}\rightarrow \tilde{\mathscr{L}}_{2}\otimes \mathcal{O}_{\text{Exp}}\rightarrow 0.\\
\end{equation}
and 
\begin{equation}\label{Equation: SES k=n 2}
    0\rightarrow\tilde{\mathcal{G}}_{2}\rightarrow {\mathcal{G}}_{1}\rightarrow {\mathscr{L}}_{2}\otimes \mathcal{O}_{\text{Exp}}\rightarrow 0.\\
\end{equation}
\end{prop}
\noindent The above proposition implies the following useful equality in cohomology. Again, we omit the proof since it is similar to the proof of Corollary \ref{corollary (use: proof of commutativity of a k neq n): difference of c(G,z)}.
\begin{cor}\label{corollary (use: proof of commutativity of a k=n): differences of c(G,z)c(G,w)}
    The following identity holds in $H^{*}(\mathcal{U}^{j,j}_{\vec{d}})[z,w]$:

    \[
    c(\mathcal{G}_{1},w)c(\mathcal{G}_{2},z)-c(\tilde{\mathcal{G}}_{1},z)c(\tilde{\mathcal{G}}_{2},w)
    \]
    
    \[=[\text{Exp}]\cdot\frac{c\left({\mathcal{E}_{j}}_{_{|_{\Gamma_{p}}}},z\right)\cdot c\left({\mathcal{E}_{j}}_{_{|_{\Gamma_{q}}}},w\right)}{c(\mathscr{L}_{1},w)\cdot c(\mathscr{L}_{1},z)\cdot c(\mathscr{L}_{2},z)}- [\text{Exp}]\cdot\frac{c\left({\mathcal{E}_{j}}_{_{|_{\Gamma_{p}}}},z\right)\cdot c\left({\mathcal{E}_{j}}_{_{|_{\Gamma_{q}}}},w\right)}{c(\tilde{\mathscr{L}}_{1},z)\cdot c(\tilde{\mathscr{L}}_{1},w)\cdot c(\tilde{\mathscr{L}}_{2},w)}.
    \]
\end{cor}

\subsection{The skew-nested Quot schemes $\tilde{\mathcal{U}}_{\vec{d}}^{i}\,$, $\mathcal{S}_{\vec{d}}^{i}$ and $\mathcal{T}_{\vec{d}}^{i}$ }\label{subsection: moduli spaces U,S,T}
Let $\vec{d}=(d_1,...,d_{n})$ be an $n-$tuple of non-negative integers. For $1\leq i\leq n$, we define spaces $\tilde{\mathcal{U}}_{\vec{d}}^{i}\,$, $\mathcal{S}_{\vec{d}}^{i}$ and $\mathcal{T}_{\vec{d}}^{i}$ to be respectively the fine moduli spaces parameterising the following diagrams of inclusions of rank $r$ sub-sheaves of $V$:

\begin{equation}\tag{$\tilde{\mathcal{U}}_{\vec{d}}^{i}$}\label{The moduli space : U tilde}
\begin{tikzcd}[cramped]
	& {E'_i} & {F_{i}} & {E_{i-1}\hookrightarrow...\hookrightarrow E_{1}} & V \\
	{E_{n}\hookrightarrow...\hookrightarrow E_{i+1}} & {\tilde{F}_{i}} & {E_{i}}
	\arrow["q", hook, from=1-2, to=1-3]
	\arrow[hook, from=1-3, to=1-4]
	\arrow[hook, from=1-4, to=1-5]
	\arrow[hook, from=2-1, to=2-2]
	\arrow["p"', hook, from=2-2, to=1-2]
	\arrow["q"', hook, from=2-2, to=2-3]
	\arrow["p"', hook, from=2-3, to=1-3]
\end{tikzcd}
\end{equation}

\begin{equation}
\tag{${\mathcal{S}}_{\vec{d}}^{i}$}\label{The moduli space : S}
\begin{tikzcd}[cramped]
	& {E'_i} & {E_{i-1}} & {E_{i-2}\hookrightarrow...\hookrightarrow E_{1}} & V \\
	{E_{n}\hookrightarrow...\hookrightarrow E_{i+1}} & {\tilde{F}_{i}} & {E_{i}}
	\arrow[hook, from=1-2, to=1-3]
	\arrow[hook, from=1-3, to=1-4]
	\arrow[hook, from=1-4, to=1-5]
	\arrow[hook, from=2-1, to=2-2]
	\arrow["p"', hook, from=2-2, to=1-2]
	\arrow["q"', hook, from=2-2, to=2-3]
	\arrow[hook, from=2-3, to=1-3]
\end{tikzcd}
\end{equation}

\begin{equation}
\tag{${\mathcal{T}}_{\vec{d}}^{i}$}\label{The moduli space : T}
\begin{tikzcd}[cramped]
	& {E'_i} & {F_{i}} & {E_{i-1}\hookrightarrow...\hookrightarrow E_{1}} & V \\
	{E_{n}\hookrightarrow...\hookrightarrow E_{i+2}} & {{E}_{i+1}} & {E_{i}}
	\arrow["q", hook, from=1-2, to=1-3]
	\arrow[hook, from=1-3, to=1-4]
	\arrow[hook, from=1-4, to=1-5]
	\arrow[hook, from=2-1, to=2-2]
	\arrow[hook, from=2-2, to=1-2]
	\arrow[hook, from=2-2, to=2-3]
	\arrow["p"', hook, from=2-3, to=1-3]
\end{tikzcd}
\end{equation}
\newline
\noindent The above are diagrams of inclusions of certain rank $r$ sub-sheaves $E_{*},E'_{i},F_{i},\tilde{F}_{i}$ of $V$ on $C$, for $*\in\{1,...,n\}$ and in each of the above diagrams we impose that $\text{dim}_{\mathbb{C}}(V/E_{*})=d_{*}$.
Note that we have maps
\begin{equation}\label{birational-ish maps : sigma and tau}
    \sigma_{\vec{d}}^{i}:\tilde{\mathcal{U}}_{\vec{d}}^{i}\rightarrow\mathcal{S}_{\vec{d}}^{i} \text{ and } \tau_{\vec{d}}^{i}:\tilde{\mathcal{U}}_{\vec{d}}\rightarrow\mathcal{T}_{\vec{d}}^{i},
\end{equation}
which forget the sheaves $F_{i}$ and $\tilde{F}_{i}$ respectively.

\begin{lem}\label{Lemma excision argument one}
    The maps $\sigma_{\vec{d}}^{i}$ and $\tau_{\vec{d}}^{i}$ are isomorphisms on the open loci given by the condition $E_{i}\neq E'_{i}$.
\end{lem}

\begin{proof}
    Let us consider a point of the form \eqref{The moduli space : S} such that $E_{i}\neq E'_{i}$. We claim that there exists a unique $\sigma_{\vec{d}}^{i}\,-$ pre-image of this point in $\tilde{\mathcal{U}}_{\vec{d}}^{i}\,$. Note that $E_{i}\neq E'_{i}$ implies that
    \[
    \frac{E_{i}+E'_{i}}{E_{i}}\cong\mathbb{C}_{p}\text{ and }\frac{E_{i}+E'_{i}}{E'_{i}}\cong \mathbb{C}_{q}.
    \]
    Since any sheaf that contains $E_{i}$ and $E'_{i}$ must contain their sum $F_{i}:=E_{i}+E'_{i}$, we have proved our claim. Indeed, this argument works in families and hence we obtain that $\sigma_{\vec{d}}^{i}$ is an isomorphism over the open locus $\left\{E_{i}\neq E'_{i}\right\}.$ The proof of the fact that $\tau_{\vec{d}}^{i}$ is an isomorphism over the locus $\left\{E_{i}\neq E'_{i}\right\}\subseteq \mathcal{T}_{\vec{d}}^{i}$ is similar, with $\tilde{F}_{i}:=E_{i}\cap E'_{i}$ defining the unique $\tau_{\vec{d}}^{i}\,-$ pre-image of a point of the form \eqref{The moduli space : T}, with $E_{i}\neq E'_{i}$. 
\end{proof}
\noindent We also have morphisms: 
\begin{equation}\label{E,E' forgetful map for moduli space S}
    k_{S}:\mathcal{S}_{\vec{d}}^{i}\rightarrow F^{n}\text{Quot}_{\vec{d}}\,(V)\times F^{n}\text{Quot}_{\vec{d}}\,(V)
\end{equation}
\begin{equation}\label{E,E' forgetful map for moduli space T}
    k_{T}:\mathcal{T}_{\vec{d}}^{i}\rightarrow F^{n}\text{Quot}_{\vec{d}}\,(V)\times F^{n}\text{Quot}_{\vec{d}}\,(V).
\end{equation}
\noindent The morphism $k_{S}$ (respectively $k_{T}$), maps a closed point of the form (\ref{The moduli space : S}) (respectively (\ref{The moduli space : T})) to the pair of flags:
\[
 \left(\left(E_{n}\subseteq...\subseteq{E}_{i+1}\subseteq E_{i}\subseteq{E}_{i-1}...\subseteq E_{1}\right),\,\left(E_{n}\subseteq...\subseteq{E}_{i+1}\subseteq E'_{i}\subseteq{E}_{i-1}...\subseteq E_{1}\right)\right),
\]
which is a point in the product $F^{n}\text{Quot}_{\vec{d}}\,(V)\times F^{n}\text{Quot}_{\vec{d}}\,(V).$\\\\
We also have morphisms 
\begin{equation}\label{morphisms p,q on U,X,W to C}
p^{\tilde{U}},q^{\tilde{U}}:\tilde{\mathcal{U}}_{\vec{d}}^{i}\rightarrow C,\,\,p^{S},q^{S}:\mathcal{S}_{\vec{d}}^{i}\rightarrow C \text{ and }p^{T},q^{T}:\mathcal{T}_{\vec{d}}^{i}\rightarrow C.
\end{equation}
The maps $p^{\tilde{U}},p^{S},p^{T}$ respectively assign to closed points of the form of the Diagrams (\ref{The moduli space : U tilde}), (\ref{The moduli space : S}) and (\ref{The moduli space : T}), the support points $p$. Whereas, the maps $q^{\tilde{U}},q^{S},q^{T}$ respectively assign to closed points of the form of the Diagrams (\ref{The moduli space : U tilde}), (\ref{The moduli space : S}) and (\ref{The moduli space : T}), the support point $q$.
\subsection{Geometry of $\tilde{\mathcal{U}}_{\vec{d}}^{i}\,$, $\mathcal{S}_{\vec{d}}^{i}$ and $\mathcal{T}_{\vec{d}}^{i}$}
We will now discuss the dimension counts, irreducibility and smoothness properties of the spaces $\tilde{\mathcal{U}}_{\vec{d}}^{i}\,$, $\mathcal{S}_{\vec{d}}^{i}$ and $\mathcal{T}_{\vec{d}}^{i}$ .

\begin{prop}\label{Geometry of moduli space U tilde}
    The moduli space $\tilde{\mathcal{U}}_{\vec{d}}^{i}$ is a smooth projective variety of dimension $rd_{n}$ if $i< n$ and of dimension $r(d_{n}+1)$ if $i=n$.
\end{prop}

\begin{proof}
    One observes that $\tilde{\mathcal{U}}_{\vec{d}}^{i}$ is precisely the space ${\mathcal{Z}}_{\vec{f}}^{i,i+1}$, where $\vec{f}$ is the \\
    $(n+1)-$tuple $(d_{1},...,d_{i-1},d_{i}-1,d_{i},...,d_{n})$. Therefore, the proposition follows from Proposition \ref{proposition: smoothness and dimension of moduli space Z}.
\end{proof}

\begin{prop}\label{Geometry of moduli space $S$}
    The moduli space $\mathcal{S}_{\vec{d}}^{i}$ is reduced, l.c.i. and consists of two components of dimension $rd_{n}$ if $i< n$ and of dimension $r(d_{n}+1)$ if $i=n$.
\end{prop}

\begin{proof}
    For notational convenience, let us define the numbers dim$(i):=rd_{n}$ if $i\in\{1,...,n-1\}$ and $dim$(i)=$r(d_{n}+1)$ if $i=n$. The closed sub-scheme of $\mathcal{S}_{\vec{d}}^{i}$ given by the condition $E_{i}=E'_{i}$ (c.f Diagram \ref{The moduli space : S}) is a clearly a Hyperquot scheme of dimension dim$(i)$ and contains the open in $\mathcal{S}_{\vec{d}}^{i}$\,, given by the condition
    \[
    p\notin \text{Supp }\frac{E_{i-1}}{E_{i}}.
    \] 
    The complement of this component, given by the condition $E_{i}\neq E'_{i}$ is also smooth of dimension dim$(i)$ thanks to Lemma \ref{Lemma excision argument one} and Proposition \ref{Geometry of moduli space U tilde}. This shows that $\mathcal{S}_{\vec{d}}^{i}$ has dimension equal to dim$(i)$ and is reduced on a dense open. The proof that $\mathcal{S}_{\vec{d}}^{i}$ is l.c.i. is left to the reader since it is similar to the proof of $\mathcal{Y}_{\vec{d}}^{i-1,i}$ being l.c.i. (cf. Proposition \ref{geometry of Y}). Since $\mathcal{S}_{\vec{d}}^{i}$ is l.c.i., we also obtain that $\mathcal{S}_{\vec{d}}^{i}$ is reduced due to the $R_{0}+S_{1}$ criterion for reducedness. 
\end{proof}

\begin{prop}\label{Geometry of moduli space $t$}
    The moduli space $\mathcal{T}_{\vec{d}}^{i}$ is reduced, l.c.i. and consists of two components of dimension $rd_{n}$ if $i< n$. If $i=n$, then it is smooth of dimension of dimension $r(d_{n}+1)$.
\end{prop}
\begin{proof}
    Note that when $i=n$, the space ${\mathcal{T}}_{\vec{d}}^{i}$ is precisely the space ${\mathfrak{P}}_{\vec{f}}$, where $\vec{f}$ is the $(n+1)-$tuple $(d_{1},...,d_{n-1},d_{n}-1,d_{n})$. Therefore, the current proposition in the case $i=n$ follows from Proposition \ref{proposition on smoothness and dimension of moduli space P}. We omit the proof of this proposition in the case $i\neq n$, since it is similar to the proof of Proposition \ref{Geometry of moduli space $S$}. 
\end{proof}

\subsection{Tautological line bundles on $\tilde{\mathcal{U}}_{\vec{d}}^{i}\,$, $\mathcal{S}_{\vec{d}}^{i}$ and $\mathcal{T}_{\vec{d}}^{i}$}
The spaces  $\tilde{\mathcal{U}}_{\vec{d}}^{i}\,$, $\mathcal{S}_{\vec{d}}^{i}$ and $\mathcal{T}_{\vec{d}}^{i}$ carry several tautological line bundles;\\\\ 
On $\tilde{\mathcal{U}}_{\vec{d}}^{i}\,$, we have line bundles 
\[
\mathscr{L}_{1},\mathscr{L}_{2},{\mathscr{L}}^{'}_{1}, {\mathscr{L}}^{'}_{2},
\]
 whose respective fibres over a closed point of the moduli space given by the Diagram \eqref{The moduli space : U tilde}, are the length one quotients $F_{i}/E_{i}, E_{i}/\tilde{F}_{i},F_{i}/E^{'}_{i},E^{'}_{i}/\tilde{F}_{i}$. We denote the respective first Chern classes of these bundles as
 \[
 \lambda_{1},\lambda_{2},{\lambda}^{'}_{1},{\lambda}^{'}_{2}.
 \]
On $\mathcal{S}_{\vec{d}}^{i}$, we have line bundles 
\[\mathscr{L}_{2},\mathscr{L}^{'}_{2},\]
 whose respective fibres over a closed point of the moduli given by the Diagram \eqref{The moduli space : S}$\,$, are the length one quotients $E_{i}/\tilde{F}_{i},E^{'}_{i}/\tilde{F}_{i}$. We denote the respective first Chern classes of these bundles as
 \[
 \lambda_{2},\lambda^{'}_{2}.
 \]
 On $\mathcal{  T}_{\vec{d}}^{i}$, we have line bundles 
\[
{\mathscr{L}}_{1}, \mathscr{L}^{'}_{1}, \]
 whose respective fibres over a closed point of the moduli given by the diagram (\ref{The moduli space : T}), are the length one quotients $F_{i}/E_{i},F_{i}/E^{'}_{i}$. Finally, we denote the respective first Chern classes of these bundles as
 \[
 {\lambda}_{1},{\lambda}^{'}_{1}.
 \]
Note that we have the following morphisms of line bundles on $\mathcal{U}_{\vec{d}}^{i,j}$:
\[
\mathscr{L}_{2}\xrightarrow{s_{1}}{\mathscr{L}}^{'}_{1} \text{ and }{\mathscr{L}}^{'}_{2}\xrightarrow{s_{2}}\mathscr{L}_{1}.
\]
These morphisms vanish precisely on the locus of $\tilde{\mathcal{U}}_{\vec{d}}^{i}$ given by the condition $E_{i}=E_{i}^{'}$ (cf. diagram \ref{The moduli space : U tilde}), and hence restrict to isomorphisms on the locus $\{E_{i}\neq E_{i}^{'}\}$. This implies the following lemma that we record for later use:

\begin{lem}\label{Lemma excision argument two}
    Let $U$ be the open set of $\tilde{\mathcal{U}}_{\vec{d}}^{i}$ given by the condition $E_{i}\neq E_{i}^{'}$, then we have the following equality of classes in $H^{*}(U):$
    \[
    {\lambda_{2}}_{|_{U}}={\lambda^{'}_{1}}_{|_{U}} \text{ and }{\lambda_{2}^{'}}_{|_{U}}={\lambda^{}_{1}}_{|_{U}}.
    \]
\end{lem}
\subsection{The formalism of correspondences}\label{subsection:formalism of correspondences}
In this brief subsection, we record some useful standard facts on the formalism of composition of correspondences in cohomology. Many of the spaces that were studied in this section are naturally supports of cohomology classes which arise from this formalism.\\\\
    Let $X$ and $Y$ be smooth projective varieties, let $\gamma_1\in H^{*}(X\times Y)$ and let $p_{_X},p_{_Y}
    $ be the projection maps from $X\times Y$ to $X$ and $Y$ respectively. Then $\gamma_1$ gives rise to a linear map $T^{X,Y}_{\gamma_1}:H^{*}(X)\rightarrow H^{*}(Y)$ defined by the formula
    \[
    T^{X,Y}_{\gamma_1}:= {p_{_Y}}_{*}(\gamma_1\cdot p_{_X}^{*}(-)).
    \]
    \begin{defn}\label{correspondence}
        We say that the linear map $T^{X,Y}_{\gamma_1}$ is given as a correspondence by the class $\gamma_1$ on $X\times Y.$
    \end{defn}
    \noindent Now, let $Z$ be another smooth projective variety and let $\gamma_2\in H^{*}(Y\times Z)$. As before, the class $\gamma_2$ gives rise to a linear map $T^{Y,Z}_{\gamma_2}:H^{*}(Y)\rightarrow H^{*}(Z)$. We describe the composition $T^{Y,Z}_{\gamma_2}\circ T^{X,Y}_{\gamma_1}:H^{*}(X)\rightarrow H^{*}(Z)$,
for which we consider the following diagram:
     \[
    \begin{tikzcd}
    &X\times Y\times Z \arrow[swap]{dl}{p_{_{XY}}}\arrow{d}{p_{_{XZ}}}\arrow{dr}{p_{_{YZ}}}&\\
    X\times Y&  X\times Z   &Y\times Z
\end{tikzcd}.
\]
The following proposition is a standard fact:
\begin{prop} \label{class of correspondence}
The composition $T^{Y,Z}_{\gamma_2}\circ T^{X,Y}_{\gamma_1}$ is equal to $T^{X,Z}_\alpha$ where the correspondence $\alpha\in H^{*}(X\times Z)$ is defined as:
\[
    \alpha:= {p_{_{XZ}}}_{*}(p_{_{XY}}^{*}(\gamma_1)\cdot p_{_{YZ}}^{*}(\gamma_2)).
    \]
\end{prop}
\noindent We would often be in a setting where
$i_1:Z_1\hookrightarrow X \times Y\times Z$ and $i_2:Z_2\hookrightarrow X \times Y\times Z$ are closed embeddings and there exist cohomology classes $c_1\in H^{*}(Z_1),c_2\in H^{*}(Z_2)$ such that $p_{_{XY}}^{*}(\gamma_1)={i_{1}}_{*}(c_1)$ and $ p_{_{YZ}}^{*}(\gamma_2)={i_{2}}_{*}(c_2)$. Therefore we would have the following fibre square:

\[\begin{tikzcd}
	&& {Z_{1}} \\
	{Z_{1}\cap Z_{2}} &&&& {X\times Y\times Z} \\
	&& {Z_{2}}
	\arrow["{i_1}",hook, from=1-3, to=2-5]
	\arrow["{j_1}", hook, from=2-1, to=1-3]
	\arrow["{j_2}"', hook, from=2-1, to=3-3]
	\arrow["{i_2}"',hook, from=3-3, to=2-5]
\end{tikzcd}\]

\noindent In practise, the following proposition helps us compute the class $p_{_{XY}}^{*}(\gamma_1)\cdot p_{_{YZ}}^{*}(\gamma_2)={i_{1}}_{*}(c_1)\cdot{i_{2}}_{*}(c_2)$:

\begin{prop}\label{composition of correspondences}
    In the setting of the above diagram, let $k:=i_1 \circ j_1=i_2 \circ j_2$. If $i_1$ and $i_2$ are regular embeddings and 
    $\text{Codim}_{X\times Y\times Z}(Z_1\cap Z_2)=  Codim_{X\times Y\times Z}(Z_1)+ Codim_{X\times Y\times Z}(Z_2)$ then $j_1,j_2$ are also regular embeddings and 
    \[
    {i_{1}}_{*}(c_1)\cdot{i_{2}}_{*}(c_2)=k_{*}(j_{1}^{*}(c_1)\cdot j_{2}^{*}(c_2)).
    \]
\end{prop}
\noindent This proposition is a straightforward consequence of the projection formula and base change for regular embeddings (cf. Theorem 6.2 (a) in \cite{Ful}).

\section{The Yangian Action}\label{section: The Yangian action}

\begin{defn}\label{Shifted Yangian defn}
    We fix positive integers $r$ and $n$. With $\hbar$ being a formal parameter, we define the shifted Yangian $Y_{\hbar}^{r}({\mathfrak{sl}_{n+1}})$ as the $\mathbb{Q}[\hbar]-$algebra generated by the formal symbols 
    \[
    \{e_{k}^{(v)},f_{k}^{(v)},m_k^{(u)}\}_{k\in\{1,...,n\},\,v\geq 0,\,u\in\{1,...,r\}}
    \]
    subject to the following relations for all $i,j\in\{1,...,n\}$ and $s,t\geq0$:
    
\begin{equation}\tag{R1}\label{mm commutation}
    \left[m_{i}^{(s)}, m_{j}^{(t)}\right]=0
\end{equation}

\begin{equation}\tag{R2}\label{ee commutation 1}
    \frac{\left[e_{i}^{(s+1)}, e_{i}^{(t)}\right]}{\hbar}-\frac{\left[e_{i}^{(s)}, e_{i}^{(t+1)}\right]}{\hbar}=e_{i}^{(s)}e_{i}^{(t)}+e_{i}^{(t)}e_{i}^{(s)}.
\end{equation}

\begin{equation}\tag{R3}\label{ee commutation 2}
    \frac{\left[e_{i}^{(s)}, e_{i-1}^{(t+1)}\right]}{\hbar}-\frac{\left[e_{i}^{(s+1)}, e_{i-1}^{(t)}\right]}{\hbar}= e_{i}^{(s)} e_{i-1}^{(t)}.
\end{equation}

\begin{equation}\tag{R4}\label{e commutation 3}
    \left[e_{i}^{(s)}, e_{j}^{(t)}\right]=0  \,\,\,\text{ for }|i-j|>1.
\end{equation}

\begin{equation}\tag{R5}\label{e Serre relations}
    \left[e_{i}^{(s)},\left[e_{i}^{(t)}, e_{j}^{(u)}\right]\right]+\left[e_{i}^{(t)},\left[e_{i}^{(s)}, e_{j}^{(u)}\right]\right]=0. \,\,\,\text{ for }|i-j|=1.
\end{equation}

\begin{equation} \tag{R6}\label{em relations}
    \frac{\left[m_{i}^{(s)}, e_{j}^{(t)}\right]}{\hbar}=-\delta_{i,j}\left(\sum_{l=0}^{s-1}e_{j}^{(t+l)}m_{i}^{(s-l-1)}(-1)^{l+1}\right).
\end{equation}

\begin{equation}\tag{R7}\label{ff commutation 1}
    \frac{\left[ f_{i}^{(t)},f_{i}^{(s+1)}\right]}{\hbar}-\frac{\left[f_{i}^{(t+1)},f_{i}^{(s)}\right]}{\hbar}=f_{i}^{(s)}f_{i}^{(t)}+f_{i}^{(t)}f_{i}^{(s)}.
\end{equation}

\begin{equation}\tag{R8}\label{ff commutation 2}
    \frac{\left[f_{i}^{(s+1)}, f_{i-1}^{(t)}\right]}{\hbar}-\frac{\left[f_{i}^{(s)}, f_{i-1}^{(t+1)}\right]}{\hbar}= f_{i}^{(t)} f_{i-1}^{(s)}
\end{equation}

\begin{equation}\tag{R9}\label{f commutation 3}
    \left[f_{i}^{(s)}, f_{j}^{(t)}\right]=0  \,\,\,\text{ for }|i-j|>1.
\end{equation}

\begin{equation}\tag{R10}\label{f Serre relations}
    \left[f_{i}^{(s)},\left[f_{i}^{(t)}, f_{j}^{(u)}\right]\right]+\left[f_{i}^{(t)},\left[f_{i}^{(s)}, f_{j}^{(u)}\right]\right]=0. \,\,\,\text{ for }|i-j|=1.
\end{equation}

\begin{equation} \tag{R11}\label{fm relations}
    \frac{\left[ f_{j}^{(t)},m_{i}^{(s)}\right]}{\hbar}=-\delta_{i,j}\left(\sum_{l=0}^{s-1}m_{i}^{(s-l-1)}f_{j}^{(t+l)}(-1)^{l+1}\right).
\end{equation}

\begin{equation}\tag{R12}\label{ef relations}
    \frac{\left[e_{i}^{(s)}, f_{j}^{(t)}\right]}{\hbar}=\delta_{i,j}\cdot h_{_{i}}^{(s+t-\delta_{i,n}\cdot r+1)}.
    \end{equation}
Where 
\[
h_i(z):=\sum\limits_{l=0}^{\infty}\frac{h_{i}^{(l)}}{z^{l+\delta_{i,n}\cdot r}}=\frac{m_{i+1}(z)m_{i-1}(z+\hbar)}{m_{i}(z)m_{i}(z+\hbar)}
\]
Here we have
\[
m_i(z):=\sum\limits_{s=0}^{r}z^{r-s}(-1)^{s}m_i^{(s)} \text{ for } i\in\{1,...,n\} \text{, }
\]
where $m_{i}^{(0)}=1$ and $m_{0}(z)$ is a fixed monic degree $r$ polynomial. We also make the conventions that $m_{i}(z)=1$ for $i>n$ and $h_{i}^{(l)}=0$ for all $i\geq 0$ and $l<0$.
\end{defn}

\begin{rem}
    The shift in the name ``shifted Yangian of $\mathfrak{sl}_{n+1}$" is reflected in the $e_{i},f_{j}$ commutation relation \eqref{ef relations}, in the case $i=j=n$. As a consequence, $e_{n}^{{(0)}}, f_{n}^{(0)}$ and $[e_{n}^{{(0)}}, f_{n}^{(0)}]$ is not an $\mathfrak{sl}_{2}-$triple for all $r\geq 1$.
\end{rem}

\noindent In this section, we will show that the operators \ref{e,f,m}, namely:
$$e_{k}^{(v)},f_{k}^{(v)},m_{k}^{(u)}:H^{*}(F^{n}\text{Quot}(V))\rightarrow H^{*}(F^{n}\text{Quot}(V)\times C)$$
satisfy the relations of Theorem \ref{Yangian relations satisfied}. This would readily imply that the above operators define a $\mathbb{Q}-$ linear map 
\[
Y_{\hbar}^{r}({\mathfrak{sl}_{n+1}})\rightarrow \text{Hom}(H^{*}(F^{n}\text{Quot}(V)), H^{*}(F^{n}\text{Quot}(V)\times C))
\]
which defines an action of the Yangian $Y_{\hbar}^{r}({\mathfrak{sl}_{n+1}})$ on $H^{*}(F^{n}\text{Quot}(V))$ in the following sense:
\begin{defn}\label{Action defn}
 An action of $Y_{\hbar}^{r}({\mathfrak{sl}_{n+1}})$ on $H^{*}(F^n\text{Quot}(V))$ is a $\mathbb{Q}$-linear map:
\[
Y_{\hbar}^{r}({\mathfrak{sl}_{n+1}})\xrightarrow{\Phi} \text{Hom}(H^{*}(F^n\text{Quot}(V)),H^{*}(F^n\text{Quot}(V)\times C))
\]
satisfying the following properties for all $x,y\in Y_{\hbar}^{r}({\mathfrak{sl}_{n+1}})$:
\begin{enumerate}
    \item $\Phi(1)=\rho^{*}$, where $\rho:F^n\text{Quot}(V)\times C\rightarrow F^n\text{Quot}(V)$ is the natural projection map.
    \item $\Phi(xy)$ is the composition:
    \begin{multline*}
    H^{*}(F^n\text{Quot}(V))\xrightarrow{\Phi(y)}H^{*}(F^n\text{Quot}(V)\times C)\\
    \xrightarrow{\Phi(x)\otimes Id_{C}}H^{*}(F^n\text{Quot}(V)\times C\times C)\xrightarrow{Id\otimes \Delta^{*}}H^{*}(F^n\text{Quot}(V)\times C).
     \end{multline*}
    Here $\Delta:C\rightarrow C \times C$ is the diagonal embedding.\\
    
    \item $\Phi(\hbar x)$ is the composition:
\[
H^{*}(F^n\text{Quot}(V))\xrightarrow{\Phi(x)}H^{*}(F^n\text{Quot}(V)\times C)\xrightarrow{Id\otimes (\cdot K_C)}H^{*}(F^n\text{Quot}(V)\times C).
\]

\item $\Delta_{*}\left(\frac{\Phi[x,y]}{-\hbar}\right)$ coincides with the difference of the following compositions: 

\[
H^{*}(F^n\text{Quot}(V))\xrightarrow{\Phi(y)}H^{*}(F^n\text{Quot}(V)\times C)\xrightarrow{\Phi(x)\otimes Id_{C}}H^{*}(F^n\text{Quot}(V)\times C\times C)
\]
and 
\begin{multline*}
H^{*}(F^n\text{Quot}(V))\xrightarrow{\Phi(x)}H^{*}(F^n\text{Quot}(V)\times C)\\
\xrightarrow{\Phi(y)\otimes Id_{C}}H^{*}(F^n\text{Quot}(V)\times C\times C)\xrightarrow{Id\otimes \,\text{s}^{*}}H^{*}(F^n\text{Quot}(V)\times C\times C).
\end{multline*}
\newline
Above, we have $s:C\times C\rightarrow C\times C\,,(a,b)\mapsto (b,a)$.\\
\end{enumerate}
\end{defn}

\subsection{Proof of Theorem \ref{Yangian relations satisfied}}
We will now prove Theorem \ref{Yangian relations satisfied} by verifying the relations \ref{mm commutation action} - \ref{f Serre relations action}.\\\\
Note that the relation \ref{mm commutation action} clearly holds since the operators $m_{i}^{(s)}$ are given by multiplication with tautological classes and hence commute.\\\\ Moreover, the relations \ref{e commutation 3 action} and \ref{f commutation 3 action} are straightforward to verify, since Propositions \ref{class of correspondence} and \ref{composition of correspondences} readily imply that the operator $e^{(s)}_{i}\circ {e}^{(t)}_{j}$ (respectively $f^{(s)}_{i}\circ {f}^{(t)}_{j}$) is given by the same correspondence as the operator $e^{(t)}_{j}\circ {e}^{(s)}_{i}$ (respectively $f^{(t)}_{j}\circ {f}^{(s)}_{i}$) whenever $|i-j|>1$.\\\\
We will now verify the rest of the relations of Theorem \ref{Yangian relations satisfied} over the next few subsections. 
\subsection{The $e_{i}^{(*)},e_{i-1}^{(*)}$ and $f_{i}^{(*)},f_{i-1}^{(*)}$ commutation relations: \eqref{ee commutation 2 action} and \eqref{ff commutation 2 action}}
Let us verify the relation \ref{ee commutation 2 action}. We omit the verification of the relation \ref{ff commutation 2 action} since it is identical.
The geometry of the moduli space $\mathcal{Y}_{\vec{d}}^{i-1,i}$ (see subsection \ref{subsection:moduli space Y} for definitions and notation) would be key in this proof.\\\\
The following two lemmas follow easily from propositions \ref{class of correspondence}, \ref{composition of correspondences} and the dimension estimate in proposition \ref{geometry of Y}. To state the lemmas let us recall some notation. Let $\vec{d}=(d_1,...,d_{n})$ be an $n-$tuple of non-negative integers and for $i>1$, let $\vec{d}_{i-1,i}$ be the $n-$tuple $(d_1,...,d_{i-2},d_{i-1}+1,d_{i}+1,d_{i+1},...,d_{n})$ and let $\vec{d}_{i-1,i}^{\bullet}$ to be the $(n+2)-$ tuple $(d_1,...,d_{i-1},d_{i-1}+1,d_{i},d_{i}+1,d_{i+1},...,d_{n})$. We remark that the statements in this subsection will be obviously true whenever the $n-$tuple $\vec{d}$ is such that the moduli space $\mathcal{Y}_{\vec{d}}^{i-1,i}$ is a Hyperquot scheme (the case $d_{i-1}=d_{i}$) or is empty. As in the statement of Theorem \ref{Yangian relations satisfied}, we will also make the convention that the operators with superscripts $(s)$ and $(t)$ will contribute to the first and second factor of $C\times C$ respectively.

\begin{lem}\label{lemma on the composition e_ie_i-1}
The linear operator
\[
e_{i}^{(s)}e^{(t)}_{i-1}:H^{*}(F^{n}\text{Quot}_{\vec{d}}(V))\rightarrow H^{*}(F^{n}\text{Quot}_{\vec{d}_{i-1,i}}(V)\times C\times C)
\]
is given as a correspondence by the class $$k_{*}\left(\lambda_{i}^{s}\lambda_{i-1}^{t}\cdot\left[F^{n+2}\text{Quot}_{\vec{d}_{i-1,i}^{\bullet}}\right]\right)$$ on $F^{n}\text{Quot}_{\vec{d}}(V)\times  F^{n}\text{Quot}_{\vec{d}_{i-1,i}}(V)\times C\times C$.
    
\end{lem}

\noindent Similarly, we have:
\begin{lem}\label{lemma on the composition e_i-1e_i}
The linear operator
\[
e^{(t)}_{i-1}e_{i}^{(s)}:H^{*}(F^{n}\text{Quot}_{\vec{d}}\,(V))\rightarrow H^{*}(F^{n}\text{Quot}_{\vec{d}_{i-1,i}}(V)\times C\times C)
\]
is given as a correspondence by the class $$k_{*}\left(\lambda_{i}^{s}\lambda_{i-1}^{t}\right)$$ on $F^{n}\text{Quot}_{\vec{d}}(V)\times  F^{n}\text{Quot}_{\vec{d}_{i-1,i}}(V)\times C\times C$.
\end{lem}
\noindent Then the following corollary follows from Propositions \ref{geometry of Y} and \ref{lci ness of Y} and the above two lemmas; \ref{lemma on the composition e_ie_i-1} and \ref{lemma on the composition e_i-1e_i}:
\begin{cor}\label{corollary on the operator [e_i,e_i-1]}
The linear operator
\[
    \left[e_{i}^{(s)}, e_{i-1}^{(t)}\right]:H^{*}(F^{n}\text{Quot}_{\vec{d}}(V))\rightarrow H^{*}(F^{n}\text{Quot}_{\vec{d}_{i-1,i}}(V)\times C\times C)
\]
is given as a correspondence by the class $$-k_{*}\left(\lambda_{i}^{s}\lambda_{i-1}^{t}\cdot\left[\mathcal{Z}_{\vec{d}}^{i-1,i}\right]\right)$$ on $F^{n}\text{Quot}_{\vec{d}}(V)\times F^{n}\text{Quot}_{\vec{d}_{i-1,i}}(V)\times C\times C$.
\end{cor}
\noindent We are now ready to verify relation \ref{ee commutation 2 action}.\\
\begin{prop}\label{e_i,e_{i-1} relation proof} The formula
\[
{\left[e_{i}^{(s)}, e_{i-1}^{(t+1)}\right]}-{\left[e_{i}^{(s+1)}, e_{i-1}^{(t)}\right]}= -\delta \cdot e_{i}^{(s)} e_{i-1}^{(t)}.
\]
holds as an equality of operators
\[
H^{*}(F^{n}\text{Quot}(V))\rightarrow H^{*}(F^{n}\text{Quot}(V)\times C\times C).
\]
As indicated in the statement of theorem\ref{Yangian relations satisfied}, we make the convention that the operators labeled $e_{i}^{(s)}$ and $e_{i}^{(s+1)}$ contribute to the first factor of $C\times C$ and the operators labeled $e_{i-1}^{(t)}$ and $e_{i-1}^{(t+1)}$ contribute to the second factor of $C\times C$.
\newline
\begin{proof}
It follows from corollary \ref{corollary on the operator [e_i,e_i-1]} that the L.H.S., which is the operator $$\left[e_{i}^{(s)}, e_{i-1}^{(t+1)}\right]-\left[e_{i}^{(s+1)}, e_{i-1}^{(t)}\right]$$ is given as a correspondence by the class $$-k_{*}\left(\lambda_{i}^{s}\lambda_{i-1}^{t}\cdot\left[\mathcal{Z}_{\vec{d}}^{i-1,i}\right]\cdot(\lambda_{i-1}-\lambda_{i})\right).$$ By Lemma \ref{key lemma to prove the e_i e_i-1 commutation}, we can re-write this expression as $$-k_{*}\left(\lambda_{i}^{s}\lambda_{i-1}^{t}\cdot\delta\cdot\left[F^{n+2}\text{Quot}_{\vec{d}_{i-1,i}^{\bullet}}\right]\right).$$ By Lemma \ref{lemma on the composition e_ie_i-1}, this is precisely the class on $F^{n}\text{Quot}_{\vec{d}}(V)\times F^{n}\text{Quot}_{\vec{d}_{i-1,i}}(V)\times C\times C$ which defines the linear operator $$-\delta\cdot e_{i}^{(s)} e_{i-1}^{(t)}$$ as a correspondence. This is precisely the R.H.S. of the formula \ref{ee commutation 2 action} and hence we are done.

\end{proof}

\end{prop}

\subsection{The cubic Serre relations: \eqref{e Serre relations action} and \eqref{f Serre relations action}}
We will now prove the identity 
\begin{equation}\label{e Serre relations i formula}
     \left[e_{j}^{(s)},\left[e_{j}^{(t)}, e_{j-1}^{(u)}\right]\right]+\left[e_{j}^{(t)},\left[e_{j}^{(s)}, e_{j-1}^{(u)}\right]\right]=0.
\end{equation}
The proof of the identity
\[
\left[e_{j}^{(s)},\left[e_{j}^{(t)}, e_{j+1}^{(u)}\right]\right]+\left[e_{j}^{(t)},\left[e_{j}^{(s)}, e_{j+1}^{(u)}\right]\right]=0
\]
is similar, the verifications of the relations \ref{f Serre relations action} are identical and hence we will omit them. The geometry of the moduli spaces $\mathcal{U}_{\vec{d}}^{j-1,j},\mathcal{W}_{\vec{d}}^{j-1,j}$ and $\mathcal{X}_{\vec{d}}^{j-1,j}$ (cf. Subsections \ref{subsection: moduli spaces U,X,W}-\ref{Subsection: Divisors} for definitions and notations) will play a key role in the proof.\\\\
First, let us recall some notations and conventions. Let $\vec{d}=(d_1,...,d_{n})$ be an $n-$tuple of non-negative integers. For $j=2,...,n$, also define $n-$tuples \\
$\vec{d'}_{j-1,j}:=(d_1,...,d_{j-2},d_{j-1}+1,d_{j}+1,d_{j+1},...,d_{n})$ and let $\vec{d''}_{j-1,j}:=\\
(d_1,...,d_{j-2},d_{j-1}+1,d_{j}+2,d_{j+1},...,d_{n})$. We remark that in the case when $\vec{d}$ or $\vec{d''}_{j-1,j}$ is not non-decreasing, then the statements in this subsection will become vacuously true as the moduli spaces $\mathcal{U}_{\vec{d}}^{j-1,j},\mathcal{W}_{\vec{d}}^{j-1,j}$ and $\mathcal{X}_{\vec{d}}^{j-1,j}$ will become empty.\\\\
Also, recall the morphisms \ref{E,H forgetful map for moduli space W},\ref{E,H forgetful map for moduli space X}, \ref{morphisms p on U,X,W to C} and \ref{morphisms q on U,X,W to C}. Using these maps let us define:
\begin{equation}\label{k_X^1}
    k^{1}_{X}:=k_{X}\times q^{X}\times p^{X}\times q^{X}  : \mathcal{X}_{\vec{d}}^{j-1,j}\rightarrow F^{n}\text{Quot}_{\vec{d}}(V)\times F^{n}\text{Quot}_{\vec{d''}_{j-1,j}}(V)\times C\times C\times C
\end{equation}
\begin{equation}\label{k_X^2}
    k^{2}_{X}:=k_{X}\times q^{X}\times q^{X}\times p^{X} :\mathcal{X}_{\vec{d}}^{j-1,j}\rightarrow F^{n}\text{Quot}_{\vec{d}}(V)\times F^{n}\text{Quot}_{\vec{d''}_{j-1,j}}(V)\times C\times C\times C
\end{equation}
\begin{equation}\label{k_W^1}
    k^{1}_{W}:=k_{W}\times q^{W}\times p^{W}\times q^{W} :\mathcal{W}_{\vec{d}}^{j-1,j}\rightarrow F^{n}\text{Quot}_{\vec{d}}(V)\times F^{n}\text{Quot}_{\vec{d''}_{j-1,j}}(V)\times C\times C\times C
\end{equation}
and
\begin{equation}\label{k_W^2}
    k^{2}_{W}:=k_{W}\times q^{W}\times q^{W}\times p^{W}:\mathcal{W}_{\vec{d}}^{j-1,j}\rightarrow F^{n}\text{Quot}_{\vec{d}}(V)\times F^{n}\text{Quot}_{\vec{d''}_{j-1,j}}(V)\times C\times C\times C.
\end{equation}
As a first step we would like to express the operators involved in \eqref{e Serre relations i formula} in terms of correspondences. Due to Corollary \ref{corollary on the operator [e_i,e_i-1]} we know how to express the commutator $[e_{j}^{(s)}, e_{j-1}^{(u)}]$ as a correspondence. Then the following lemmas become straightforward consequences of Propositions \ref{class of correspondence} and \ref{composition of correspondences} on the formalism of correspondences and the dimension counts in Propositions \ref{Geometry of moduli space $X$} and \ref{Geometry of moduli space $W$}.\\\\
As in the statement of Theorem \ref{Yangian relations satisfied}, we also make the convention that the operators $e_{j-1}^{(u)},e_{j}^{(s)}$ and $e_{j}^{(t)}$ contribute to the first, second and third factor of $C\times C\times C$ respectively. 

\begin{lem}\label{corollary on the operator e_is[e_i,e_i-1]}
 The linear operator
\[
    e_{j}^{(s)}\left[e_{j}^{(t)}, e_{j-1}^{(u)}\right]:H^{*}(F^{n}\text{Quot}_{\vec{d}}(V))\rightarrow H^{*}(F^{n}\text{Quot}_{\vec{d''}_{j-1,j}}(V)\times C\times C\times C)
\]
is given as a correspondence by the class $$-{k^{1}_{X}}_{*}\left(\lambda_{1}^{t}\lambda_{2}^{s}\lambda_{3}^{u}\right)$$ on $F^{n}\text{Quot}_{\vec{d}}(V)\times F^{n}\text{Quot}_{\vec{d''}_{j-1,j}}(V)\times C\times C\times C$.
\newline
\end{lem}

\begin{lem}\label{corollary on the operator e_it[e_i,e_i-1]}
 The linear operator
\[
    e_{j}^{(t)}\left[e_{j}^{(s)}, e_{j-1}^{(u)}\right]:H^{*}(F^{n}\text{Quot}_{\vec{d}}(V))\rightarrow H^{*}\left(F^{n}\text{Quot}_{\vec{d''}_{j-1,j}}(V)\times C\times C\times C\right)
\]
is given as a correspondence by the class $$-{k^{2}_{X}}_{*}\left(\lambda_{1}^{s}\lambda_{2}^{t}\lambda_{3}^{u}\right)$$ on $F^{n}\text{Quot}_{\vec{d}}(V)\times F^{n}\text{Quot}_{\vec{d''}_{j-1,j}}(V)\times C\times C\times C$.
\newline
\end{lem}

\begin{lem}\label{corollary on the operator [e_i,e_i-1]e_i(s)}
 The linear operator
\[
    \left[e_{j}^{(t)}, e_{j-1}^{(u)}\right] e_{j}^{(s)}:H^{*}\left(F^{n}\text{Quot}_{\vec{d}}(V)\right)\rightarrow H^{*}\left(F^{n}\text{Quot}_{\vec{d''}_{j-1,j}}(V)\times C\times C\times C\right)
\]
is given as a correspondence by the class $$-{k^{1}_{W}}_{*}\left(\tilde{\lambda}_{1}^{s}\tilde{\lambda}_{2}^{t}{\lambda}_{3}^{u}\right)$$ on $F^{n}\text{Quot}_{\vec{d}}(V)\times F^{n}\text{Quot}_{\vec{d''}_{j-1,j}}(V)\times C\times C\times C$.
\newline
\end{lem}

\begin{lem}\label{corollary on the operator [e_i,e_i-1]e_i(t)}
 The linear operator
\[
    \left[e_{j}^{(s)}, e_{j-1}^{(u)}\right] e_{j}^{(t)}:H^{*}(F^{n}\text{Quot}_{\vec{d}}(V))\rightarrow H^{*}(F^{n}\text{Quot}_{\vec{d''}_{j-1,j}}(V)\times C\times C\times C)
\]
is given as a correspondence by the class $$-{k^{2}_{W}}_{*}\left(\tilde{\lambda}_{1}^{t}\tilde{\lambda}_{2}^{s}{\lambda}_{3}^{u}\right)$$ on $F^{n}\text{Quot}_{\vec{d}}(V)\times F^{n}\text{Quot}_{\vec{d''}_{j-1,j}}(V)\times C\times C\times C$.
\newline
\end{lem}

\begin{prop}\label{Serre relations proof }
The formula
\[
 \left[e_{j}^{(s)},\left[e_{j}^{(t)}, e_{j-1}^{(u)}\right]\right]+\left[e_{j}^{(t)},\left[e_{j}^{(s)}, e_{j-1}^{(u)}\right]\right]=0
\]
holds as an equality of operators
\[
H^{*}(F^{n}\text{Quot}(V))\rightarrow H^{*}(F^{n}\text{Quot}(V)\times C\times C\times C)
\]
for all $j=2,...,n$. Here, we make the convention that the operators $e_{j-1}^{(u)},\,e_{j}^{(s)}$ and $e_{j}^{(t)}$ in the above equation contribute to the first, second, and third factor of $C\times C\times C$ respectively.
\end{prop}

\begin{proof}
    From Lemmas \ref{corollary on the operator e_is[e_i,e_i-1]}, \ref{corollary on the operator e_it[e_i,e_i-1]}, \ref{corollary on the operator [e_i,e_i-1]e_i(s)} and \ref{corollary on the operator [e_i,e_i-1]e_i(t)}, it is enough to show that 

    \begin{equation}\label{What must be shown in the proof of the Serre relations}
        -{k^{2}_{X}}_{*}\left(\lambda_{1}^{s}\lambda_{2}^{t}\lambda_{3}^{u}\right)-{k^{1}_{X}}_{*}\left(\lambda_{1}^{t}\lambda_{2}^{s}\lambda_{3}^{u}\right)+{k^{1}_{W}}_{*}\left(\tilde{\lambda}_{1}^{s}\tilde{\lambda}_{2}^{t}{\lambda}_{3}^{u}\right)+{k^{2}_{W}}_{*}\left(\tilde{\lambda}_{1}^{t}\tilde{\lambda}_{2}^{s}{\lambda}_{3}^{u}\right)=0.
    \end{equation}
    \newline
    on $F^{n}\text{Quot}_{\vec{d}}(V)\times F^{n}\text{Quot}_{\vec{d''}_{j-1,j}}(V)\times C\times C\times C$.\\\\
    Recall the birational morphisms $ \eta_{\vec{d}}^{j-1,j}$ and $ \theta_{\vec{d}}^{j-1,j}$ (cf. equation \ref{birational maps : eta and theta}) and note that 
    \[
     k_{X}\circ\eta_{\vec{d}}^{j-1,j}=  k_{W}\circ \theta_{\vec{d}}^{j-1,j},
    \]
    as maps between  $\mathcal{U}_{\vec{d}}^{j-1,j}$ and $F^{n}\text{Quot}_{\vec{d}}(V)\times F^{n}\text{Quot}_{\vec{d''}_{j-1,j}}(V)$. This leads us to define morphisms $$\alpha^{1},\alpha^{2}:\mathcal{U}_{\vec{d}}^{j-1,j}\rightarrow F^{n}\text{Quot}_{\vec{d}}(V)\times F^{n}\text{Quot}_{\vec{d''}_{j-1,j}}(V)\times C\times C\times C$$ as:
    \[
    \alpha^{1}:=k_{X}^{1}\circ\eta_{\vec{d}}^{j-1,j}= k_{W}^{1}\circ\theta_{\vec{d}}^{j-1,j} \text{ and }\alpha^{2}:=k_{X}^{2}\circ\eta_{\vec{d}}^{j-1,j}= k_{W}^{2}\circ\theta_{\vec{d}}^{j-1,j}.
    \]
    Since $\eta_{\vec{d}}^{j-1,j}$ and $\theta_{\vec{d}}^{j-1,j}$ are birational, it follows from the projection formula that we can re-write the expression  
    \[ 
    -{k^{2}_{X}}_{*}(\lambda_{1}^{s}\lambda_{2}^{t}\lambda_{3}^{u})-{k^{1}_{X}}_{*}(\lambda_{1}^{t}\lambda_{2}^{s}\lambda_{3}^{u})+{k^{1}_{W}}_{*}(\tilde{\lambda}_{1}^{s}\tilde{\lambda}_{2}^{t}{\lambda}_{3}^{u})+{k^{2}_{W}}_{*}(\tilde{\lambda}_{1}^{t}\tilde{\lambda}_{2}^{s}{\lambda}_{3}^{u})
    \]
   in equation \ref{What must be shown in the proof of the Serre relations}, as:
    \[
     -{ \alpha^{2}}_{*}(\lambda_{1}^{s}\lambda_{2}^{t}\lambda_{3}^{u})-{ \alpha^{1}}_{*}(\lambda_{1}^{t}\lambda_{2}^{s}\lambda_{3}^{u})+{ \alpha^{1}}_{*}(\tilde{\lambda}_{1}^{s}\tilde{\lambda}_{2}^{t}{\lambda}_{3}^{u})+{ \alpha^{2}}_{*}(\tilde{\lambda}_{1}^{t}\tilde{\lambda}_{2}^{s}{\lambda}_{3}^{u}).
    \]
    Therefore, we must show that
    \begin{equation}\label{What must be shown, part 2, in the proof of the Serre relations}
        {\alpha^{1}}_{*}(\tilde{\lambda}_{1}^{s}\tilde{\lambda}_{2}^{t}{\lambda}_{3}^{u}-\lambda_{1}^{t}\lambda_{2}^{s}\lambda_{3}^{u}) = {\alpha^{2}}_{*}(\lambda_{1}^{s}\lambda_{2}^{t}\lambda_{3}^{u}-\tilde{\lambda}_{1}^{t}\tilde{\lambda}_{2}^{s}{\lambda}_{3}^{u}).
    \end{equation}
    \newline
    To this end, let us recall the divisor $\text{Diag}$ of $\mathcal{U}_{\vec{d}}^{j-1,j}$ (cf. Subsection \ref{Subsection: Divisors}) and observe that 
    \[
    {\alpha^{1}}_{_{|_{\text{Diag}}}}={\alpha^{2}}_{_{|_{\text{Diag}}}}.
    \]
    So we define 
    \[
    \alpha:\text{Diag} \rightarrow F^{n}\text{Quot}_{\vec{d}}(V)\times F^{n}\text{Quot}_{\vec{d''}_{j-1,j}}(V)\times C\times C\times C,\,\alpha:={\alpha^{1}}_{_{|_{\text{Diag}}}}={\alpha^{2}}_{_{|_{\text{Diag}}}}
    \]
    and note that for any cohomology class $\gamma$ on $\mathcal{U}_{\vec{d}}^{j-1,j}$, the projection formula implies:
    \begin{equation}\label{observation 1 in proof of serre relations}
\alpha^{1}_{*}\left([\text{Diag}]\cdot\gamma\right)=\alpha_{*}\left(\gamma_{_{|_{\text{Diag}}}}\right)=\alpha^{2}_{*}\left([\text{Diag}]\cdot\gamma\right).
    \end{equation}
    
    \noindent We also make a key observation about the restriction of the tautological line bundles on $\mathcal{U}_{\vec{d}}^{j-1,j}$ to Diag, which just follows from the definition of Diag as the sub-scheme of $\mathcal{U}_{\vec{d}}^{j-1,j}$ given by the relation $F_{j}=F_{j}^{'}$ (cf. Diagrams \ref{The moduli space: U} and \ref{The moduli space: Diag} ):
\begin{equation}\label{observation 2 in proof of serre relations}
\mathscr{L}_{{1}_{{|}_{\text{Diag}}}}={\tilde{\mathscr{L}}}_{{1}_{{|}_{\text{Diag}}}} \text{ and } \mathscr{L}_{{2}_{{|}_{\text{Diag}}}}={\tilde{\mathscr{L}}}_{{2}_{{|}_{\text{Diag}}}}.
\end{equation}
Coming back to verifying equation \ref{What must be shown, part 2, in the proof of the Serre relations}, the L.H.S.
\[
{\alpha^{1}}_{*}\left(\tilde{\lambda}_{1}^{s}\tilde{\lambda}_{2}^{t}{\lambda}_{3}^{u}-\lambda_{1}^{t}\lambda_{2}^{s}\lambda_{3}^{u}\right)
\]
can be re-written using Proposition \ref{Diag in cohomology} as:
\[
{\alpha^{1}}_{*}\left(\left(\lambda_{2}+[\text{Diag}]\right)^{s}\tilde{\lambda}_{2}^{t}{\lambda}_{3}^{u}-\left(\tilde{\lambda}_{2}+[\text{Diag}]\right)^{t}\lambda_{2}^{s}\lambda_{3}^{u}\right).
\]
By doing a simple binomial expansion we see that this expression can be written in the form of equation \ref{observation 1 in proof of serre relations}. Explicitly, we have
\[
\text{L.H.S.}={\alpha^{1}}_{*}\left(\left(\sum\limits_{i=0}^{s}\binom{s}{i}\lambda_{2}^{s-i}[\text{Diag}]^{i}\right)\tilde{\lambda}_{2}^{t}{\lambda}_{3}^{u}-\left(\sum\limits_{j=0}^{t}\binom{t}{j}\tilde{\lambda}_{2}^{t-j}[\text{Diag}]^{j}\right)\lambda_{2}^{s}\lambda_{3}^{u}\right)
\]
\[
={\alpha^{1}}_{*}\left(\left(\sum\limits_{i=1}^{s}\binom{s}{i}\lambda_{2}^{s-i}[\text{Diag}]^{i}\right)\tilde{\lambda}_{2}^{t}{\lambda}_{3}^{u}-\left(\sum\limits_{j=1}^{t}\binom{t}{j}\tilde{\lambda}_{2}^{t-j}[\text{Diag}]^{j}\right)\lambda_{2}^{s}\lambda_{3}^{u}\right)
\]
\[
=\alpha_{*}\left(\left(\left(\sum\limits_{i=1}^{s}\binom{s}{i}\lambda_{2}^{s-i}[\text{Diag}]^{i-1}\right)\tilde{\lambda}_{2}^{t}{\lambda}_{3}^{u}-\left(\sum\limits_{j=1}^{t}\binom{t}{j}\tilde{\lambda}_{2}^{t-j}[\text{Diag}]^{j-1}\right)\lambda_{2}^{s}\lambda_{3}^{u}\right)
_{|_{\text{Diag}}}\right).
\]
Then by \ref{observation 2 in proof of serre relations}, we can switch $\lambda_{2}$ and $\tilde{\lambda}_{2}$ in the above expression:
\begin{multline*}
\text{L.H.S.}=\alpha_{*}\left(\left(\sum\limits_{i=1}^{s}\binom{s}{i}\tilde{\lambda}_{2}^{s-i}[\text{Diag}]^{i-1}\right)\lambda_{2}^{t}{\lambda}_{3}^{u}\right.\\\left.-\left(\sum\limits_{j=1}^{t}\binom{t}{j}\lambda_{2}^{t-j}[\text{Diag}]^{j-1}\right)\tilde{\lambda}_{2}^{s}\lambda_{3}^{u}\right)
_{|_{\text{Diag}}}.
\end{multline*}
Again, by \ref{observation 1 in proof of serre relations}, this is equal to 
\[
\alpha^{2}_{*}\left(\left(\sum\limits_{i=1}^{s}\binom{s}{i}\tilde{\lambda}_{2}^{s-i}[\text{Diag}]^{i}\right)\lambda_{2}^{t}{\lambda}_{3}^{u}-\left(\sum\limits_{j=1}^{t}\binom{t}{j}\lambda_{2}^{t-j}[\text{Diag}]^{j}\right)\tilde{\lambda}_{2}^{s}\lambda_{3}^{u}\right)
\]
\[
=\alpha^{2}_{*}\left(\left(\sum\limits_{i=0}^{s}\binom{s}{i}\tilde{\lambda}_{2}^{s-i}[\text{Diag}]^{i}\right)\lambda_{2}^{t}{\lambda}_{3}^{u}-\left(\sum\limits_{j=0}^{t}\binom{t}{j}\lambda_{2}^{t-j}[\text{Diag}]^{j}\right)\tilde{\lambda}_{2}^{s}\lambda_{3}^{u}\right)
\]
\[
=\alpha^{2}_{*}\left(\left(\tilde{\lambda}_{2}+[\text{Diag}]\right)^{s}\lambda_{2}^{t}{\lambda}_{3}^{u}-\left(\lambda_{2}+[\text{Diag}]\right)^{t}\tilde{\lambda}_{2}^{s}\lambda_{3}^{u}\right).
\]
\newline
By Proposition \ref{Diag in cohomology}, this expression is equal to 
\[
\alpha^{2}_{*}\left(\lambda_{1}^{s}\lambda_{2}^{t}{\lambda}_{3}^{u}-\tilde{\lambda}_{1}^{t}\tilde{\lambda}_{2}^{s}\lambda_{3}^{u}\right),
\]
which is precisely the R.H.S. of \ref{What must be shown, part 2, in the proof of the Serre relations}, finishing the proof.
\end{proof}

\subsection{The $e_{j}^{(*)},e_{j}^{(*)}$ and $f_{j}^{(*)},f_{j}^{(*)}$ commutation relations: \eqref{ee commutation 1 action} and \eqref{ff commutation 1 action}}
We will now verify the relation \ref{ee commutation 1 action}. We omit the verification of the relation \ref{ff commutation 1 action} since it is identical.
We will crucially use the geometry of the moduli spaces $\mathcal{U}_{\vec{d}}^{j,j},\, \mathcal{X}_{\vec{d}}^{j,j}$ and $\mathcal{W}_{\vec{d}}^{j,j}$ (cf. Subsection \ref{subsection: U,X,W case:i=j} for definitions and notation) in this proof.\\\\
Let us set up notations and conventions for this subsection. Let $\vec{d}=(d_1,...,d_{n})$ be an $n-$tuple of non-negative integers. For $j=1,...,n$, we define the $n-$tuple $\vec{d''}_{j,j}:=(d_1,...,d_{j-1},d_{j}+2,d_{j+1},...,d_{n})$. In the case when $\vec{d}$ or $\vec{d''}_{j,j}$ is not non-decreasing, then the statements in this subsection will be vacuously true as the moduli spaces $\mathcal{U}_{\vec{d}}^{j,j},\mathcal{W}_{\vec{d}}^{j,j}$ and $\mathcal{X}_{\vec{d}}^{j,j}$ will become empty. \\\\
Following the convention of the statement of Theorem \ref{Yangian relations satisfied}, the operator $e_{j}^{(s)}$ contributes to the first factor of $C\times C$ and the operators $e_{j}^{(t)}$ contributes to the second factor of $C\times C$.\\\\
The following lemmas follow from Propositions \ref{class of correspondence} and \ref{composition of correspondences} and the dimension counts in propositions \ref{Geometry of moduli space $X$} and \ref{Geometry of moduli space $W$}. To state the lemmas let us also recall the maps \ref{morphisms q on U,X,W to C}, \ref{morphisms p on U,X,W to C}, \ref{E,H forgetful map for moduli space W} and \ref{E,H forgetful map for moduli space X}.

\begin{lem}\label{corollary on the operator ejej s,t}
The linear operator
\[
    e_{j}^{(s)}\circ e_{j}^{(t)}:H^{*}(F^{n}\text{Quot}_{\vec{d}}(V))\rightarrow H^{*}(F^{n}\text{Quot}_{\vec{d''}_{j,j}}(V)\times C\times C)
\]
is given as a correspondence by the class $${(k_{X}\times p^{X}\times q^{X})}_{*}\left(\lambda_{1}^{t}\lambda_{2}^{s}\right)$$ on $F^{n}\text{Quot}_{\vec{d}}(V)\times F^{n}\text{Quot}_{\vec{d''}_{j,j}}(V)\times C\times C$.
\newline
\end{lem}

\begin{lem}\label{corollary on the operator ejej t,s}
The linear operator
\[
    e_{j}^{(t)}\circ e_{j}^{(s)}:H^{*}(F^{n}\text{Quot}_{\vec{d}}(V))\rightarrow H^{*}(F^{n}\text{Quot}_{\vec{d''}_{j,j}}(V)\times C\times C)
\]
is given as a correspondence by the class $${(k_{W}\times p^{W}\times q^{W})}_{*}\left(\tilde{\lambda}_{1}^{s}\tilde{\lambda}_{2}^{t}\right)$$ on $F^{n}\text{Quot}_{\vec{d}}(V)\times F^{n}\text{Quot}_{\vec{d''}_{j,j}}(V)\times C\times C$.
\newline
\end{lem}
\noindent Now, let us observe that there is an equality of morphisms:
\begin{multline*}
\left(k_{X}\times p^{X}\times q^{X}\right)\circ \,\eta^{j,j}_{\vec{d}}=\left(k_{W}\times p^{W}\times q^{W}\right)\circ\,\theta^{j,j}_{\vec{d}}:\\
\,\mathcal{U}^{j,j}_{\vec{d}}
\rightarrow F^{n}\text{Quot}_{\vec{d}}(V)\times F^{n}\text{Quot}_{\vec{d''}_{j,j}}(V)\times C\times C
\end{multline*}
and denote the above map by $\beta$.
\newline
\begin{prop}\label{e_i,e_i relations}
The formula
\[
 {\left[e_{j}^{(s+1)}, e_{j}^{(t)}\right]}-{\left[e_{j}^{(s)}, e_{j}^{(t+1)}\right]}=-\delta \cdot\left(e_{j}^{(s)}e_{j}^{(t)}+e_{j}^{(t)}e_{j}^{(s)}\right)
\]
holds as an equality of operators
\[
H^{*}(F^{n}\text{Quot}(V))\rightarrow H^{*}(F^{n}\text{Quot}(V)\times C\times C)
\]
for all $j=1,2,...,n$, $s,t\geq 0$. Here, we make the convention that the operators $e_{j}^{(s)},e_{j}^{(s+1)}$ contribute to the first factor of $C\times C$ and the operators $e_{j}^{(t)},e_{j}^{(t+1)}$ contribute to the second factor of $C\times C$. The class $\delta\in H^{*}(F^{n}\text{Quot}(V)\times C\times C)$ is the pull-back of the diagonal class in $ H^{*}( C\times C)$ via the natural projection.
\end{prop}

\begin{proof}
    We have to show the following equality of operators:
   \begin{equation}\label{What must be shown in the proof of the ee 1 relation}
       e_{j}^{(s+1)}e_{j}^{(t)}-e_{j}^{(s)}e_{j}^{(t+1)}+\delta\cdot e_{j}^{(s)}e_{j}^{(t)}=e_{j}^{(t)}e_{j}^{(s+1)}-e_{j}^{(t+1)}e_{j}^{(s)} -\delta\cdot e_{j}^{(t)}e_{j}^{(s)}.
   \end{equation}
For simplicity of notation, we denote the respective pull-backs of $\delta\in H^{2}(C\times C)$ to $\mathcal{X}^{j,j}_{\vec{d}}$ and $\mathcal{W}^{j,j}_{\vec{d}}$, via $p^{X}\times q^{X}$ and $p^{W}\times q^{W}$, by $\delta$ as well. Then by Lemmas \ref{corollary on the operator ejej s,t} and \ref{corollary on the operator ejej t,s} the L.H.S. and R.H.S. of equation \ref{What must be shown in the proof of the ee 1 relation} are respectively given as correspondences by the following classes
on $F^{n}\text{Quot}_{\vec{d}}(V)\times F^{n}\text{Quot}_{\vec{d''}_{j,j}}(V)\times C\times C$:

\begin{equation}\label{correspondence of LHS in proof of ee relation 2}
 {\left(k_{X}\times p^{X}\times q^{X}\right)}_{*}\left({\lambda}_{1}^{t}{\lambda}_{2}^{s+1}-{\lambda}_{1}^{t+1}{\lambda}_{2}^{s}+\delta\cdot {\lambda}_{1}^{t}{\lambda}_{2}^{s}\right)
 \newline
 \end{equation}
 and
 \begin{equation}\label{correspondence of RHS in proof of ee relation 2}
 {\left(k_{W}\times p^{W}\times q^{W}\right)}_{*}\left(\tilde{\lambda}_{1}^{s+1}\tilde{\lambda}_{2}^{t}-\tilde{\lambda}_{1}^{s}\tilde{\lambda}_{2}^{t+1}-\delta\cdot \tilde{\lambda}_{1}^{s}\tilde{\lambda}_{2}^{t}\right).
\end{equation}

\noindent Now, using the projection formula, the fact that the maps $\eta^{jj}_{\vec{d}}$, $\theta^{jj}_{\vec{d}}$ are birational and corollary \ref{corollay about pushforward of diag being delta}; The class \ref{correspondence of LHS in proof of ee relation 2} can be re-expressed as:
\begin{equation}\label{correspondence of LHS in proof of ee relation 2 re-written}
 {\beta}_{*}\left({\lambda}_{1}^{t}{\lambda}_{2}^{s+1}-{\lambda}_{1}^{t+1}{\lambda}_{2}^{s}+[\text{Diag}]\cdot {\lambda}_{1}^{t}{\lambda}_{2}^{s}\right).
 \newline
 \end{equation}

 \noindent Similarly, the class \ref{correspondence of RHS in proof of ee relation 2} is equal to
\begin{equation}\label{correspondence of RHS in proof of ee relation 2 re-written}
 {\beta}_{*}\left(\tilde{\lambda}_{1}^{s+1}\tilde{\lambda}_{2}^{t}-\tilde{\lambda}_{1}^{s}\tilde{\lambda}_{2}^{t+1}-[\text{Diag}]\cdot \tilde{\lambda}_{1}^{s}\tilde{\lambda}_{2}^{t}\right).
 \end{equation}
 \noindent Therefore to prove the equality \ref{What must be shown in the proof of the ee 1 relation}, it suffices to show the following equality of cohomology classes on $\mathcal{U}^{j,j}_{\vec{d}}$:
\begin{equation}\label{What must be shown in the proof of the ee 1 relation, simplified}
{\lambda}_{1}^{t}{\lambda}_{2}^{s+1}-{\lambda}_{1}^{t+1}{\lambda}_{2}^{s}+[\text{Diag}]\cdot {\lambda}_{1}^{t}{\lambda}_{2}^{s}=
\tilde{\lambda}_{1}^{s+1}\tilde{\lambda}_{2}^{t}-\tilde{\lambda}_{1}^{s}\tilde{\lambda}_{2}^{t+1}-[\text{Diag}]\cdot \tilde{\lambda}_{1}^{s}\tilde{\lambda}_{2}^{t}.
\end{equation}
\noindent The L.H.S. of equation \ref{What must be shown in the proof of the ee 1 relation, simplified} is equal to 
\[
{\lambda}_{1}^{t}{\lambda}_{2}^{s}(\lambda_{2}-\lambda_{1}+[\text{Diag}]).
\]
By Proposition \ref{Diag in cohomology}, the above expression is equal to:
\begin{equation}\label{LHS OF What must be shown in the proof of the ee 1 relation, simplified}
    {\lambda}_{1}^{t}{\lambda}_{2}^{s}\left(\lambda_{2}-\tilde{\lambda}_{2}\right).
\end{equation}
\noindent Now, observe that
\[
[\text{Diag}]\mid {\lambda}_{1}^{t}{\lambda}_{2}^{s}-\tilde{\lambda}_{1}^{s}\tilde{\lambda}_{2}^{t}.
\]
Indeed, this follows from re-writing $\lambda_{1}$ and  $\tilde{\lambda}_{1}$ respectively as $\tilde{\lambda}_{2}+[\text{Diag}]$ and ${\lambda}_{2}+[\text{Diag}]$ followed by a binomial expansion. Also by \ref{observation 2 in proof of serre relations}, we have that
\[
[\text{Diag}]\cdot\left(\tilde{\lambda}_{2}-\lambda_{2}\right)=0
\]
and hence the expression \ref{LHS OF What must be shown in the proof of the ee 1 relation, simplified} is equal to 
\[
{\tilde{\lambda}^{s}_{1}}{\tilde{\lambda}^{t}_{2}}\left(\lambda_{2}-\tilde{\lambda}_{2}\right).
\]
Again, Proposition \ref{Diag in cohomology} implies that the above expression can be written as
\[
{\tilde{\lambda}^{s}_{1}}{\tilde{\lambda}^{t}_{2}}\left(\tilde{\lambda}_{1}-\tilde{\lambda}_{2}-[\text{Diag}]\right),
\]
which is precisely the R.H.S. of \ref{What must be shown in the proof of the ee 1 relation, simplified}, finishing the proof.
\end{proof}

\subsection{The $e,m$ and $f,m$ commutation relations: \eqref{em relations action} and \eqref{fm relations action}}
We will verify relation \ref{em relations action} and omit the similar proof of relation \ref{fm relations action}.\\\\
First note that if $i\neq j$ then $[m_{i}^{(s)}, e_{j}^{(t)}]=0$. Indeed, one checks using propositions \ref{class of correspondence} and \ref{composition of correspondences} that the compositions $m_{i}^{(s)} e_{j}^{(t)}$ and $e_{j}^{(t)}m_{i}^{(s)}$ are given by the same correspondences.\\\\
Let us set up some notations and conventions for this subsection. As before, let $\vec{d}=(d_1,...,d_{n})$ be an $n-$tuple of non-negative integers. For $j=1,...,n$, we define the $n-$tuple $\vec{d}'_{j}:=(d_1,...,d_{j-1},d_{j}+1,d_{j+1},...,d_{n})$ and the $(n+1)-$tuple $\vec{d}''_{j}:=(d_1,...,d_{j-1},d_{j},d_{j}+1,d_{j+1},...,d_{n})$. When $\vec{d}$ or $\vec{d}'_{j}$ is not non-decreasing, then the statements in this subsection will be vacuously true as the moduli space $F^{n}\text{Quot}_{\vec{d}'_{j}}\,(V)$ will become empty. We also follow the convention of the statement of Theorem \ref{Yangian relations satisfied}; The operators $m_{j}^{(s)}$ contribute to the first factor of $C\times C$ and the operators $e_{j}^{(t)}$ contributes to the second factor of $C\times C$.\\\\
To verify the relation in the case $i=j$, it would be convenient to arrange the operators $e_{j}^{(t)}$ and $m_{j}^{(s)}$ into power series, as in \eqref{currents}.
Then, relation \ref{em relations action} is equivalent to the following equality of power series:

\begin{equation}\label{em relations action current}
    m_{j}(w)e_{j}(z)=\left[e_{j}(z)m_{j}(w)\frac{w-z+\delta}{w-z}\right]_{w>>z\,\mid z^{<0},w^{\geq 0}},
\end{equation}
\newline
\noindent as an equality of operators $H^{*}(F^{n}\text{Quot}_{\vec{d}}\,(V))\rightarrow H^{*}(F^{n}\text{Quot}_{\vec{d}'_{j}}\,(V))[[z^{-1}]][w]$. In equation \ref{em relations action current}, the subscripts indicate that we expand the R.H.S. as a power series with the rule $w>>z$ and only keep terms with negative powers of $z$ and non-negative powers of $w$.\\\\
We will crucially use the geometry of the hyperquot scheme $F^{n+1}\text{Quot}_{\vec{d}''_{j}}$ and refer the reader to subsection \ref{subsection on moduli space: Fn+1 Quot} for conventions and notations. Let us define the map 
\[
\Psi:=\pi_{-}\times\pi_{+}\times Id_{C}\times \pi_{C} :F^{n+1}\text{Quot}_{\vec{d}''_{j}}\times C\rightarrow F^{n+1}\text{Quot}_{\vec{d}}\times F^{n+1}\text{Quot}_{\vec{d}'_{j}}\times C\times C.
\]
Then the following lemmas follow from propositions \ref{class of correspondence} and \ref{composition of correspondences}.

\begin{lem}\label{corollary on the operator e(z)m(w)}
     The linear operator
\[
    e_{j}(z)m_{j}(w):H^{*}(F^{n}\text{Quot}_{\vec{d}}(V))\rightarrow H^{*}(F^{n}\text{Quot}_{\vec{d}'_{j}}(V)\times C\times C)[[z^{-1}]][w]
\]
is given as a correspondence by the class $${\Psi}_{*}\left[\frac{c(\mathcal{E}_{j},w)}{c(\mathscr{L}_{j},z)}\right]$$ on $F^{n}\text{Quot}_{\vec{d}}(V)\times F^{n}\text{Quot}_{\vec{d}'_{j}}(V)\times C\times C$.
\newline
\end{lem}

\begin{lem}\label{corollary on the operator m(w)e(z)}
         The linear operator
\[
    m_{j}(w)e_{j}(z):H^{*}(F^{n}\text{Quot}_{\vec{d}}(V))\rightarrow H^{*}(F^{n}\text{Quot}_{\vec{d}'_{j}}(V)\times C\times C)[[z^{-1}]][w]
\]
is given as a correspondence by the class $${\Psi}_{*}\left[\frac{c(\mathcal{F}_{j},w)}{c(\mathscr{L}_{j},z)}\right]$$ on $F^{n}\text{Quot}_{\vec{d}}(V)\times F^{n}\text{Quot}_{\vec{d}'_{j}}(V)\times C\times C$.
\newline
\end{lem}

\begin{prop}\label{m,e relations}
The formula
\[
 {\left[m_{i}^{(s)}, e_{j}^{(t)}\right]}=\delta_{i,j}\cdot\delta\cdot\left(\sum_{l=0}^{s-1}e_{j}^{(t+l)}m_{i}^{(s-l-1)}(-1)^{l+1}\right).
\]
holds as an equality of operators
\[
H^{*}(F^{n}\text{Quot}(V))\rightarrow H^{*}(F^{n}\text{Quot}(V)\times C\times C)
\]
for all $i,j\in \{1,2,...,n\}$, $s,t\geq 0$. Here, we make the convention that the operators with $(s)$ in the superscript contribute to the first factor of $C\times C$ and the operators with $(t)$ in the superscript, contribute to the second factor of $C\times C$. The class $\delta\in H^{*}(F^{n}\text{Quot}(V)\times C\times C)$ is the pull-back of the diagonal class in $ H^{*}( C\times C)$ via the natural projection.
    
\end{prop}

\begin{proof}
    We already addressed the case $i\neq j$, for the case $i=j$, let us verify the equivalent formula \ref{em relations action current}. By Lemmas \ref{corollary on the operator e(z)m(w)} and \ref{corollary on the operator m(w)e(z)}, it is enough to show that the coefficients of the variables $z^{k}w^{l}$, in the expression:
\begin{equation}\label{What must be shown in the proof of the em relation}
    \frac{c(\mathcal{F}_{j},w)}{c(\mathscr{L}_{j},z)}-\frac{c(\mathcal{E}_{j},w)}{c(\mathscr{L}_{j},z)}\cdot\frac{w-z+\delta}{w-z},
\end{equation}
    when expanded as a power series with the rule $w>>z$, are non-zero only if $k\geq0$ or $l<0$.\\\\
    By the identity \ref{short exact seq cohomology identity on Fn+1Quot x C}, we obtain that the expression \ref{What must be shown in the proof of the em relation} can be re-written as:
    
    \[
    \frac{c(\mathcal{E}_{j},w)}{c(\mathscr{L}_{j},z)}\left[\frac{w-\lambda_{j}+\delta}{w-\lambda_{j}}-\frac{w-z+\delta}{w-z}\right]
= -\frac{1}{w-z}\left[\frac{\delta\cdot c(\mathcal{E}_{j},w)}{w-\lambda_{j}}\right].
\]    
\newline
Then, by the definition of $\mathcal{G}_{j}$ (cf. equation \ref{definition of r-1 rk bundle G}) and the projection formula, the above expression is equal to 
\[
-\Delta_{*}\left(\frac{c(\mathcal{G}_{j},w)}{w-z}\right).
\]
This doesn't have any terms in negative powers of $z$, when expanded using the rule $w>>z$, which is what we wanted to prove.

\end{proof}

\subsection{The $e,f$ commutation relations: \eqref{ef relations action}}
Let us set up some notation for this subsection. Let $\vec{d}=(d_{1},...,d_{n})$ be an $n-$tuple of non-negative integers. For $j=1,...,n$, we define $n-$tuples $\vec{d}'_{j}:=(d_1,...,d_{j-1},d_{j}+1,d_{j+1},...,d_{n})$ and ${}^{'}\vec{d}_{j}:=(d_1,...,d_{j-1},d_{j}-1,d_{j+1},...,d_{n})$. We also define $(n+1)-$tuples  $\vec{d}''_{j}:=(d_1,...,d_{j-1},d_{j},d_{j}+1,d_{j+1},...,d_{n})$ and ${}^{''}\vec{d}_{j}:=(d_1,...,d_{j-1},d_{j}-1,d_{j},d_{j+1},...,d_{n})$.\\\\
When $i\neq j$, then one observes that $[e_{i}^{(s)}, f_{j}^{(t)}]=0$ for all $s,t$. This follows directly from Propositions \ref{class of correspondence} and \ref{composition of correspondences}, which imply that the compositions $e_{i}^{(s)}f_{j}^{(t)}$ and $f_{j}^{(t)} e_{i}^{(s)}$ are given by the same correspondences.\\\\
To tackle the $i=j$ case, as in the last subsection, it would be convenient to arrange the operators $e_{j}^{(s)}$ and $f_{j}^{(t)}$ into power series $e_{j}(w)$ and $f_{j}(z)$ (cf. \eqref{currents}). The proof will consist of two major steps:
\begin{enumerate}[align=left, widest=97,leftmargin=\parindent, labelsep=*]
    \item\emph{ } We will first show that the operator
    \[
    e_{j}(w)f_{j}(z)-f_{j}(z)e_{j}(w):H^{*}(F^{n}\text{Quot}_{\vec{d}}(V))\rightarrow H^{*}(F^{n}\text{Quot}_{\vec{d}}(V)\times C \times C)[[z^{-1},w^{-1}]]
    \]
\noindent is given as a correspondence by a class of the form $\Delta_{*}(\Phi)$ on $F^{n}\text{Quot}_{\vec{d}}(V)\times F^{n}\text{Quot}_{\vec{d}}(V)\times C\times C$, where 
\[
\Delta:F^{n}\text{Quot}_{\vec{d}}(V)\times C\hookrightarrow F^{n}\text{Quot}_{\vec{d}}(V)\times F^{n}\text{Quot}_{\vec{d}}(V) \times C \times C
\] 
is the diagonal embedding.\\\\
\item\emph{ } Part (1) would readily imply that the operator $e_{j}(w)f_{j}(z)-f_{j}(z)e_{j}(w)$ is given by pulling back  to $H^{*}(F^{n}\text{Quot}_{\vec{d}}(V)\times C \times C)[[z^{-1},w^{-1}]]$ and then multiplying by some class $\Theta$. Then to identify $\Theta$, one would just need to compute
\[
 \left[e_{j}(w)f_{j}(z)-f_{j}(z)e_{j}(w)\right](1).
\]
Where we denote the fundamental class of $F^{n}\text{Quot}_{\vec{d}}(V)$ by $1$. 
\newline
\end{enumerate}
\textbf{Step I:}
 Let us show that $e_{j}(w)f_{j}(z)-f_{j}(z)e_{j}(w)$ is given a correspondence by a class supported on the diagonal in $F^{n}\text{Quot}_{\vec{d}}(V)\times F^{n}\text{Quot}_{\vec{d}}(V)\times C\times C$.\\\\  
The following lemmas follow from Propositions \ref{class of correspondence}, \ref{composition of correspondences} and the dimension counts in Propositions \ref{Geometry of moduli space U tilde}-\ref{Geometry of moduli space $t$}. We borrow notation from subsection \ref{subsection: moduli spaces U,S,T} and crucially use the geometry of $\tilde{\mathcal{U}}_{\vec{d}}^{i}\,$, $\mathcal{S}_{\vec{d}}^{i}$ and $\mathcal{T}_{\vec{d}}^{i}$.

\begin{lem}\label{corollary on the operator f(z)e(w)}
    The operator \[
    f_{j}(z)e_{j}(w):H^{*}(F^{n}\text{Quot}_{\vec{d}}(V))\rightarrow H^{*}(F^{n}\text{Quot}_{\vec{d}}(V)\times C \times C)[[z^{-1},w^{-1}]]
    \] 
    is given as a correspondence by the class
    \[
    (k_{S}\times q^{S}\times p^{S})_{*}\left[\frac{1}{w-\lambda_{2}}\cdot \frac{1}{z-\lambda^{'}_{2}}\right]
    \]
    on $F^{n}\text{Quot}_{\vec{d}}(V)\times F^{n}\text{Quot}_{\vec{d}}(V)\times C\times C$.
    \newline
\end{lem}

\begin{lem}\label{corollary on the operator e(w)f(z)}
The operator \[
    e_{j}(w)f_{j}(z):H^{*}(F^{n}\text{Quot}_{\vec{d}}(V))\rightarrow H^{*}(F^{n}\text{Quot}_{\vec{d}}(V)\times C \times C)[[z^{-1},w^{-1}]]
    \] 
    is given as a correspondence by the class
    \[
    (k_{T}\times q^{T}\times p^{T})_{*}\left[\frac{1}{w-\lambda^{'}_{1}}\cdot \frac{1}{z-\lambda_{1}}\right]
    \]
    on $F^{n}\text{Quot}_{\vec{d}}(V)\times F^{n}\text{Quot}_{\vec{d}}(V)\times C\times C$.
\end{lem}

\noindent Now, to show that  $e_{j}(w)f_{j}(z)-f_{j}(z)e_{j}(w)$ is given by a correspondence by a class supported on $\Delta$, by the excision long exact sequence, it is sufficient to show that classes in lemmas \ref{corollary on the operator f(z)e(w)} and \ref{corollary on the operator e(w)f(z)} restrict to the same class on the complement of $\Delta$ in $F^{n}\text{Quot}_{\vec{d}}(V)\times F^{n}\text{Quot}_{\vec{d}}(V)\times C\times C$. But this just follows from Lemmas $\ref{Lemma excision argument one}$ and $\ref{Lemma excision argument two}$.\\\\
\textbf{Step II:}
We will now show that 
\[
\left[e_{j}(w)f_{j}(z)-f_{j}(z)e_{j}(w)\right](1)=\delta\cdot\frac{h_{j}(w)-h_{j}(z)}{w-z}.
\]
Where
\[
h_{j}(x):=\frac{c(\mathcal{E}_{j+1},x)c(\mathcal{E}_{j-1},x+K_C)}{c(\mathcal{E}_{j},x)c(\mathcal{E}_{j},x+K_C)}
\]
and we make the convention that $c(\mathcal{E}_{n+1},x)=1$.
Step I would then immediately imply:

\begin{prop}\label{e,f relations}
The formulas
\[
    {[e_{j}^{(s)}, f_{j}^{(t)}]}=-\delta\cdot h_{_{j}}^{(s+t+1)} \text{ if $j\neq n$ .}
\]
\[
    {[e_{j}^{(s)}, f_{j}^{(t)}]}=-\delta\cdot h_{_{j}}^{(s+t-r+1)} \text{ if $j= n$.} 
\]
hold as equalities of operators
\[
H^{*}(F^{n}\text{Quot}(V))\rightarrow H^{*}(F^{n}\text{Quot}(V)\times C\times C)
\]
for all $i,j\in \{1,2,...,n\}$, $s,t\geq 0$. Here, we make the convention that the operators with $(s)$ in the superscript contribute to the first factor of $C\times C$ and the operators with $(t)$ in the superscript, contribute to the second factor of $C\times C$. The class $\delta\in H^{*}(F^{n}\text{Quot}(V)\times C\times C)$ is the pull-back of the diagonal class in $ H^{*}( C\times C)$ via the natural projection.
    
\end{prop}

\begin{proof}
This computation is mutatis mutandis the one in \cite{MaNe1}.
We will borrow notation from subsections \ref{subsection on moduli space: Fn+1 Quot} and \ref{subsection: projective bundles}. For any space $X$, we will denote the pull-back of the canonical class on $C$ to the cohomology of $X\times C$ by $K_{C}$.\\\\
Let us first calculate $f_{j}(z)e_{j}(w)\cdot1$ on $F^{n}\text{Quot}_{\vec{d}}(V)$. Let
\[
\mathcal{E}_{n}\subseteq ...\mathcal{E}_{j+1}\subseteq\mathcal{F}_{j}\subseteq\mathcal{E}_{j-1}\subseteq...\mathcal{E}_1\subseteq \mathcal{E}_{0}=\pi^{*}(V).
\]
be the universal flag on $F^{n}\text{Quot}_{{\vec{d}'}_{j}}\,(V)$  and let
\[
\mathcal{E}_{n}\subseteq ...\mathcal{E}_{j+1}\subseteq\mathcal{F}_{j}\subseteq\mathcal{E}_{j}\subseteq\mathcal{E}_{j-1}\subseteq...\mathcal{E}_1\subseteq \mathcal{E}_{0}=\pi^{*}(V).
\]
be the universal flag on $F^{n+1}\text{Quot}_{{\vec{d}''}_{j}}\,(V)$. We will denote the first Chern class of the tautological line bundle $\mathscr{L}_{j}$ (cf. \ref{subsection on moduli space: Fn+1 Quot}) on $F^{n+1}\text{Quot}_{{\vec{d}''}_{j}}\,(V)$ by $\lambda_{j}$. \\\\ 
From formula \ref{pushforward formula projective bundle} and Proposition \ref{pi+times piC as a projective bundle}, we have
    
  \[
  f_{j}(z)e_{j}(w)\cdot 1=f_{j}(z)\cdot\left[1-\frac{c(\mathcal{E}_{j-1},w+K_{C})}{c(\mathcal{F}_{j},w+K_{C})}\right].
  \]
  \newline
  \[
  =(\pi_{-}\times \pi_{C})_{*}\left[\frac{1}{z-\lambda_{j}}\cdot\left(1-\frac{c(\mathcal{E}_{j-1},w+K_{C})}{c(\mathcal{F}_{j},w+K_{C})}\right)\right].
  \]
  \newline
  Then formula \ref{short exact seq cohomology identity on Fn+1Quot x C} implies that his expression is equal to

  \[
  (\pi_{-}\times \pi_{C})_{*}\left[\frac{1}{z-\lambda_{j}}\cdot\left[1-\frac{c(\mathcal{E}_{j-1},w+K_{C})}{c(\mathcal{E}_{j},w+K_{C})}\cdot\frac{w-\lambda_{j}+K_{C}}{w-\lambda_{j}+K_{C}+\delta}\right]\right]
  \]
  \newline
    \begin{multline*}
  =(\pi_{-}\times \pi_{C})_{*}\left[\frac{1}{z-\lambda_{j}}\cdot\left(1-\frac{c(\mathcal{E}_{j-1},w+K_{C})}{c(\mathcal{E}_{j},w+K_{C})}\right)+\right.\\\left.
  \delta\cdot\frac{c(\mathcal{E}_{j-1},w+K_{C})}{c(\mathcal{E}_{j},w+K_{C})}\cdot\frac{1}{(z-\lambda_{j})(w-\lambda_{j}+K_{C}+\delta)}\right].
  \end{multline*}
  \newline
  Then the identity $\delta(\delta+K_{C})=0$ in $H^{*}(C\times C)$, implies that
  
  \[
  \frac{\delta}{w-\lambda_{j}+K_{C}+\delta}=\frac{\delta}{w-\lambda_{j}},
  \]
  \newline
  using which we can simplify $f(z)e(w)\cdot 1$ to
  
  \begin{multline*}
  (\pi_{-}\times \pi_{C})_{*}\left[\frac{1}{z-\lambda_{j}}\cdot\left(1-\frac{c(\mathcal{E}_{j-1},w+K_{C})}{c(\mathcal{E}_{j},w+K_{C})}\right)+\right.\\\left.
  \frac{\delta}{z-w}\cdot\frac{c(\mathcal{E}_{j-1},w+K_{C})}{c(\mathcal{E}_{j},w+K_{C})}\cdot\left(\frac{1}{w-\lambda_{j}}-\frac{1}{z-\lambda_{j}}\right)\right].
  \end{multline*}
  \newline
  Then formula \ref{pushforward formula projective bundle} and Proposition \ref{pi-times piC as a projective bundle} yield:

  \begin{multline}\label{f(z)e(w).1}
f_{j}(z)e_{j}(w)\cdot 1= \left[\frac{c(\mathcal{E}_{j+1},z)}{c(\mathcal{E}_{j},z)}\cdot\left(1-\frac{c(\mathcal{E}_{j-1},w+K_{C})}{c(\mathcal{E}_{j},w+K_{C})}\right)+\right.\\ \left. \frac{\delta}{z-w}\cdot\frac{c(\mathcal{E}_{j-1},w+K_{C})}{c(\mathcal{E}_{j},w+K_{C})}\cdot\left(\frac{c(\mathcal{E}_{j+1},w)}{c(\mathcal{E}_{j},w)}-\frac{c(\mathcal{E}_{j+1},z)}{c(\mathcal{E}_{j},z)}\right)\right].
  \end{multline}
Now, let us compute $e_{j}(w)f_{j}(z)\cdot 1$. Let
\[
\mathcal{E}_{n}\subseteq ...\mathcal{E}_{j+1}\subseteq\tilde{\mathcal{F}}_{j}\subseteq\mathcal{E}_{j-1}\subseteq...\mathcal{E}_1\subseteq \mathcal{E}_{0}=\pi^{*}(V).
\]
be the universal flag on $F^{n}\text{Quot}_{\,{{}^{'}\vec{d}}_{j}}\,(V)$  and let
\[
\mathcal{E}_{n}\subseteq ...\mathcal{E}_{j+1}\subseteq\mathcal{E}_{j}\subseteq\tilde{\mathcal{F}}_{j}\subseteq\mathcal{E}_{j-1}\subseteq...\mathcal{E}_1\subseteq \mathcal{E}_{0}=\pi^{*}(V).
\]
be the universal flag on $F^{n+1}\text{Quot}_{\,{}^{''}{\vec{d}}_{j}}\,(V)$. Again, we will denote the first Chern class of the tautological line bundle on $F^{n+1}\text{Quot}_{{}^{''}{\vec{d}}_{j}}\,(V)$ by $\lambda_{j}$.\\\\
Using formula \ref{pushforward formula projective bundle} and proposition \ref{pi-times piC as a projective bundle}, we have
\[
e_{j}(w)f_{j}(z)\cdot 1=e_{j}(w)\left[\frac{c(\mathcal{E}_{j+1},z)}{c(\tilde{\mathcal{F}_{j}},z)}\right]
\]
\[
=(\pi_{+}\times\pi_{C})_{*}\left[\frac{1}{w-\lambda_{j}}\cdot \frac{c(\mathcal{E}_{j+1},z)}{c(\tilde{\mathcal{F}_{j}},z)}\right].
\]
By formula \ref{short exact seq cohomology identity on Fn+1Quot x C}, this is equal to

\[
(\pi_{+}\times\pi_{C})_{*}\left[\frac{1}{w-\lambda_{j}}\cdot \frac{c(\mathcal{E}_{j+1},z)}{c({\mathcal{E}_{j}},z)}\cdot\frac{z-\lambda_{j}+\delta}{z-\lambda_{j}}\right]
\]
\newline
\[
=(\pi_{+}\times\pi_{C})_{*}\left[\frac{1}{w-\lambda_{j}}\cdot \frac{c(\mathcal{E}_{j+1},z)}{c({\mathcal{E}_{j}},z)}+\frac{\delta}{z-w}\cdot\frac{c(\mathcal{E}_{j+1},z)}{c({\mathcal{E}_{j}},z)}\cdot\left(\frac{1}{w-\lambda_{j}}-\frac{1}{z-\lambda_{j}}\right)\right].
\]
\newline
Applying formula \ref{pushforward formula projective bundle} and Proposition \ref{pi+times piC as a projective bundle}, we obtain

\begin{multline}\label{e(w)f(z).1}
e_{j}(w)f_{j}(z)\cdot 1=\left(1-\frac{c(\mathcal{E}_{j-1},w+K_{C})}{c(\mathcal{E}_{j},w+K_{C})}\right)\cdot \frac{c(\mathcal{E}_{j+1},z)}{c({\mathcal{E}_{j}},z)}\\+\frac{\delta}{z-w}\cdot\frac{c(\mathcal{E}_{j+1},z)}{c({\mathcal{E}_{j}},z)}\cdot\left(\frac{c(\mathcal{E}_{j-1},z+K_{C})}{c(\mathcal{E}_{j},z+K_{C})}-\frac{c(\mathcal{E}_{j-1},w+K_{C})}{c(\mathcal{E}_{j},w+K_{C})}\right).
  \end{multline}

Subtracting the expression  \ref{f(z)e(w).1} from the expression \ref{e(w)f(z).1} gives:

\begin{multline*}
\left[e_{j}(w)f_{j}(z)-f_{j}(z)e_{j}(w)\right]\cdot 1= \\
\frac{\delta}{w-z}\cdot\left[\frac{c(\mathcal{E}_{j+1},w)c(\mathcal{E}_{j-1},w+K_C)}{c(\mathcal{E}_{j},w)c(\mathcal{E}_{j},w+K_C)}-\frac{c(\mathcal{E}_{j+1},z)c(\mathcal{E}_{j-1},z+K_C)}{c(\mathcal{E}_{j},z)c(\mathcal{E}_{j},z+K_C)}\right],
\end{multline*}
which is what we wanted to prove.
\end{proof}

\section{A Natural Basis for the Cohomology}\label{section: natural basis}

Let us recall from the introduction (cf. equation \ref{Operators a's definition}), the operators
\[
  a_{k}^{(v)}\colon H^{*}(F^{n}\text{Quot}(V))\rightarrow H^{*}(F^{n}\text{Quot}(V)\times C),
\]
that are defined for $k\in\{1,...,n\}$ and $v\in\{0,...,r-1\}$.
Let $\gamma_1,\gamma_{2},...,\gamma_{2g+2}$ be a basis of $H^{*}(C)$. We may then define
\[
a_{k}^{(v)}(\gamma_{i})\colon H^{*}(F^{n}\text{Quot}(V))\rightarrow H^{*}(F^{n}\text{Quot}(V))
\]
for $k\in\{1,...,n\}, v\in\{0,...,r-1\}$ and $i\in\{1,...,2g+2\}$ (cf. equation \ref{operators coloured with cohomology classes T(gamma)}).\\\\
The goal of this section is to describe the graded vector space $H^{*}(F^{n}\text{Quot}(V))$ with the aid of the operators $a_{k}^{(v)}(\gamma_{i})$. More precisely, we will prove Theorem \ref{a's commute intro}, that the operators $a_{k}^{(v)}(\gamma_{i})$ commute, and prove Theorem \ref{basis of cohomology intro}, that the set

\[
\mathcal{B}:=\left\{a_{k_l}^{(v_l)}(\gamma_{i_l})...a_{k_2}^{(v_2)}(\gamma_{i_2})a_{k_1}^{(v_1)}(\gamma_{i_1})|0\rangle\right\}_{l\geq 0}\subseteq H^{*}(F^{n}\text{Quot}(V))
\]
\newline
forms a basis for $H^{*}(F^{n}\text{Quot}(V))$. Here, the symbol $|0\rangle$ denotes a fixed generator of the one dimensional $\mathbb{Q}-$vector space $H^{*}(F^{n}\text{Quot}_{(0,...,0)}(V)).$ These results will immediately yield Corollary \ref{description as graded vec space}, that is the following isomorphism of graded vector spaces
\[
H^{*}(F^{n}\text{Quot}(V))\cong \text{Sym}\left(\bigoplus\limits^{i=1,...,2g+2}_{\substack{{k=1,...,n} \\ {v=0,...,r-1}}}\mathbb{Q}\cdot a_{k}^{(v)}(\gamma_{i})\right) |0\rangle.
\]
\newline
The proofs of the Theorems \ref{a's commute intro} and \ref{basis of cohomology intro} will be carried out in four major steps:
\begin{enumerate}[align=left, widest=99,leftmargin=\parindent, labelsep=*]
\item \emph{} We impose a lexicographic order on the operators $a_{k}^{(v)}$ and declare $a_{k^{'}}^{(v^{'})}(\gamma_{i^{'}})> a_{k}^{(v)}(\gamma_{i})$ if $k^{'}>k$ or $k^{'}=k,$ $v^{'}>v$ or $k^{'}=k,$ $v^{'}=v,$ $i^{'}>i$. Then, we will show that the elements
\[
a_{k_t}^{(v_t)}(\gamma_{i_t})...a_{k_2}^{(v_2)}(\gamma_{i_2})a_{k_1}^{(v_1)}(\gamma_{i_1})|0\rangle\in H^{*}(F^{n}\text{Quot}(V)),
\]
for arbitrary $t$, with the property that 
\[
a_{k_1}^{(v_1)}(\gamma_{i_1})\leq a_{k_2}^{(v_2)}(\gamma_{i_2})\leq ...\leq a_{k_t}^{(v_t)}(\gamma_{i_t}),
\]
are linearly independent. 
\item \emph{ }We will show that the cohomology of $H^{*}(F^{n}\text{Quot}_{\vec{d}}(V))$ is generated as a ring, by so-called universal classes; This is equivalent to saying that any class in $H^{*}(F^{n}\text{Quot}_{\vec{d}}(V))$ can be written as a linear combination of elements of the form 
\[
m_{i_{t}}^{(k_{t})}(\gamma_{j_{t}}) m_{i_{t-1}}^{(k_{t-1})}(\gamma_{j_{t-1}})... m_{i_{1}}^{k_{1}}(\gamma_{j_{1}})[F^{n}\text{Quot}_{\vec{d}}(V)],
\]
for various $t\geq0$.
\\
\item \emph{ }Using (2), we will show that the elements of the form
\[
a_{k_l}^{(v_l)}(\gamma_{i_l})...a_{k_2}^{(v_2)}(\gamma_{i_2})a_{k_1}^{(v_1)}(\gamma_{i_1})|0\rangle,
\]
with arbitrary $l$ and $n\geq k_{l}\geq k_{l-1}\geq...\geq k_{1}\geq 1$, span $H^{*}(F^{n}\text{Quot}(V))$. Note that a priori, this collection of elements is possibly larger than the one considered in step $(1)$.\\
\item\emph{ }Finally, we will crucially use step (3) to show Theorem \ref{a's commute intro}. In turn, this implies that the spanning set considered in step (3) is exactly equal to the linearly independent set of step (1) and hence yields a basis for $H^{*}(F^{n}\text{Quot}(V))$.\\  

\end{enumerate}
Before we carry out the above steps, let us set-up some definitions and prove some important formulas;\\\\
Given non-negative integers $m\leq n$, and an $m-$tuple $\vec{e}=(e_{1},...,e_{m})$, we can realise $\vec{e}$ as an $n-$tuple by defining the $n-$tuple $r(\vec{e})$ as:
\[
r(\vec{e}):=(e_{1},...,e_{m},e_{m},...,e_{m}).
\]
For example, if $m=3$, $n=6$ and $\vec{e}=(3,5,9)$, then $r(\vec{e})=(3,5,9,9,9,9)$. In this way, there is an obvious identification
\begin{equation}\label{identification of hyperquot schemes}
    F^{m}\text{Quot}_{\vec{e}}\,(V)\cong F^{n}\text{Quot}_{r(\vec{e})}(V).
\end{equation}
Then we define $\mathcal{H}_{m}$ to be the following subspace of $H^{*}(F^{n}\text{Quot}(V))$:

\[
\mathcal{H}_{m}:= \bigoplus\limits_{\vec{e}=(e_{1},...,e_{m})\in \mathbb{N}^{m}}H^{*}(F^{n}\text{Quot}_{r(\vec{e})}(V)).
\]
\newline
Note that the identification \ref{identification of hyperquot schemes} implies
\begin{equation}\label{identification of Hm}
\mathcal{H}_{m}\cong H^{*}(F^{m}\text{Quot}(V))
\end{equation}
and that we have an increasing filtration
\begin{equation}\label{filtration of cohomology of Hyperquaot}
H^{*}(F^{n}\text{Quot}_{(0,...,0)}(V))=\mathcal{H}_{0}\subseteq \mathcal{H}_{1}\subseteq...\subseteq\mathcal{H}_{n}=H^{*}(F^{n}\text{Quot}(V)).
\end{equation}
Let us make note of an easy but very useful observation:
\begin{rem}\label{remark on am}
    Under the identification \ref{identification of Hm}, the action of the operator $a_{m}^{(v)}$ on\\
    $H^{*}(F^{m}\text{Quot}(V))$ is intertwined with the action of the restriction to $\mathcal{H}_{m}$ of the operator  $a_{m}^{(v)}$ on $H^{*}(F^{n}\text{Quot}(V))$.
\end{rem}

\noindent Now, for $k=1,...,n$ and $i\geq 0$, let us define operators
\[
b_{k}^{(i)}:{H^{*}(F^{n}\text{Quot}(V))}\rightarrow H^{*}(F^{n}\text{Quot}(V)\times C).
\]
 Let $\vec{d}=(d_1,...,d_{n})$ be an $n-$tuple of non-negative integers and for $1\leq k\leq n$, let $\vec{d}_{k,n}$ be the $n-$tuple $(d_1,...,d_{k-1},d_{k}+1,...,d_{n}+1)$. Then, let us recall the moduli space $\mathcal{Z}_{\vec{d}}^{k,n}$, which fits into the Diagram \ref{Diagram for Moduli space: Z}. The space $\mathcal{Z}_{\vec{d}}^{k,n}$ carries the tautological line bundle $\mathscr{L}_{k}$ (cf. \eqref{construction of line bundle L}), whose first Chern class $\lambda_{k}$ can be used to define

\[
b_{k}^{(i)}:= (q_{-}\times q_{C})_{*}(\lambda_{k}^{i}\cdot q_{+}^{*}(-)).
\]
\newline
In the spirit of Remark \ref{remark on am}, we make the following useful note:
\begin{rem}\label{remark on bm}
    Under the identification \eqref{identification of Hm}, the operator ${b_{k}^{(i)}}_{|_{\mathcal{H}_{k}}}$ is equal to the operator $f_{k}^{(i)}$ on $H^{*}(F^{k}\text{Quot}(V))$.
\end{rem}

\noindent Let us now prove a key proposition, which is the analogue of Proposition $2$, \cite{MaNe1} for Hyperquot schemes:

\begin{prop}\label{multiplication by chern classes formula}
 Recall the universal flag \ref{Universal flag} on $F^{n}\text{Quot}_{\vec{d}}\,(V)$. The operator $H^{*}(F^{n}\text{Quot}_{\vec{d}}\,(V))\rightarrow H^{*}(F^{n}\text{Quot}_{\vec{d}}\,(V)\times C)$, given by pull-back and multiplication by the class 
\[
c_{k}\left((\mathcal{E}_{n-1}-\mathcal{E}_{n})\otimes \omega_{C}^{-1}\right)
\]
is equal to the operator:
\[
\sum_{i=0}^{r-1}\left[a_{n}^{(i)}f_{n}^{(r+k-i-2)}\right]_{|_{\Delta}}(-1)^{i-k}.
\]
 Here, the products $a_{n}^{(*)}f_{n}^{(*)}$ are operators 
 \[
 H^{*}(F^{n}\text{Quot}_{\vec{d}}\,(V))\rightarrow H^{*}(F^{n}\text{Quot}_{\vec{d}}\,(V)\times C\times C)
 \] 
 which we then restrict to the diagonal
 \[
 \Delta:F^{n}\text{Quot}_{\vec{d}}\,(V)\times C\hookrightarrow F^{n}\text{Quot}_{\vec{d}}\,(V)\times C\times C.
 \]
\end{prop}

\begin{proof}
We will use notations and conventions from Subsection \ref{subsection on diagonal embeddings} pertaining to the the moduli space $\mathfrak{Z}_{\vec{d}}$. Let us re-write equation \ref{equation that will yield multiplication operators} as:

\begin{multline}\label{equation that will yield multiplication operators in the proof}
\left[\Delta_{F^{n}\text{Quot}_{{\vec{d}}}\,(V)}\times {C}\right]-\left(k_{1}\times k_{2}\times t_{C}\right)_{*}\left(\frac{c(\mathcal{G}_{2},\lambda_{1})}{z-\lambda_{1}}\right)\\=\left(\Delta_{F^{n}\text{Quot}_{{\vec{d}}}\,(V)}\times Id_{C}\right)_{*}\left(\frac{c(\mathcal{E}_{n-1},z+K_{C})}{c(\mathcal{E}_{n},z+K_{C})}\right).
\end{multline}
Also note that

\[
\frac{c(\mathcal{G}_{2},\lambda_{1})}{z-\lambda_{1}}=\left[\frac{c(\mathcal{G}_{2},z)}{z-\lambda_{1}}\right]_{z<0}\in H^{*}(\mathfrak{Z}_{\vec{d}}\,)[[z^{-1}]].
\]
\newline
\noindent Here the $z<0$ in the subscript means that we only keep those terms in the power series, that are in negative powers of $z$.\\\\
So it follows from propositions \ref{composition of correspondences} and \ref{class of correspondence}, that the L.H.S. of equation \ref{equation that will yield multiplication operators in the proof} is the class that defines the following operator $$H^{*}(F^{n}\text{Quot}_{\vec{d}}\,(V))
\rightarrow H^{*}(F^{n}\text{Quot}_{\vec{d}}\,(V)\times C)[[z^{-1}]]$$ as a correspondence:
\[
\text{Pull-back}-\left[a(z)f(z)_{|_{\Delta}}\right]_{z<0}.
\]
\newline
Whereas, the R.H.S. of the equation \ref{equation that will yield multiplication operators in the proof} is the class which defines the operator $H^{*}(F^{n}\text{Quot}_{\vec{d}}\,(V))\rightarrow H^{*}(F^{n}\text{Quot}_{\vec{d}}\,(V)\times C)[[z^{-1}]]$ given by pulling back and multiplication by

\[
\frac{c(\mathcal{E}_{n-1},z+K_{C})}{c(\mathcal{E}_{n},z+K_{C})}, 
\]
\newline
as a correspondence.\\\\
Finally, extracting the $z^{-k}$ coefficient finishes the proof.
\end{proof}

\subsection{Step I : Linear independence}
In this subsection, we will show that the following subset of $H^{*}(F^{n}\text{Quot}(V))$ is linearly independent:
\begin{equation}\label{linearly independent subset definition}
    \mathfrak{A}:=\left\{a_{k_t}^{(v_t)}(\gamma_{i_t})...a_{k_2}^{(v_2)}(\gamma_{i_2})a_{k_1}^{(v_1)}(\gamma_{i_1})|0\rangle\mid t\geq0\text{ and }  a_{k_1}^{(v_1)}(\gamma_{i_1})\leq ...\leq a_{k_t}^{(v_t)}(\gamma_{i_t})\right\}.
\end{equation}
Since the moduli space $\mathcal{Z}_{\vec{d}}^{n,n}$ (cf. Subsection \ref{Moduli space: The one key}) is just the hyperquot scheme $F^{n+1}\text{Quot}_{\vec{d}_{n}''}(V)$ (cf. Subsection \ref{subsection on moduli space: Fn+1 Quot}), the identity \ref{equality in cohomology:a=em} immediately implies the following equality of operators:
\begin{equation}\label{an=me}
a_{n}(z)=[e_{n}(z)m_{n}(z)]_{|_{\Delta}} :  H^{*}(F^{n}\text{Quot}(V))\rightarrow H^{*}(F^{n}\text{Quot}(V)\times C). 
\end{equation}
 Then, using the commutation relations in Theorem \ref{Yangian relations satisfied} and Proposition \ref{multiplication by chern classes formula}, we obtain the following proposition, whose proof we omit, since it identically follows the proof of Proposition $3$ in \cite{MaNe1}:
\begin{prop}\label{f,a commutation proposition}
    For $i,j\in\{0,...,r-1\}$, we have the following identities of linear maps
    \[
H^{*}(F^{n}\text{Quot}(V))\rightarrow H^{*}(F^{n}\text{Quot}(V)\times C\times C):
\]

    \[
    \left[f_{n}^{(j)},a_{n}^{(i)}\right]=\delta\cdot\begin{cases}
        (-1)^{i}\delta_{i+j,r-1}+\sum_{s=0}^{i-1}(-1)^{i-s}a_{n}^{(s)}f_{n}^{(i+j-s-1)},\,\,\,\,\,\,\,\text{ if }i+j\leq r-1\\\\
        (-1)^{i-s+1}\sum_{s=i}^{r-1}a_{n}^{(s)}f_{n}^{(i+j-s-1)},\,\,\,\,\,\,\,\,\,\,\,\,\,\,\,\,\,\,\,\,\,\,\,\,\,\,\,\,\,\,\,\,\,\,\,\,\,\,\,\,\text{ if }i+j\geq r
    \end{cases}.
    \]
    We make the convention that $a_{n}^{(i)}$ and $f_{n}^{(j)}$ respectively contribute to the first and second factor of $C\times C$ on both sides of the above equalities. Also, note that in the above equalities $\delta$ denotes the diagonal class in $F^{n}\text{Quot}(V)\times C\times C$, whereas $\delta_{i+j,r-1}$ denotes the Kronecker Delta. 
    \newline
\end{prop}
\noindent For $1\leq k\leq n$, by restricting the action of $a_k$ and $b_{k}$ to $\mathcal{H}_{k}$, remarks \ref{remark on am}, \ref{remark on bm} and proposition \ref{f,a commutation proposition} yield:

\begin{prop}\label{b,a commutation proposition}
    For $i,j\in\{0,...,r-1\}$ and $1\leq k\leq n$, we have the following identities of linear maps
    \[
\mathcal{H}_{k}\rightarrow H^{*}(F^{n}\text{Quot}(V)\times C\times C):
\]

    \[
    \left[b_{k_{|_{\mathcal{H}_{k}}}}^{(j)},a_{k_{|_{\mathcal{H}_{k}}}}^{(i)}\right]=\delta\cdot\begin{cases}
        (-1)^{i}\delta_{i+j,r-1}+\sum_{s=0}^{i-1}(-1)^{i-s}a_{k_{|_{\mathcal{H}_{k}}}}^{(s)}b_{k_{|_{\mathcal{H}_{k}}}}^{(i+j-s-1)},\,\,\,\,\,\,\,\text{ if }i+j\leq r-1\\\\
        (-1)^{i-s+1}\sum_{s=i}^{r-1}a_{k_{|_{\mathcal{H}_{k}}}}^{(s)}b_{k_{|_{\mathcal{H}_{k}}}}^{(i+j-s-1)},\,\,\,\,\,\,\,\,\,\,\,\,\,\,\,\,\,\,\,\,\,\,\,\,\,\,\,\,\,\,\,\,\,\,\,\,\,\,\,\,\text{ if }i+j\geq r
    \end{cases}.
    \]
   we make the convention that $a_{k_{|_{\mathcal{H}_{k}}}}^{(i)}$ and $b_{k_{|_{\mathcal{H}_{k}}}}^{(j)}$ respectively contribute to the first and second factor of $C\times C$ on both sides of the above equalities.  
\end{prop}
\noindent Before we proceed, let us make note of a useful fact that we will use several times without explicit reference. Let $\gamma_{1}\in H^{*}(C^{m})$ and $\gamma_{2}\in H^{*}(C^{l})$ be cohomology classes for some $l,m>0$
and suppose that we have two linear operators that are given by correspondences:
\begin{multline*}
T_{1}:H^{*}(F^n\text{Quot}(V))\rightarrow H^{*}(F^n\text{Quot}(V)\times C^{m})\\
\text{ and }T_{2}:H^{*}(F^n\text{Quot}(V))\rightarrow H^{*}(F^n\text{Quot}(V)\times C^{l}).
\end{multline*}
Then it is not hard to observe that
\[
T_{1}(\gamma_{1})T_{2}(\gamma_{2})=\left(T_{1} T_{2}\right)(\gamma_{1}\boxtimes \gamma_{2}).
\]
Please refer to \eqref{operators coloured with cohomology classes T(gamma)} for a reminder on notation.\\
\begin{lem}\label{lemma in proof of linear independence}
     Let us fix $u>0$ and $k\in\{0,...,n-1\}$. Let $\tilde{\gamma}\in H^{*}{(C^{2u})}$, $\nu\in \mathcal{H}_{k}$ and $p_{1},...,p_{u},q_{1},...,q_{u}\in\{0,...,r-1\}$. Suppose $\psi$ is an operator that is a product of $2u$ terms, consisting of all $a_{k+1}^{(p_{j})}$ and $b_{k+1}^{(q_{j})}$ for $j=1,...,u$, in any order:
     \begin{enumerate}
         \item  If $\sum\limits_{j=1}^{u}(p_{j}+q_{j})< u(r-1)$, then
    \[
    \psi(\tilde{\gamma})\nu= 0.
    \]
    \item If $\sum\limits_{j=1}^{u}(p_{j}+q_{j})= u(r-1)$ and if we write $\psi=\psi_{1}\left(b_{k+1}^{(q_{j})} a_{k+1}^{(p_{i})} \right)\psi_{2}$ for some operators $\psi_{1}$ and $\psi_{2}$ and some $i,j\in\{1,...,u\}$, then
    \[
    \left (\psi_{1}\left(b_{k+1}^{(q_{j})} a_{k+1}^{(p_{i})} \right)\psi_{2}\right)(\tilde{\gamma})\nu=\begin{cases}
        \left (\psi_{1}\left(a_{k+1}^{(p_{i})}  b_{k+1}^{(q_{j})}\right)\psi_{2}\right)(\tilde{\gamma})\nu,\\\text{ if }p_{i}+q_{j}\neq r-1.\\\\
       \left (\psi_{1}\left( a_{k+1}^{(p_{i})}  b_{k+1}^{(q_{j})}\right)\psi_{2}\right)(\tilde{\gamma})\nu+\left (\psi_{1}\cdot\delta\cdot\psi_{2}\right)(\tilde{\gamma})\nu,\\\text{ if }p_{i}+q_{j}= r-1.
    \end{cases}.
    \]
     \end{enumerate}
\end{lem}

\begin{proof}
    \begin{enumerate}
        \item Let us make a couple of observations. First, note that for any $m\geq 0$, we have that $b_{k+1}^{(m)} \nu=0$.
        Also note that Proposition \ref{b,a commutation proposition} implies that the commutator $$\left[b_{k+1_{|_{\mathcal{H}_{k+1}}}}^{(j)},a_{k+1_{|_{\mathcal{H}_{k+1}}}}^{(i)}\right]$$ is equal to $\delta$ times, a linear combination of $$a_{k+1_{|_{\mathcal{H}_{k+1}}}}^{(i')} b_{k+1_{|_{\mathcal{H}_{k+1}}}}^{(j')}$$ with $i'+j'<i+j$, plus or minus $\delta_{i+j,r-1}$. Then part (1) of the lemma follows from a double induction; By inducting on $u$ and on the number of $a_{k+1}^{(*)}$ terms between $\nu$ and the rightmost $b_{k+1}^{(*)}$ in the expression for $\psi$ as a product of the $a_{k+1}^{(*)}$'s and $b_{k+1}^{(*)}$'s.
        \item Part (2) of the lemma follows from part (1) and Proposition \ref{b,a commutation proposition}.
        \end{enumerate}
\end{proof}

\noindent We are now ready to prove the following proposition:

\begin{prop}\label{Linear Independence}
    The set $\mathfrak{A}$ (cf. equation \ref{linearly independent subset definition}) is a linearly independent subset of $H^{*}(F^{n}\text{Quot}(V))$.
\end{prop}

\begin{proof}
We proceed by induction on $k$ to show that 
\[
\mathfrak{A}\,\cap \mathcal{H}_{k}
\]
is a linearly independent subset of $H^{*}(F^{n}\text{Quot}(V))$ for all $k=0,1,...,n$. The case $k=0$ is clear. assuming that the statement holds for some $k\in\{0,...,n-1\}$, let us prove that 
\[
\mathfrak{A}\,\cap \mathcal{H}_{k+1}
\]
is a linearly independent subset of $H^{*}(F^{n}\text{Quot}(V))$.\\\\
To this end, let us consider a non-trivial linear relation of the elements in \\
$\mathfrak{A}\,\cap \mathcal{H}_{k+1}$ and arrive at a contradiction. Indeed, we can in particular consider a non-trivial linear relation of elements of $\mathfrak{A}\,\cap H^{*}(F^{n}\text{Quot}_{\vec{d}}\,(V))$ for some $\vec{d}$, such that $H^{*}(F^{n}\text{Quot}_{\vec{d}}\,(V))\subseteq \mathcal{H}_{k+1}$. Observe that every non-zero term of such a linear relation is of the following form for some fixed $l\geq 0$ depending on $\vec{d}$:
\[
c\cdot a_{k+1}^{(j_{1})}(\gamma_{1})a_{k+1}^{(j_{2})}(\gamma_{2})...a_{k+1}^{(j_{l})}(\gamma_{l})\beta,
\]
where $c\in\mathbb{Q},\,j_{1}\geq j_{2}...\geq j_{l}$ and $\beta\in \mathfrak{A}\,\cap \mathcal{H}_{k}$. Since the operators $a_{j}^{(v)}$ are given by correspondences, the above term is equal to:
\[
c\cdot a_{k+1}^{(j_{1})}a_{k+1}^{(j_{2})}...a_{k+1}^{(j_{l})}(\gamma_{1}\boxtimes...\boxtimes \gamma_{l})\beta,
\]
where we make the convention that $a_{k+1}^{(j_{i})}$ contributes to the $i^{th}$ copy of $C$ in $C^{l}$. Let us then re-write any non-trivial linear relation of elements in $\mathfrak{A}\,\cap H^{*}(F^{n}\text{Quot}_{\vec{d}}\,(V))$ as:
\begin{multline}\label{linear relation in proof of linear independence}
\sum\limits_{\beta_{\alpha}\in\,\mathfrak{A}\,\cap \mathcal{H}_{k} } c_{\alpha}\cdot a_{k+1}^{(i_{1})}a_{k+1}^{(i_{2})}...a_{k+1}^{(i_{l})}(\gamma_{\alpha})\beta_{\alpha}\\= \sum\limits_{\beta_{\alpha}\in\,\mathfrak{A}\,\cap \mathcal{H}_{k}}\left(\sum\limits_{\lambda=r-1\geq i_{1}'\geq ...\geq i_{l}'\geq 0} c_{\alpha}^{\lambda}\cdot a_{k+1}^{(i'_{1})}a_{k+1}^{(i'_{2})}...a_{k+1}^{(i'_{l})}(\gamma')\beta_{\alpha}\right),
\end{multline}
such that the left hand side of the above equation is a sum over only elements of $\mathfrak{A}\,\cap \mathcal{H}_{k}$ and the right hand side is a sum over $\mathfrak{A}\,\cap \mathcal{H}_{k}$ as well as $l-$tuples $r-1\geq i_{1}'\geq ...\geq i_{l}'\geq 0$; With either $i_1+...+i_{l}>i'_1+...+i'_{l}$ or $i_1+...+i_{l}=i'_1+...+i'_{l}$ but the partition $i_{1}\geq ...\geq i_{l}$ is greater than the partition $i_{1}'\geq ...\geq i_{l}'$, in lexicographic order.\\\\
Let $\phi\in H^{*}(C^{l})$, and consider the operator
\begin{equation}\label{lowering operator in proof of linear independence}
    b_{k+1}^{(r-i_{l}-1)}...b_{k+1}^{(r-i_{1}-1)}(\phi).
\end{equation}
Then using Lemma \ref{lemma in proof of linear independence}, one observes that
    the operator \ref{lowering operator in proof of linear independence} annihilates the R.H.S. of equation \ref{linear relation in proof of linear independence} and that the operator \ref{lowering operator in proof of linear independence} applied to the L.H.S. of equation \ref{linear relation in proof of linear independence} is a linear combination of the form:
\begin{equation}\label{expression in claim}
    \sum\limits_{\beta_{\alpha}\in\,\mathfrak{A}\,\cap \mathcal{H}_{k} } \pm \,c_{\alpha}\cdot\left(\int\gamma_{\alpha}\cdot\phi\right)\cdot\beta_{\alpha}.
\end{equation}

\noindent Then due to the non-degeneracy of the intersection pairing on $H^{*}(C^{l})$, with a suitable choice of $\phi$, we can ensure that the expression \ref{expression in claim} is non-zero. However, this yields a contradiction since the $\beta_{\alpha}\in \mathfrak{A}\,\cap \mathcal{H}_{k}$ are linearly independent by the induction hypothesis. 
\end{proof}

\subsection{Step II: Tautological generation of cohomology}
In this subsection, we will show that the cohomology of $F^{n}\text{Quot}_{\vec{d}}\,(V)$ is generated by tautological classes. Let us first explain what this means.\\\\
Recall that $F^{n}\text{Quot}_{\vec{d}}\,(V)\times C$ carries a universal flag (cf. \eqref{Universal flag}):
\[
\mathcal{E}_{n}\subseteq \mathcal{E}_{n-1}\subseteq...\mathcal{E}_1\subseteq \mathcal{E}_{0}=\pi^{*}(V)
\]
and has natural projection maps $\rho$ and $\pi$ to $F^{n}\text{Quot}_{\vec{d}}\,(V)$ and $C$ respectively. Tautological classes on $F^{n}\text{Quot}_{\vec{d}}\,(V)$ are classes of the form 
\begin{equation}\label{Tautological class definition}
\rho_{*}(c_{k}(\mathcal{E}_{i})\cdot \pi^{*}(\gamma)),
\end{equation}
with $i\in \{1,...,n\},\,k\in\{0,...,r\}, \gamma\in H^{*}(C)$. Note that in-terms of the Yangian action, the class \ref{Tautological class definition} can be written as:
\[
m_{i}^{(k)}(\gamma)\left[F^{n}\text{Quot}_{\vec{d}}\,(V)\right].
\]
 Tautological generation of the cohomology of $F^{n}\text{Quot}_{\vec{d}}\,(V)$ means that the classes of the form \ref{Tautological class definition} generate $H^{*}(F^{n}\text{Quot}_{\vec{d}}\,(V))$ as a ring. To show this, we will make use of the following general fact:
 \begin{lem}\label{lemma on kunneth components of diagonal}
     Let $X\rightarrow S$ be a smooth morphism of smooth projective varieties. Let $\Delta$ denote the diagonal of $X\times_{S}X$ and let $p_{1},p_{2}:X\times_{S}X\rightarrow X$ be the projection to the two factors. If we can write
     \[
     [\Delta]=\sum_{i=1}^{l}p_{1}^{*}(a_{i})\cdot p_{2}^{*}(b_i),
     \]
     for some $l>0$ and cohomology classes $a_{i},b_{i}\in H^{*}(X)$, then $H^{*}(X)$ is generated as a $H^{*}(S)$ module by the classes $a_{i}$.
 \end{lem}

 \noindent Now let us prove the tautological generation of the cohomology of $F^{n}\text{Quot}_{\vec{d}}\,(V)$, we will use definitions and notations from Subsection \ref{subsection on diagonal embeddings}.

\begin{prop}\label{Tautological generation}
    The cohomology of $F^{n}\text{Quot}_{\vec{d}}\,(V)$ is generated as a ring by classes of the form \ref{Tautological class definition}.
\end{prop}

\begin{proof}
To show that the cohomology of $F^{n}\text{Quot}_{\vec{d}}\,(V)$ is tautologically generated for all $n\geq 0$ and all $n-$tuples $\vec{d}$, let us induct on $n$. The case $n=0$ is trivial so let us assume that tautological generation holds for $H^{*}(F^{n-1}\text{Quot}_{\vec{f}}\,(V))$ for all $(n-1)-$ tuples of non-negative integers $\vec{f}$.\\\\
Let $\vec{d}_{\downarrow}$ be the $(n-1)-$ tuple $(d_{1},...,d_{n-1}).$ Recall the map \ref{forgetful map FnQuot to Fn-1}: $$\varepsilon:F^{n}\text{Quot}_{{\vec{d}}}\,(V)\rightarrow F^{n-1}\text{Quot}_{{\vec{d}_{\downarrow}}}(V),$$
and the moduli space \ref{Moduli space : mathfrak P}, that is the fibre product
\[
\mathfrak{P}_{\vec{d}}=F^{n}\text{Quot}_{{\vec{d}}}\,(V)\times_{F^{n-1}\text{Quot}_{{\vec{d}_{\downarrow}}}(V)} F^{n}\text{Quot}_{{\vec{d}}}\,(V)
\]
of $\varepsilon$ with itself. Let $p_{1},p_{2}:\mathfrak{P}_{\vec{d}}\rightarrow F^{n}\text{Quot}_{{\vec{d}}}\,(V)$ be the projection maps to the two factors. In order to prove that the cohomology of $F^{n}\text{Quot}_{{\vec{d}}}\,(V)$ is tautologically generated, it follows from the induction hypothesis and Lemma \ref{lemma on kunneth components of diagonal} that it is sufficient to show that 
\[
[\Delta]=\sum_{i=1}^{l}p_{1}^{*}(a_{i})\cdot p_{2}^{*}(b_i)
\]
and that the $a_{i}$'s or the $b_{i}$'s are of the form \ref{Tautological class definition}. Let $\rho:\mathfrak{P}_{\vec{d}}\times C\rightarrow \mathfrak{P}_{\vec{d}}$ be the natural projection and recall the vector bundle
\[
\rho_{*}\mathcal{H}om(\mathcal{E}_{n},\mathscr{F}^{'}_{n})
\]
on $\mathfrak{P}_{\vec{d}}$ (cf. Proposition \ref{proposition to be used in tautological generation}).
Also, Proposition \ref{proposition to be used in tautological generation} implies that
\[
[\Delta] = c_{r(d_{n}-d_{n-1})}\left(\rho_{*}\mathcal{H}om\left(\mathcal{E}_{n},\mathscr{F}^{'}_{n}\right)\right)
\]
and since $R^{i}\rho_{*}\mathcal{H}om(\mathcal{E}_{n},\mathscr{F}^{'}_{n})=0$ for $i>0$ (cf. Proposition \ref{proposition to be used in tautological generation}), we can conclude by a standard Grothendieck-Riemann-Roch argument applied to the coherent sheaf $\mathcal{H}om(\mathcal{E}_{n},\mathscr{F}^{'}_{n})$ and the morphism $\rho$. 
\end{proof}

\subsection{Step III: Span}\label{subsection on span}
In this subsection we will show that the set of elements 
\[
\mathfrak{S}:=\left\{a_{k_l}^{(v_l)}(\gamma_{i_l})...a_{k_2}^{(v_2)}(\gamma_{i_2})a_{k_1}^{(v_1)}(\gamma_{i_1})|0\rangle\mid l\geq 0 \text{ and } n\geq k_{l}\geq k_{l-1}\geq...\geq k_{1}\geq 1 \right \},
\]
spans $H^{*}(F^{n}\text{Quot}(V)).$ It is clear that the set $\mathfrak{S}$ contains the linearly independent set $\mathfrak{A}$ (cf. the set \ref{linearly independent subset definition}). Note that Theorem \ref{a's commute intro}, which we would prove in the next subsection, would then imply that $\mathfrak{A}=\mathfrak{S}$ and that this set gives a basis for $H^{*}(F^{n}\text{Quot}(V))$.\\\\
Let us make note of a useful proposition that follows directly from Proposition \ref{multiplication by chern classes formula} and Remarks \ref{remark on am} and \ref{remark on bm}:
\begin{prop}\label{multiplication by chern classes prop am bm version}
    Let $k>0$ and $j\in \{0,...,n\}$, then the operator\\
    $\mathcal{H}_{j}\rightarrow H^{*}(F^{n}\text{Quot}_{\vec{d}}\,(V)\times C)$, given by pull-back and multiplication by the class 
\[
c_{k}\left((\mathcal{E}_{j-1}-\mathcal{E}_{j})\otimes \omega_{C}^{-1}\right)
\]
is equal to the operator:
\[
\sum_{i=0}^{r-1}\left[a_{j_{|_{\mathcal{H}_{j}}}}^{(i)}b_{j_{|_{\mathcal{H}_{j}}}}^{(r+k-i-2)}\right]_{|_{\Delta}}(-1)^{i-k}.
\]
\end{prop}
\noindent Finally, we now prove the following proposition, which immediately implies that the set $\mathfrak{S}$ spans $H^{*}(F^{n}\text{Quot}(V))$:

\begin{prop}\label{Span}
    Given $j=0,...,n-1$, the action of the operators $a_{j+1}^{(v)}(\gamma_{i})$, for $v=0,...,r-1$ and $i=1,...,2g+2$, on $\mathcal{H}_{j}$ generates $\mathcal{H}_{j+1}$.
\end{prop} 

\begin{proof} 
    For all non-decreasing sequences on non-negative integers $\vec{d}=(d_{1},...,d_{n})$, with $d_{j+1}=d_{j+2}=...=d_{n}$, we must show that $H^{*}(F^{n}\text{Quot}_{\vec{d}}\,(V))$ is spanned by classes of the form 
    \begin{equation}\label{what we want to prove in the proof of span}
            a_{j+1}^{(v_{{l}})}(\gamma_{s_{l}}) a_{j+1}^{(v_{{l-1}})}(\gamma_{s_{l-1}}) ... a_{j+1}^{(v_{{1}})}(\gamma_{s_{1}}) w_{j}.
    \end{equation}
\noindent With $l\geq 0$ and $w_{j}\in \mathcal{H}_{j}$. \\\\
To prove this, let us induct on $d_{j+1}-d_{j}$. The case $d_{j+1}-d_{j}=0$ is clear since in this case $H^{*}(F^{n}\text{Quot}_{\vec{d}}\,(V))\subseteq \mathcal{H}_{j}$.\\\\
Now, assuming that for some $m\geq 0$, that the cohomology of all $F^{n}\text{Quot}_{\vec{d}}\,(V)$, with $\vec{d}$ such that $d_{j+1}=d_{j+2}=...=d_{n}$, is spanned by classes of the form \ref{what we want to prove in the proof of span}, whenever $d_{j+1}-d_{j}=m$; Let us prove that $H^{*}(F^n\text{Quot}_{\vec{d}_{j+1,n}}(V))$ is generated by classes of the form \eqref{what we want to prove in the proof of span}, where
\[
\vec{d}_{j+1,n}=(d_{1},..,d_{j},d_{j+1}+1,...,d_{n}+1).
\]
\newline
First, let us show that the fundamental class $$[F^n\text{Quot}_{\vec{d}_{j+1,n}}(V)]$$ can be written as a linear combination of classes of the form \eqref{what we want to prove in the proof of span}. Observe that $\mathcal{Z}^{j+1,n}_{\vec{d}}(V)$ is a Hyperquot scheme, since $d_{j+1}=...=d_{n}$, and is therefore, generically a $(d_{j+1}-d_{j}+1):1$ cover of $F^n\text{Quot}_{\vec{d}_{j+1,n}}(V)$. This implies that 

\[
[F^n\text{Quot}_{\vec{d}_{j+1,n}}(V)]=\frac{1}{(d_{j+1}-d_{j}+1)}\cdot\left( a_{j+1}^{(0)}[F^n\text{Quot}_{\vec{d}}\,(V)]\right).
\]
\newline
Then, by the induction hypothesis, since $[F^n\text{Quot}_{\vec{d}}\,(V)]$ can be written as a linear combination of classes of the form \eqref{what we want to prove in the proof of span}, so can $[F^n\text{Quot}_{\vec{d}_{j+1,n}}(V)].$\\\\
We showed in Proposition \ref{Tautological generation} that any class in $H^{*}(F^n\text{Quot}_{\vec{d}_{j+1,n}}(V))$ is a linear combination of classes of the form
\[
m_{i_{t}}^{(k_{t})}(\gamma_{y_{t}}) m^{(k_{t-1})}_{i_{t-1}}(\gamma_{y_{t-1}})...m_{i_{1}}^{(k_{1})}(\gamma_{y_{1}})[F^{n}\text{Quot}_{\vec{d}_{j+1,n}}(V)].
\]
for various arbitrary $t$'s. We can then re-write the above expression as a linear combination of expressions of the form
\[
m_{i_{t}}^{(k_{t})}(\gamma_{y_{t}})m^{(k_{t-1})}_{i_{t-1}}(\gamma_{y_{t-1}})...m_{i_{1}}^{(k_{1})}(\gamma_{y_{1}}) a_{j+1}^{(v_{{l}})}(\gamma_{s_{l}}) a_{j+1}^{(v_{{l-1}})}(\gamma_{s_{l-1}}) ... a_{j+1}^{(v_{{1}})}(\gamma_{s_{1}}) w_{j}
\]
with $l=d_{j+1}-d_{j}$ and $w_{j}\in \mathcal{H}_{j}$. In the above expression, whenever some $i<j+1$, it follows from the equality \eqref{an=me}, Remark \ref{remark on am} and the relations in Theorem \ref{Yangian relations satisfied}, that we can commute the $m_{i}^{(*)}(*)$ past the other $m$'s and the $a_{j+1}^{(*)}(*)$'s. Therefore, any class in $H^{*}(F^{n}\text{Quot}_{\vec{d}_{j+1,n}}(V))$ is a linear combination of classes of the form 
\[
m_{j+1}^{(k_{p})}(\gamma_{y_{p}}) m^{(k_{p-1})}_{j+1}(\gamma_{y_{p-1}})... m_{j+1}^{(k_{1})}(\gamma_{y_{1}}) a_{j+1}^{(v_{{l}})}(\gamma_{s_{l}}) a_{j+1}^{(v_{{l-1}})}(\gamma_{s_{l-1}}) a_{j+1}^{(v_{{1}})}(\gamma_{s_{1}}) w_{j}'
\]
for arbitrary $p$ and with $w_{j}'\in \mathcal{H}_{j}$.\\\\
 Then, using the multiplicativity of the total Chern class in exact sequences and Proposition \ref{multiplication by chern classes prop am bm version}, we may write each $m_{j+1}^{(*)}(*)$ as a polynomial in the operators $m_{j}^{(*)}(*)$ and $a_{j+1}^{(*)}(*)b^{(*)}_{j+1}(*)$. As before, we may commute the various $m_{j}^{(*)}(*)$ past the other $m_{j+1}^{(*)}(*)$'s, $a_{j+1}^{(*)}(*)$ and $b_{j+1}^{(*)}$'s. Hence, we obtain that any class in $H^{*}(F^{n}\text{Quot}_{\vec{d}_{j+1,n}}(V))$ is a linear combination of classes of the form:
 \begin{equation}\label{penultimate expression in the proof of span}
   a_{j+1}^{(y_{{q}})}(\gamma_{z_{q}}) b_{j+1}^{(y'_{q})}(\gamma_{x_{q}}) ... a_{j+1}^{(y_{{1}})}(\gamma_{z_{1}}) b_{j+1}^{(y'_{1})}(\gamma_{x_{1}}) a_{j+1}^{(v_{{l}})}(\gamma_{s_{l}})  a_{j+1}^{(v_{{l-1}})}(\gamma_{s_{l-1}}) ... a_{j+1}^{(v_{{1}})}(\gamma_{s_{1}})w_{j}''.  
 \end{equation}
\noindent for various arbitrary $q$'s and with $w_{j}''\in \mathcal{H}_{j}$.
 However,
 \[
 b_{j+1}^{(y'_{q})}(\gamma_{x_{q}}) a_{j+1}^{(y_{{q-1}})}(\gamma_{z_{q-1}})...a_{j+1}^{(y_{{1}})}(\gamma_{z_{1}})b_{j+1}^{(y'_{1})}(\gamma_{x_{1}})  a_{j+1}^{(v_{{l}})}(\gamma_{s_{l}}) a_{j+1}^{(v_{{l-1}})}(\gamma_{s_{l-1}}) ... a_{j+1}^{(v_{{1}})}(\gamma_{s_{1}})w_{j}''.
 \]
 is a class in $H^{*}(F^{n}\text{Quot}_{\vec{d}}\,(V))$, which by the induction hypothesis can be written as a linear combination of classes of the form \eqref{what we want to prove in the proof of span} and hence, so can the expression \eqref{penultimate expression in the proof of span}, which finishes the proof.
\end{proof}

\subsection{Step IV: Commutativity of the operators \text{$a_{k}^{(v)}$}}\label{subsection on the commutativity o operator a's}
In this subsection, we will prove Theorem \ref{a's commute intro}, that the operators $a_{k}^{(v)}$ commute, for all $k=1,...,n$ and $v=0,...,r-1$. Together with Propositions \ref{Span} and \ref{Linear Independence}, this yields Theorem \ref{basis of cohomology intro} as well.\\\\
We will crucially use the geometry of the moduli spaces \ref{The moduli space: U}, \ref{The moduli space: W} and \ref{The moduli space: X} and therefore will use notations and conventions from Subsections \ref{subsection: moduli spaces U,X,W}-\ref{Subsection: Divisors}. Let us recall some of the notation;  Let $\vec{d}=(d_1,...,d_{n})$ be an $n-$tuple of non-negative integers. For $1\leq i\leq n$, recall $n-$tuples $\vec{d'}_{i,n}:=(d_1,...,d_{i-1},d_{i}+1,...,d_{n}+1)$ and $\vec{d''}_{i,n}:=(d_1,...,d_{i-1},d_{i}+1,...,d_{n-1}+1,d_{n}+2)$.\\\\
To prove the desired commutativity, for $k\in\{1,...,n\}$ it would be convenient to sometimes arrange the operators $a_{k}^{(v)}$ into the polynomial \ref{a current}:
\[
    a_{k}(w):=\sum\limits_{v=0}^{r-1}(-1)^{v}a_{k}^{(v)}w^{r-1-v}.
\]
The following lemmas are consequences of Propositions \ref{class of correspondence}, \ref{composition of correspondences} and the dimension computations in Propositions \ref{Geometry of moduli space $X$} and \ref{Geometry of moduli space $W$}. 

\begin{lem}\label{lemma on an.ak composition}
    The linear operator
\[
    a_{n}(z) a_{k}(w):H^{*}(F^{n}\text{Quot}_{\vec{d}}\,(V))\rightarrow H^{*}(F^{n}\text{Quot}_{\vec{d''}_{k,n}}(V)\times C\times C)[w,z]
\]
is given as a correspondence by the class 
\[
\left(k_{X}\times q^{X}\times p^{X}\right)_{*}\left(c(\mathcal{G}_{3},w)\cdot c(\mathcal{G}_{2},z) \right)
\] 
on $F^{n}\text{Quot}_{\vec{d}}(V)\times F^{n}\text{Quot}_{\vec{d''}_{k,n}}(V)\times C\times C$. Note that in the case $k=n$, we have $\mathcal{G}_{3}=\mathcal{G}_{1}$.
\end{lem}

\begin{lem}\label{lemma on ak.an composition}
     The linear operator
\[
    a_{k}(w)a_{n}(z):H^{*}(F^{n}\text{Quot}_{\vec{d}}\,(V))\rightarrow H^{*}(F^{n}\text{Quot}_{\vec{d''}_{k,n}}(V)\times C\times C)[w,z]
\]
is given as a correspondence by the class 
\[
\left(k_{W}\times q^{W}\times p^{W}\right)_{*}\left(c(\mathcal{G}_{3},w)\cdot c(\tilde{\mathcal{G}}_{1},z) \right)
\]
on $F^{n}\text{Quot}_{\vec{d}}(V)\times F^{n}\text{Quot}_{\vec{d''}_{k,n}}(V)\times C\times C$ if $k\neq n$. If $k=n$, then the operator
\[
a_{n}(w) a_{n}(z):H^{*}(F^{n}\text{Quot}_{\vec{d}}\,(V))\rightarrow H^{*}(F^{n}\text{Quot}_{\vec{d''}_{k,n}}(V)\times C\times C)[w,z]
\]
is given as a correspondence by the class
\[
\left(k_{W}\times q^{W}\times p^{W}\right)_{*}\left(c(\tilde{\mathcal{G}}_{2},w)\cdot c(\tilde{\mathcal{G}}_{1},z) \right)
\]
on $F^{n}\text{Quot}_{\vec{d}}(V)\times F^{n}\text{Quot}_{\vec{d''}_{n,n}}(V)\times C\times C$.
\end{lem}
\noindent In the above lemmas we make the convention that the operators in the polynomials $a_{k}(w)$ and $a_{n}(z)$ respectively contribute to the first and second copies of $C$, in the product $C\times C$.\\\\
\noindent Now, observe that:
\begin{multline*}
\left(k_{X}\times p^{X}\times q^{X}\right)\circ \,\eta^{k,n}_{\vec{d}}=\left(k_{W}\times p^{W}\times q^{W}\right)\circ\,\theta^{k,n}_{\vec{d}}\colon\\\,\mathcal{U}^{k,n}_{\vec{d}}
\rightarrow F^{n}\text{Quot}_{\vec{d}}(V)\times F^{n}\text{Quot}_{\vec{d''}_{k,n}}(V)\times C\times C.
\end{multline*}
Let us denote the above map by $\beta$.

\begin{prop}\label{Proposition: an,ak commute}
For all $s,t\in \{0,...,r-1\}$ and $k\in\{1,...,n\}$, we have that 
\[
a_{n}^{(t)} a_{k}^{(s)}= a_{k}^{(s)} a_{n}^{(t)}
\]
as operators $H^{*}(F^{n}\text{Quot}_{\vec{d}}\,(V))\rightarrow H^{*}(F^{n}\text{Quot}_{\vec{d''}_{k,n}}(V)\times C\times C)$, with the convention that the operator $a_{k}^{(s)}$ contributes to the first factor of $C\times C$ and operator $a_{n}^{(t)}$ contributes to the second factor of $C\times C$.
\end{prop}

\begin{proof}
    Let us treat the cases $k\neq n$ and $k=n$ separately.\\\\
    \textbf{Case I: k$\neq $n:}
    From Lemmas \ref{lemma on an.ak composition} and \ref{lemma on ak.an composition}, it suffices to show that the following expression in $H^*(F^{n}\text{Quot}_{\vec{d}}(V)\times F^{n}\text{Quot}_{\vec{d''}_{k,n}}(V)\times C\times C)$ is equal to $0$:
    \begin{equation}\label{equation: expression which should be 0 an ak case.}
\left(k_{X}\times q^{X}\times p^{X}\right)_{*}\left(c(\mathcal{G}_{3},w)\cdot c(\mathcal{G}_{2},z) \right)-\left(k_{W}\times q^{W}\times p^{W}\right)_{*}\left(c(\mathcal{G}_{3},w)\cdot c(\tilde{\mathcal{G}}_{1},z) \right) .
    \end{equation}
\noindent Then, the projection formula, along with the fact that $\eta_{\vec{d}}^{k,n}$ and $\theta_{\vec{d}}^{k,n}$ are birational morphisms, allows us to re-write expression \ref{equation: expression which should be 0 an ak case.} as 
\[
\beta_{*}\left(c(\mathcal{G}_{3},w)\cdot c(\mathcal{G}_{2},z)-c(\mathcal{G}_{3},w)\cdot c(\tilde{\mathcal{G}}_{1},z) \right)
\]
\begin{multline*}
=\beta_{*}\left(c(\mathcal{G}_{3},w)\cdot\left[ c(\mathcal{G}_{2},z)- c(\tilde{\mathcal{G}}_{1},z)\right] \right)=\\
(k_{X}\times p^{X}\times q^{X})_{*}\circ \left(\,{\eta^{k,n}_{\vec{d}}}\right)_{*}\left(c(\mathcal{G}_{3},w)\cdot\left[ c(\mathcal{G}_{2},z)- c(\tilde{\mathcal{G}}_{1},z)\right] \right).
\end{multline*}
By the projection formula, this is equal to:
\[
(k_{X}\times p^{X}\times q^{X})_{*}\left[c(\mathcal{G}_{3},w)\cdot {\eta^{k,n}_{\vec{d}}}_{*} \left( c(\mathcal{G}_{2},z)- c(\tilde{\mathcal{G}}_{1},z) \right) \right].
\]
Then, using Corollary \ref{corollary (use: proof of commutativity of a k neq n): difference of c(G,z)}, we can re-write the above expression as:
\[
(k_{X}\times p^{X}\times q^{X})_{*}\left[c(\mathcal{G}_{3},w)\cdot {\eta^{k,n}_{\vec{d}}}_{*}  \left([\text{Exp}]\cdot\frac{c\left({\mathcal{E}_{n}}_{_{|_{\Gamma_{p}}}},z\right)}{c(\mathscr{L}_{1},z)\cdot c(\mathscr{L}_{2},z)}\right) \right].
\]
\newline
\[
=(k_{X}\times p^{X}\times q^{X})_{*}\left[c(\mathcal{G}_{3},w) \cdot\frac{c\left(\mathcal{E}_{{n}_{_{|_{\Gamma_{p}}}}},z\right)}{c(\mathscr{L}_{1},z)\cdot c(\mathscr{L}_{2},z)} \cdot {\eta^{k,n}_{\vec{d}}}_{*} \left[\text{Exp}\right] \right].
\]
\newline
Again, the last line follows from the projection formula. Since Exp is the exceptional divisor of $\eta^{k,n}_{\vec{d}}$, we have that ${\eta^{k,n}_{\vec{d}}}_{*} \text{Exp}=0$ and hence the above expression is equal to $0$. This finishes the proof in the case $k\neq n$.\\\\
\textbf{Case II: k$= $n:}
Using Lemmas \ref{lemma on an.ak composition} and \ref{lemma on ak.an composition}, it is enough to show that the following expression is equal to 0:
\[
\left(k_{X}\times q^{X}\times p^{X}\right)_{*}\left(c(\mathcal{G}_{1},w)\cdot c(\mathcal{G}_{2},z) \right)-\left(k_{W}\times q^{W}\times p^{W}\right)_{*}\left(c(\tilde{\mathcal{G}}_{2},w)\cdot c(\tilde{\mathcal{G}}_{1},z) \right). \\\\
\]
By the projection formula, we can re-write this expression as 
\[
\beta_{*}\left(c(\mathcal{G}_{1},w)c(\mathcal{G}_{2},z)-c(\tilde{\mathcal{G}}_{1},z)c(\tilde{\mathcal{G}}_{2},w)\right).
\]
Then, Corollary \ref{corollary (use: proof of commutativity of a k=n): differences of c(G,z)c(G,w)} implies that the expression is equal to 
\[
\beta_{*}\left([\text{Exp}]\cdot\frac{c\left({\mathcal{E}_{n}}_{_{|_{\Gamma_{p}}}},z\right)\cdot c\left({\mathcal{E}_{n}}_{_{|_{\Gamma_{q}}}},w\right)}{c(\mathscr{L}_{1},w)\cdot c(\mathscr{L}_{1},z)\cdot c(\mathscr{L}_{2},z)}- [\text{Exp}]\cdot\frac{c\left({\mathcal{E}_{n}}_{_{|_{\Gamma_{p}}}},z\right)\cdot c\left({\mathcal{E}_{n}}_{_{|_{\Gamma_{q}}}},w\right)}{c(\tilde{\mathscr{L}}_{1},z)\cdot c(\tilde{\mathscr{L}}_{1},w)\cdot c(\tilde{\mathscr{L}}_{2},w)}\right)
\]
\[
=(k_{X}\times p^{X}\times q^{X})_{*}\circ \left(\,{\eta^{n,n}_{\vec{d}}}\right)_{*}\left([\text{Exp}]\cdot\frac{c\left({\mathcal{E}_{n}}_{_{|_{\Gamma_{p}}}},z\right)\cdot c\left({\mathcal{E}_{n}}_{_{|_{\Gamma_{q}}}},w\right)}{c(\mathscr{L}_{1},w)\cdot c(\mathscr{L}_{1},z)\cdot c(\mathscr{L}_{2},z)}\right)-\]
\[(k_{W}\times p^{W}\times q^{W})_{*}\circ \left(\,{\theta^{n,n}_{\vec{d}}}\right)_{*}\left([\text{Exp}]\cdot\frac{c\left({\mathcal{E}_{n}}_{_{|_{\Gamma_{p}}}},z\right)\cdot c\left({\mathcal{E}_{n}}_{_{|_{\Gamma_{q}}}},w\right)}{c(\tilde{\mathscr{L}}_{1},z)\cdot c(\tilde{\mathscr{L}}_{1},w)\cdot c(\tilde{\mathscr{L}}_{2},w)}\right).
\]
By the projection formula, this is equal to 
\[
(k_{X}\times p^{X}\times q^{X})_{*} \left(\left(\,{\eta^{n,n}_{\vec{d}}}_{*}[\text{Exp}]\right)\cdot\frac{c\left({\mathcal{E}_{n}}_{_{|_{\Gamma_{p}}}},z\right)\cdot c\left({\mathcal{E}_{n}}_{_{|_{\Gamma_{q}}}},w\right)}{c(\mathscr{L}_{1},w)\cdot c(\mathscr{L}_{1},z)\cdot c(\mathscr{L}_{2},z)}\right)-\]

\[(k_{W}\times p^{W}\times q^{W})_{*}\left(\left(\,{\theta^{n,n}_{\vec{d}}}_{*}[\text{Exp}]\right)\cdot\frac{c\left({\mathcal{E}_{n}}_{_{|_{\Gamma_{p}}}},z\right)\cdot c\left({\mathcal{E}_{n}}_{_{|_{\Gamma_{q}}}},w\right)}{c(\tilde{\mathscr{L}}_{1},z)\cdot c(\tilde{\mathscr{L}}_{1},w)\cdot c(\tilde{\mathscr{L}}_{2},w)}\right).
\]
\newline
Since Exp is contracted by both ${\eta^{n,n}_{\vec{d}}}$ and ${\theta^{n,n}_{\vec{d}}}$, the above expression is equal to $0$ and we are done.
\end{proof}

\noindent The following proposition is clearly a consequence of Proposition \ref{Proposition: an,ak commute} and Remark $\ref{remark on am}$:
\begin{prop}\label{proposition: al ak commute restricted to Hl}
       For all $s,t\in \{0,...,r-1\}$ and $1\leq k\leq l \leq n\}$, we have that 
\[
{a_{l}}_{_{|_{\mathcal{H}_l}}}^{(t)} {a_{k}}_{_{|_{\mathcal{H}_l}}}^{(s)}= {a_{k}}_{_{|_{\mathcal{H}_l}}}^{(s)} {a_{l}}_{_{|_{\mathcal{H}_l}}}^{(t)}
\]
as operators $\mathcal{H}_{l}\rightarrow H^{*}(F^{n}\text{Quot}(V)\times C\times C)$.
\end{prop}

\noindent Finally, we can use Proposition \ref{Span}, along with the above proposition to prove Theorem \ref{a's commute intro}, that the above commutation holds not only when restricted to $\mathcal{H}_{l}$ but on the whole of $H^{*}(F^{n}\text{Quot}(V))$:

\begin{prop}\label{proposition: a's commute}
 For all $s,t\in \{0,...,r-1\}$ and $1\leq k\leq l \leq n\}$, we have
\[
a_{l}^{(t)} a_{k}^{(s)}= a_{k}^{(s)} a_{l}^{(t)}
\]
as operators $H^{*}(F^{n}\text{Quot}(V))\rightarrow H^{*}(F^{n}\text{Quot}(V)\times C\times C)$, with the convention that the operator $a_{k}^{(s)}$ contributes to the first factor of $C\times C$ and operator $a_{l}^{(t)}$ contributes to the second factor of $C\times C$.
\end{prop}

\begin{proof}
    Let us prove the proposition by descending induction on $l$. The base case $l=n$ is exactly Proposition \ref{Proposition: an,ak commute}. Let us now prove the proposition for arbitrary $l\in {1,...,n-1}$, assuming that the proposition holds for $l+1$.\\\\
    To show that $a_{k}^{(s)}$ and $a_{l}^{(t)}$ commute on $H^{*}(F^{n}\text{Quot}(V))$, indeed by Proposition \ref{Span} it suffices to show that 
    \[
    a_{l}^{(t)} a_{k}^{(s)} w= a_{k}^{(s)} a_{l}^{(t)} w
    \]
    for all $w$ in the spanning set $\mathfrak{S}$, in step $3$. Observe that any $w\in \mathfrak{S}$ can be written as 
    \[
    w=a_{k_{m}}^{(v_m)}(\gamma_{i_l})...a_{k_2}^{(v_2)}(\gamma_{i_2})a_{k_1}^{(v_1)}(\gamma_{i_1}) w_{l}
    \]
    for some $m\geq 0$, $w_{l}\in \mathcal{H}_{l}$ and $n\geq k_{m}\geq k_{m-1}\geq...\geq k_{1}\geq l+1$. By the induction hypothesis we have that 

    \[
     a_{l}^{(t)}a_{k}^{(s)}w=a_{k_{m}}^{(v_m)}(\gamma_{i_l})...a_{k_2}^{(v_2)}(\gamma_{i_2})a_{k_1}^{(v_1)}(\gamma_{i_1})\left(a_{l}^{(t)}a_{k}^{(s)}\right)   w_{l}.
    \]
\newline
Proposition \ref{proposition: al ak commute restricted to Hl} then implies that:

\[
a_{l}^{(t)}a_{k}^{(s)} w=a_{k_{m}}^{(v_m)}(\gamma_{i_l})...a_{k_2}^{(v_2)}(\gamma_{i_2})a_{k_1}^{(v_1)}(\gamma_{i_1})\left( a_{k}^{(s)} a_{l}^{(t)}\right)  w_{l}.
\]
\newline
Finally, by the induction hypothesis again, we obtain:

\[
a_{l}^{(t)} a_{k}^{(s)} w=\left( a_{k}^{(s)} a_{l}^{(t)}\right)  a_{k_{m}}^{(v_m)}(\gamma_{i_l})...a_{k_2}^{(v_2)}(\gamma_{i_2})a_{k_1}^{(v_1)}(\gamma_{i_1})  w_{l}= a_{k}^{(s)} a_{l}^{(t)} w.
\]
\newline
which finishes the proof.
\end{proof}

\newpage
\bibliographystyle{plain} 
\bibliography{refs}

\end{document}